\documentclass[twoside,final]{amsart}

\usepackage[utf8]{inputenc}
\usepackage[T1]{fontenc}
\usepackage[english]{babel}

\usepackage{mathpazo}
\usepackage{mathptmx}
\usepackage{amsmath}
\usepackage{amssymb}
\usepackage{mathrsfs}
\usepackage{amsthm}
\usepackage{amsfonts}
\usepackage{fancyhdr}
\usepackage[all]{xy}

\usepackage[left=3cm, right=3cm]{geometry}

\date{}

\usepackage{amsfonts, amssymb, amsmath, amsthm, multicol,bbm}

\usepackage[colorlinks=true, pdfstartview=FitV, linkcolor=blue,
            citecolor=blue, urlcolor=blue]{hyperref}
\usepackage[numbers,sort&compress]{natbib}
\usepackage{tikz}
\usetikzlibrary{matrix}

\usepackage[notref,notcite,color]{showkeys}
\definecolor{labelkey}{rgb}{0,0,1}

\numberwithin{equation}{section}

\newtheorem{Thm}{Theorem}[section]
\newtheorem{Lem}{Lemma}[section]

\newtheorem{Cor}{Corollary}[section]
\newtheorem{Ex}{Example}[section]

\newtheorem{Rmk}{Remark}[section]

\newtheorem*{Thm*}{Theorem}

\usepackage{cjhebrew}

\newcommand{\E}{\mathbb{E}}
\newcommand{\Prb}{\mathbb{P}}
\newcommand{\RR}{\mathbb{R}}

\newcommand{\al}{\alpha}

\newcommand{\T}{T}

\newcommand{\Trs}{\theta}

\newcommand{\eps}{\varepsilon}

\newcommand{\lqq}[1]{\lleq_{\hspace{-4pt} #1 \hspace{2pt}}}

\newcommand{\sst}{{s^*}}
\newcommand{\sstt}{{s^*_1}}

{\theoremstyle {definition} \newtheorem {defi} {Definition} [section] }
{\theoremstyle {plain}  \newtheorem {thm} [defi] {Theorem}}
{\theoremstyle {plain}  \newtheorem {cor} [defi]{Corollary}}
{\theoremstyle {plain} \newtheorem {prop} [defi]{Proposition}}
{\theoremstyle {plain} \newtheorem {nem}[defi] {Lemma}}
{\theoremstyle {plain} }
{\theoremstyle {plain} \newtheorem {rmq}[defi] {Remark}}

{\theoremstyle {plain}  }
{\theoremstyle {plain}  }
{\theoremstyle {plain}  }
{\theoremstyle {plain} }
{\theoremstyle {plain} }
{\theoremstyle {plain} }
{\theoremstyle {plain} }

\newcommand{\uua}{u}

\def\E{{\mathbb{E}}}
\def\T{{\mathbb{T}}}
\def\P{{\mathbb{P}}}
\def\R{{\mathbb{R}}}

\def\Z{{\mathbb{Z}^3}}
\def\N{{\mathbb{N}}}

\def\i{{\textbf{i}}}

\def\dt{{\partial_t}}

\def\ggeq{{\ \gtrsim\ \ }}
\def\lleq{{\ \lesssim\ \ }}

\begin{document}

\author[F\"oldes and Sy]{Juraj F\"oldes and Mouhamadou Sy}

\address{Juraj F\"oldes
\footnote{Juraj F\"oldes was partly supported under the grant NSF DMS-1816408}
\newline 
\indent Department of Mathematics, University of Virginia \indent 
\newline \indent   Kerchof Hall, Charlottesville, VA 22904-4137,\indent }
\email{jf8dc@virginia.edu}

\address{Mouhamadou Sy\footnote{The research of M. Sy is funded by the Alexander von Humboldt foundation under the “German Research Chair programme” financed by the Federal Ministry of Education and Research (BMBF).}
\newline 
\indent AIMS Senegal \indent 
\newline \indent Km 2, Route de Joal, Mbour, Senegal\indent }
\email{mouhamadou.sy@aims-senegal.org}

\title[GWP for 3D Euler and other fluid dynamics models ]{Almost sure global well-posedness for 3D Euler equation and other fluid dynamics models}

\begin{abstract}
We construct various statistical ensembles associated to the 3D Euler equations and prove global regularity of these equations for data living on these sets. Similar results are also proven for generalized SQG equations and some shell models. Qualitative properties of the ensembles and the constructed flows are also given.
\\


\end{abstract}

\maketitle

\nocite{*}

\section{Introduction}
In this manuscript we study equations, on a torus
$\T^d := \R^d/2\pi \mathbb{Z}^d$, $d = 2, 3$, 
of the form: 
\begin{equation}
\dt u + B(u, u) = 0  \label{Euler1}
\end{equation}
which include Euler equation, SQG equation, active scalar equations such as infinite Prandt Boussinesq system, inviscid porous medium equation, or MG equation, as well as various 
shell models (see Examples \ref{ex:euler} -- \ref{ex:shell} below). Some of our main results are formulated in the following meta-theorem, for more precise assumptions, notation, and statements see Section \ref{Section_res}, and for details see Proposition \ref{PropGWP}, Sections \ref{Subs.Euler}, \ref{3D Euler}, \ref{sec:lar-data}, \ref{sec:inf-dim}.

\begin{thm}\label{thm-meta}
Assume \eqref{Euler1} preserves the $L^2-$norm, that is, $t \mapsto \|u(t)\|_{L^2}$ is constant in time (see assumptions \eqref{cnclg}--\eqref{zmcd} on $B$ below). 
Then there exists a set $\Sigma\subset H^s(\T^d)$ such that 
\begin{enumerate}
    \item \eqref{Euler1} admits a global dynamics on $\Sigma$.
    \item The constructed solutions enjoy a slow growth property as $t \to \infty$ (see estimate \eqref{Bound_on_solutions} below). 
    \item $\Sigma$  contains a (large) set of initial conditions supported on finitely many Fourier modes. 
    \item $\Sigma$ contains arbitrary large data
\end{enumerate}
If in addition the equation \eqref{Euler1} possesses another coercive conservation law (other than $L^2$ norm), then the equation admits a non-trivial invariant measure $\mu$ such that $\mu(\Sigma)=1$. Finally, if the equation \eqref{Euler1} possesses infinitely many conservation laws (for example Casimirs), then the support of $\mu$ is not finite dimensional, that is, it is not locally contained in any finite dimensional manifold. 
\end{thm}


Our method is based on the fluctuation-dissipation method for 
Galerkin approximations and a passage to the limit with the help of the  Bourgain globalization technique  \cite{bourg94} (this approach, called IID limit\footnote{Inviscid-Infinite-dimensional limit}, already appeared in \cite{syNLS7, syyu} in the context of energy supercritical Schr\"odinger equations, and \cite{latocca} for the 2D Euler). Here we develop it further in order to tackle equations of the 3D Euler type which was until this work unreachable by the existing invariant measure arguments. Specifically, a key idea that enables us to treat the 3D Euler relies on analysis of the statistical ensemble rather than the infinite-dimensional limiting measure. It turns out that this new perspective circumvents several difficulties related to lack of conservation inherent to the Euler system. Recall that the fluctuation-dissipation method
introduces a stochastic forcing to the equation, which randomizes the solutions and presumably chooses stable orbits. In order to balance the energy, one adds a dissipation, which
allows for a use of techniques from dissipative equations. Then, one passes the strength of dissipation and the forcing to zero and recovers solutions of the original equation \eqref{Euler1}. 

Let us first fix the notation and provide motivation for our study. We assume that the equation \eqref{Euler1} is posed on a $d$-dimensional flat torus $\T^d := [0, 2\pi]^d$ with $d \leq 3$.   We remark that the arguments can be easily adapted to non-square torus, and furthermore
most arguments can be modified for the bounded domains with appropriate boundary conditions. The solution $u$ has the range $\R^k$ or $\mathbb{C}^k$, where $k \geq 1$ is an integer, most frequently $k = 1$ (as for generalized SQG equation) or $k = d$ (as for Euler equation). 
Concerning the nonlinearity, we assume that  $B$ is a bilinear function satisfying usual cancellation and regularity properties detailed below in \eqref{cnclg}--\eqref{libb}. In addition, 
if $k > 1$, we assume that the range of $B$ consists of divergence free functions, and consequently the divergence free condition is preserved by the flow of \eqref{Euler1}. 

Since we work on a torus, as usual we assume that the initial condition $u(0) = u_0$ has zero mean, that is, 
\begin{equation}
\int_{\T^d}u_0(x)dx = 0 
\end{equation}
and  below we impose assumption on $B$ such that the zero mean condition is preserved by the flow. It means that for any $t$
\begin{align}
\int_{\T^3}u(t,x)dx = \int_{\T^3}u_0(x)dx = 0 \,.
\end{align}
The major feature of \eqref{Euler1} that we use in the present manuscript is the invariance of $L^2$ norm by the equation which follows from assumption \eqref{cncl} below. The 
invariance of  $L^2$ norm can be interpreted as a conservation of (kinetic) energy which is ubiquitous in physical sciences.


Equations of the form \eqref{Euler1} have so many applications in geophysics, astronomy, or meterology that it would be impossible to list them all here. An interested reader can 
find many applications for example in books \cite{Chemin1998, MarchioloPulvirenti1994}.
Mathematically, under mild additional assumptions, \eqref{Euler1} is known to be locally well posed  in the Sobolev spaces $H^s$ for any large $s$. For example 
if \eqref{Euler1} is 3D Euler equation or SQG equation, then it is well posed respectively for $s > \frac{5}{2}$ or $s > 3$ (see for example 
\cite{majdabertozzi} and \cite{Majda1984}). 
Note that  local well posedness of  SQG equation was proved by geometrical arguments in \cite{Inci2018} if $k > 2$. 
In Section \ref{SectionLWP} below, we recall details of the proof for general form \eqref{Euler1} and provide lower bound on the existence time, with explicit dependence on relevant parameters. 

However an understanding of the global well posedness of smooth solutions of \eqref{Euler1} (e.g. 3D Euler or SQG equation) is one of the most famous problems in the fluid dynamics. Well known sufficient criteria for the global existence are given as a boundedness of certain norms, see for example a sufficient condition of Beale-Kato-Majda
\cite{bkm}, however they require stronger 
assumptions than known 
 a priori bounds, for example,  given by the conservation of the energy
\begin{equation}
\mathcal{E}(u)=\frac{1}{2}\int_{\T^3}|u(t,x)|^2dx \,.
\end{equation}
We remark that the conservation of the energy $\mathcal{E}$ also plays a crucial rule in the present manuscript, as detailed below. 
More geometrical conditions on the global existence were first obtained by \cite{ConstantinFeffermanMajda1996}, see also \cite{HouLi2006} and references therein, however the validity of   the conditions for particular problems is not known.
We remark that in two dimensions (2D), the Euler equation is 
known to be globally well posed in $L^\infty$ see \cite{Judovic1963}, due to different structure of the vorticity equation (vortex stretching term $(\omega \cdot \nabla) u$ is missing, and therefore it is
merely a transport equation). 

On the other hand, it is known that some solutions of \eqref{Euler1}
can grow-up (diverge to infinity in some norm as $t \to \infty$)
or even blow-up (diverge to infinity in some norm as $t \nearrow T < \infty$). 
Indeed, even for 2D Euler equation it was proved in 
\cite{ElgindiMasmoudi2020} that $C^k$ norm of the solution can diverge as $t \to \infty$ for solutions starting from $C^k$ initial conditions. Even before, Sverak and Kiselev \cite{KiselevSverak2014} (see also \cite{Zlatos2015}) proved that there are grow-up solutions of 2D Euler equation with double exponential growth of derivatives of vorticity.

The existence of blow-up for 3D Euler equation was a long standing open problem and it was predicted in several numerical attempts, see \cite{Gibbon2008} and references therein for older results. More recently, a convincing numerical scenario was proposed in \cite{LuoHou2014}.
Recently, the blow-up solutions for 3D Euler equation on $\R^3$ were constructed by Elgindi in \cite{Elgindi2021}. The solutions consdiered in \cite{Elgindi2021} have $C^{1,\alpha}$ regularity, which is much lower than the regularity considered in the present paper. 
On the other hand, it was shown in \cite{VasseurVishik2020} that if there are blowing-up solutions of 3D Euler equation, then they are unstable. Tao \cite{tao2016finite} has introduced an averaged version of the $3D$ Navier-Stokes equations and proved finite time blow up. Due to the perturbative nature of the dissipation term in this model, the same explosion holds for the averaged version of $3D$ Euler (Tao's model with zero viscosity). Let's notice that the nonlinear averaged operator of Tao's model satisfies the assumptions on $B$ below, in particular the crucial cancellation property \eqref{cncl} allowing the conservation of kinetic energy.
For the generalized SQG equation (see Example \ref{ex:active-scalar} below)  the 
existence of a finite time blow-up was shown in  \cite{KiselevRyshikYaoZlatos2016} for non-smooth initial data (vortex patches).   In summary, there are scenarios in which the solutions of \eqref{Euler1} can develop a singularity from regular initial conditions. A central question we are interested in,  is whether the blow up is generic, or it occurs only for well prepared initial data. Before proceeding, let us first discuss known global solutions. Again, our problem \eqref{Euler1} is very general, there is a huge amount of the fluid literature, and therefore we restrict our discussion to the Euler and SQG equations.  

First, observe that any solution $(u_1, u_2)$ of 2D Euler equation on $\T^2$, that is depending on $(x_1, x_2)$, can be extended to a global solution $(u_1, u_2, 0)$ of 3D Euler equation  on $\T^3$, still depending on $(x_1, x_2)$. Since 2D Euler equation 
is globally well posed in $L^\infty$, we obtain large class of global solutions, which are, however, not truly three dimensional, and we believe that these flows are not stable in 3D. 
There are also several special solutions, see for example the 
construction of DiPerna and Lions \cite{Lions2013} or Yudovich \cite{Yudovich2000} of
global solutions of the form (shear flows)
\begin{equation}
u(t, x) = (u_1(x_2), 0, u_3(x_1 - tu_1(x_2))) 
\end{equation}
of 3D Euler equation with $p \equiv 0$. Note, that for smooth, periodic functions $u_1, u_3$ one obtains that $\|u(t, \cdot)\|_{W^{1, p}}$, $p > 1$, diverges to infinity
as $t \to \infty$. Also, there are infinitely many equilibria for 3D Euler equation called Beltrami flows, that satisfy 
\begin{equation}\label{Bflow}
\textrm{curl} u(x) = \lambda(x) u(x)
\end{equation}
for some $\lambda (x) \neq 0$, see see \cite{majdabertozzi}. However, most of the known solutions of \eqref{Bflow} relies on 2D stream function, and as such the resulting 
flows depend on two variables only, the exception is GABC flow, see \cite{Arnold1965}. It is also not expected that these stationary flows are generic solutions of the 3D Euler equation. 
Finally, we remark that there are also global solutions on different domains such are shear flows in the whole space, axi-symmetric flows in cylindrical domains, or Hill's spherical vortex for the rotating fluid, see \cite{majdabertozzi}. 

There are also weak solutions to \eqref{Euler1}, and in particular to the  3D Euler equation, that are global in time. However, weak solutions possess a low regularity and it is not known if they are unique. In fact, the non-uniqueness is known if the assumed regularity is sufficiently small, see for example 
\cite{DeLellisSzekelyhidi2009, DeLellisSzekelyhidi2013, Isett2018, BuckmasterEtAl2023}. The existence of weak solutions can be established as inviscid limit of solutions to the Navier-Stokes equation, see Proposition \ref{UniformLWP} below for details. In addition, it was proposed in \cite{BardosTitiWiedemann2012} that the vanishing viscosity can serve as a selector for the weak solutions.  

In the present manuscript we are interested in the global smooth solutions of \eqref{Euler1} which are stable in a probabilistic sense as detailed below. As in our previous work \cite{foldessy}, an important step of our methods relies on the fluctuation-dissipation method. Specifically, we add a stochastic forcing of order $\sqrt{\eps}$ to \eqref{Euler1} which randomizes the dynamics and presumably make the solutions generic. In order to balance the influx of the (stochastic) energy to the system, we also add a dissipation term of order $\eps$. The dissipation term is carefully chosen so that its smoothing effect allows us to prove the well posedness of the dynamics. In addition, we show that for each $\eps > 0$ there are  invariant measures $\mu_\eps$ and in the support of $\mu_\eps$ are almost surely global solutions of the modified dynamics with a controlled growth. By passing $\eps \to 0$, we show that $\mu_\eps$ converge to the limiting measure $\mu$, which is invariant under the dynamics of \eqref{Euler1} and supports global solutions. 

The fluctuation-dissipation method naturally leads to a study of dissipative and stochastic problems arising in fluid dynamics. Probabilistic well-posedness, invariant measures, and related questions were investigated in the context of fluid dynamics in the last decades under various assumptions on the noise and underlying problem, see for example 
\cite{BricmontKupiainenLefevere2001, DaPratoDebussche2003, 
FlandoliMaslowski1, FGHRT, FGHRW, 
HairerMattingly2008, KS12}. More specifically, the fluctuation-dissipation was used for the 2D Euler equation in \cite{kuk_eul_lim,KS12}, SQG equation \cite{foldessy}, as well as  KdV, Benjamin-Ono, Klein-Gordon, or NLS equations \cite{KS04,sykg, sybo}. 
Note that there is another construction of
invariant states based on the Gibbs measures (see e.g. \cite{bourg94,bourgNLS96,zhd,zhidwave,zhidnls,tzvNLS,tzvNLS06,btt18}). Notice, however, that, as already identified in \cite{NahmodPavlovicStaffilaniTotz2018}, a `traditional way' of constructing Gibbs measure cannot be used for SQG or 3D Euler equation since the nonlinearities cannot be defined in the sense of distributions. 

The important feature of our setting is that we assume the existence of only one conserved quantity, the $L^2$ norm. The other invariants are only used in derivation of quantitative properties of the support of the invariant measure. In addition, we can only rely on the local-well well-posedness of \eqref{Euler1} since in general the global well-posedness is either not known or does not hold true. The major advantage of the 
probabilistic approach is that we obtain a priori bounds on the invariant measures that are not available for the deterministic dynamics. Specifically, due to our choice of the dissipation operator, we obtain exponential moment bounds on high Sobolev norms. We would like to stress, that our main goal is to pass $\eps \to 0$ in front of the stochastic and dissipative terms. In particular, if $\eps \approx 0$, then the smoothing and random effects are very small compared to the deterministic evolution of \eqref{Euler1}. 
Moreover, even the moment bounds in high Sobolev norms were not sufficient for proofs of our main results.

To resolve these issues, we work with the finite dimensional projections of \eqref{Euler1} onto Fourier modes, see \eqref{Euler(N)}, \eqref{ID_Euler(N)} below. The fluctuation-dissipation method for the finite dimensional projections follow standard arguments, but  very special care should be taken in the choice of the smoothing operator. It should provide sufficient smoothing and control on the temporal growth of the norms. However, at the same time, it must give enough flexibility for integration against non-linear functions needed in the proof of infinite-dimensionality. We establish that for each projection and $\eps > 0$ there is an invariant measure of the finite dimensional dynamics with the moment bounds that are independent of $\eps > 0$. Afterwards, we pass to the limit $\eps \to 0$ and obtain an invariant measure of the deterministic finite dimensional system. This result is probably known to experts, but we could not find it in the literature in full generality. It is essential to obtain moment bounds that are indepedendent of the dimension, which allows us to pass to the limit and treat the full infinite-dimensional problem. The proof of the invariance of the measure is already non-trivial and requires rather delicate decomposition of the probability space. 

The passage of the dimension $N$ of the projection to infinity is the main technical achievement of the present manuscript. By using the uniform in dimension estimates, we show that there exists an invariant measure $\mu$ of \eqref{Euler1}. The study of $\mu$
 requires a delicate application of  the Bourgain's argument \cite{bourg94} and deterministic estimates on the maximal existence time. As such we establish that for any projection the solutions have controlled growth with high probability. Importantly neither the probability nor the growth depends on dimension of the projection. This observation allows us to pass to the limit and obtain that the functions in the support of finite-dimensional invariant measures are in fact global solutions of the full dynamics. As mentioned above, the approach that consists in combining fluctuation-dissipation and Bourgain's Gibbs measure argument has been originally developed in \cite{syNLS7}. However, some important conclusions of \cite{syNLS7} heavily depended on the existence of two coercive conservation laws; this property is not satisfied by the $3D$ Euler equations. In the present work, we were able to tackle this problem by adopting a different perspective: instead of relying our analysis on the support of the infinite-dimensional limiting measure $\mu$, we make a careful analysis of the statistical ensemble to obtain global solutions for data living in the statistical ensemble and not necessarily  belonging the support of $\mu$.
It is important to notice that the concerned data are in the statistical ensembles of the finite dimensional projections. To the authors' best knowledge this observation was not used in the previous works. We also show that the flow of \eqref{Euler1} on the support of $\mu$ is well defined and global.

The rest of the manuscript is dedicated to quantitative properties of $\mu$ if the system possess additional conservation laws. In particular, if there is another coercive conservation law as for some shell models, or generalized SQG equation, we show that the support of $\mu$ contains arbitrary large functions. Of course, these functions lie on the global trajectories, but the large global solutions arise from the global ensemble as well (see remarks above). The arguments follow standard procedures, only the extra care need to be taken when passing to appropriate limits. The final section is dedicated to the infinite-dimensionality of the support of the invariant measure, which follows if there are infinitely many conserved quantities. We remark that if there is only finitely many invariants, say $K$, then it is possible to show that the support of $\mu$ is at least $K$-dimensional. However, we were not able to locate in the literature any good example, thus  for the sake of clarity we only discuss a specific model given by Casimirs of the generalized SQG equation. However, we believe that our argument can be generalized if necessary. The main challenge of this part is an appropriate choice of the test function in conjunction with the properties of our dissipation operator. At this moment we justify our choice of non-linear operator (combination of $p$-Laplacian with its higher order version) rather than a simpler linear one.

Next, let us provide assumptions on the bi-linear term $B$ in \eqref{Euler1}. Fix any integer $k \geq 1$, usually $k = 1$ or $k = d$. 
Assume a solution of \eqref{Euler1} is $k$-dimensional vector field, that is, $u : \T^d \to \mathbb{R}^k$ or $\mathbb{C}^k$. 
For any $m \geq \frac{d}{2}$  define $\mathcal{H}^m = \{u \in (H^m)^k : \textrm{div } u = 0 \textrm{ if } k \geq 2\}$, where $H^m$ is the usual Sobolev-Slobodeckij, or equivalently potential space. 
Observe that if $k \geq 2$, the divergence is well defined point-wise since $u$ is sufficiently smooth. 
We assume that $B: \mathcal{H}^m \times \mathcal{H}^m \to  \mathcal{H}^{m-1}$
is a bilinear map that satisfies for any $2 \leq k \leq m - 1$ and any $u, v, w \in \mathcal{H}^k$ 
\begin{align}\label{cnclg}
\langle B(v, u), w\rangle &= - \langle B(v, w), u\rangle \,, \\ \label{upb}
\|B(u, v)\|_{H^k} &\lqq{k}  \|u\|_{H^{k}} \|v\|_{H^{k + 1}}\,,   \\ \label{libb}
\|B(u, v)\| &\lqq{} \|u\|_{H^{\ell}} \|v\|_{H^{3-\ell}} \qquad \textrm{for any} \quad \ell \in \{0, 1, 2\}  \,,  \\ \label{zmcd}
\int_{\T^d} B(u, u) dx  &= 0 \,.
\end{align}  
By setting $w = u$ in \eqref{cnclg} we obtain
\begin{equation}\label{cncl}
\langle B(v, u), u\rangle = 0 \,.
\end{equation}
There are many examples of $B$ in the fluid dynamics that satisfy the above hypotheses. We provide some for illustration, but we do not 
attempt to give a complete list or optimal conditions. 

\begin{Ex}\label{ex:euler}
The first example is the classical Euler equation in two or three dimensions ($d = 2, 3$), in which case $u : \T^d \to \R^d$ is a vector field. Then, 
\begin{equation}
B(u, v) = \Pi ((u \nabla) v) \,,
\end{equation}
where $\Pi$ is the projection on divergence free vector fields. The property \eqref{cnclg} follows directly from the integration by parts and the fact that $u$ is divergence free. Since 
$H^k$ in 2D or 3D is an algebra for $k \geq  2$, by classical properties of Helmholtz-Leray projection $\Pi$ one has
\begin{equation}
\|B(u, v)\|_{H^k} \lqq{k} \|(u \nabla) v\|_{H^k}  \lqq{k} \|u\|_{H^k} \||\nabla v|\|_{H^k} \lqq{k} \|u\|_{H^k} \|v\|_{H^{k+1}}
\end{equation}
and \eqref{upb} follows. Also, since $H^2 \hookrightarrow L^\infty$ when $d = 2, 3$, 
\begin{equation}
\|B(u, v)\| \lqq{} \|(u \nabla) v\| \lqq{} \|u\| \||\nabla v|\|_{L^\infty} \lqq{k} \|u\| \|v\|_{H^3}
\end{equation}
and \eqref{libb} follows for $\ell = 0$. Similarly one obtains  \eqref{libb}  for $\ell = 2$. Finally,  since $H^1 \hookrightarrow L^4$
\begin{equation}
\|B(u, v)\| \lqq{} \|(u \nabla) v\| \lqq{} \|u\|_{L^4} \||\nabla v|\|_{L^4} \lqq{k} \|u\|_{H^1} \|v\|_{H^2} \,,
\end{equation}
and we established \eqref{libb}. Finally, by the properties of $P$, $B(u, u) = \textrm{div}(u \otimes u) + \nabla p$ for appropriate function $p$, and \eqref{zmcd} follows 
from divergence theorem and periodicity of $p$. 

We remark that if $d = 3$, then one can add Coriolis force $\Pi (f \times u)$ to $B$ as
\begin{equation}
    B_f(u, v) := B(u, v) + \frac{1}{2} (\Pi (f \times u) + \Pi (f \times v)) \,,
\end{equation}
where $f$ is a fixed vector field on $\T^d$. Then, by properties of triple product,  \eqref{cnclg} follows from
\begin{equation}
\langle f\times u, w\rangle = - \langle f\times w, u\rangle \,.
\end{equation}
The remaining assumptions on $B$ hold if $f$ is smooth. 

The Euler's equation has also the vorticity formulation. Specifically, if we define 
$\omega = \textrm{curl}(u)$, then $\omega$ satisfies
\begin{equation}
   0 = \partial_t \omega + u\nabla \omega + \omega\nabla u =: \partial_t \omega + \tilde{B}(u, \omega) \,.
\end{equation}
Note that the new bi-linear form $\tilde{B}$ in general does not satisfy \eqref{cncl} due to the presence of the vorticity stretching term $\omega \nabla u$. Hence if $d = 3$, we have to work with the velocity formulation \eqref{Euler1} and since there are no other (useful) conservation laws, the results of Section \ref{sec:lar-data} and \ref{sec:inf-dim} do not apply. We remark that we could not find an use for other conservation laws known to 3D Euler equation such as helicity, and invariants originating in the Kelvin circulation theorem. 

If $d=2$, then the situation changes and $\omega =  \textrm{curl}(u)$ can be interpreted as a
scalar function. Also, the vorticity stretching term $\omega\nabla u$ vanishes and we obtain that
$\tilde{B}(u, v) = u \nabla v$ and all assumptions \eqref{cnclg}-\eqref{zmcd} are satisfied. In fact, due to the Biot-Savart law, $u = \nabla^\perp \Delta^{-1} \omega$ with $\nabla^\perp g = (-\partial_{x_2}g, \partial_{x_1}g)$, and therefore 2D Euler equation is a special of an active scalar equation discussed in Example \ref{ex:active-scalar}. Thus, if one can works with 2D Euler equation in the velocity formulation (unknown field is velocity), then all results in Sections \eqref{SectionLWP} - \eqref{ASection6InvarMeas} apply. If one works in the vorticity formulation (unknown function is vorticity), in addition, there is other conservation law (energy):
\begin{equation}
    H(\omega) = \int_{\T^2} |\nabla^\perp \Delta^{-1} \omega|^2 \, dx 
\end{equation}
and in addition to the results of Sections \eqref{SectionLWP} - \eqref{ASection6InvarMeas}
the conclusions of Section \ref{sec:lar-data} are valid as well (for details see the discussion at the beginning of Section \ref{sec:lar-data}).

Also, 2D Euler equation has other invariants called Casimirs (see \cite{Wolibner1933, Holder1933, majdabertozzi}):
\begin{equation}\label{casimir}
    \int_{\T^2} f(\omega) \, dx \,,
\end{equation}
where $f:\R \to \R$ is any sufficiently regular function. Hence, the results of Section \ref{sec:inf-dim} are also valid for 2D Euler equation in the vorticity formulation. 

\end{Ex}

\begin{Ex}\label{ex:active-scalar}
Another important class of examples covered in the present paper are active scalar equations in two or three dimensions ($d = 2, 3$). In this case $u : \T^d \to \R$ is scalar valued and 
\begin{equation}
B(u, v) = K[u] \nabla v \,,
\end{equation}
where $K: H^m \to \mathcal{H}^m$, $m \geq 0$ has range in the class of  divergence free vector fields. Again, \eqref{cnclg} follows from  integration by parts 
($m >\frac{d}{2}$ and the integrals are well defined) and the divergence free condition. By using algebra properties of $H^k$, we have 
\begin{equation}
\|B(u, v)\|_{H^k} = \| K[u] \nabla v\|_{H^k}  \lqq{k} \|K[u]\|_{H^k} \||\nabla v|\|_{H^k} \lqq{k} \|u\|_{H^k} \|v\|_{H^{k+1}}
\end{equation}
and \eqref{upb} follows. In addition, similarly as above, $H^2 \hookrightarrow L^\infty$ in two or three dimensions yields
\begin{equation}
\|B(u, v)\| \lqq{} \|K[u]\| \||\nabla v|\|_{L^\infty} \lqq{k} \|u\| \|v\|_{H^3}
\end{equation}
and similarly one obtains \eqref{libb} for $\ell = 0$.  Finally, $H^1 \hookrightarrow L^4$ implies 
\begin{equation}
\|B(u, v)\| \lqq{} \|K[u]\|_{L^4} \||\nabla v|\|_{L^4} \lqq{k} \|K[u]\|_{H^1} \||v|\|_{H^2} \lqq{} \|u\|_{H^1} \|v\|_{H^2}\,,
\end{equation}
and \eqref{libb} follows. Also, \eqref{zmcd} is satisfied since $K[u]$ is divergence free and therefore  $B(u, u) = \textrm{div}(K[u] u)$. 

The active scalar equations in particular cover SQG equation and its generalizations,  see \cite{KiselevRyshikYaoZlatos2016},
 for which $K[u] = \nabla^\perp (-\Delta)^{-1 + \alpha}u$ for any $\alpha \leq \frac{1}{2}$ (possibly negative).  They also cover
diffusive MG equation (see \cite{FriedlanderRusinVicol}), where one can weaken the viscosity by replacing Laplacian by its power smaller than 1.

We remark 
that our construction of invariant measures remains valid for more singular active scalar equations, for example for generalized SQG with some $\alpha > \frac{1}{2}$, but 
in such case it is not clear how to establish the uniqueness of (smooth) solutions of the deterministic equation. 
Overall, we obtain that 
for general active scalar equations all results of Sections \eqref{SectionLWP} - \eqref{ASection6InvarMeas} apply.

The active scalar equations also have Casimirs defined in \eqref{casimir} as conservation laws
and in particular the results of Section \ref{sec:inf-dim} are valid. 

Next, we discuss only special cases of the active scalar equation, namely
$K[u] = \nabla^\perp (-\Delta)^{-1+\alpha}u$ as for the generalized SQG.  Then 
\begin{equation}
    H(u) = \frac{1}{2}\int_{\T^2} ((-\Delta)^{\frac{-1 + \alpha}{2}}u)^2 \, dx
\end{equation}
is conserved and then results of Section \ref{sec:lar-data} hold true (see discussion in Section \ref{sec:lar-data}). 

We remark that the presence of the conservation law that would imply results of Section \ref{sec:lar-data} is not guaranteed for the active scalar equations, see for example incompressible porous medium equation, where 
$K[u] = \nabla^\perp \partial_{x_1}(-\Delta)^{-1}u$, see \cite{Elgindi2017,KiselevYao2023} and references therein. 
\end{Ex}

\begin{Ex}\label{ex:shell}
In this example we verify the assumption for some shell models in fluid dynamics. Assume that $(\phi_j)_{j = 1}^\infty$ is an orthonormal basis of some Hilbert space $H$ consisting 
of zero mean functions.   
For any $u, v \in H$ set 
\begin{equation}
u = \sum_{j = 1}^\infty u_j \phi_j\,, \qquad v = \sum_{j = 1}^\infty v_j \phi_j \,,
\end{equation}
where $(u_j)$ and $(v_j)$ are sequences of complex numbers. For definiteness (and convenience) we take $(\phi_j)$ to be the basis consisting of eigenfuctions of the Laplacian. Then, 
 sabra model with parameters $a, b\in \R$ and  $\lambda > 1$ analyzed in \cite{ConstantinLevantTiti} is given by
\begin{equation}
B(u, v) = -i k_0 \sum_{j = 1}^\infty \left(a\lambda^{j+1} v_{j + 2} \overline{u}_{j + 1} + b \lambda^j v_{j + 1} \overline{u}_{j-1}+ a \lambda^{j - 1}u_{j-1}v_{j-2} + b\lambda^{j-1}v_{j-1}u_{j-2} \right)\phi_j
\end{equation} 
and  Gledzer-Okhitani-Yamada (GOY) shell model  is defined as 
\begin{equation}
B(u, v) = -i k_0 \sum_{j = 1}^\infty \left(\overline{a\lambda^{j+1} v_{j + 2} u_{j + 1} + b \lambda^j v_{j + 1} u_{j-1}+ a \lambda^{j - 1}u_{j-1}v_{j-2} + b\lambda^{j-1}v_{j-1}u_{j-2}} \right)\phi_j
\end{equation}
 Then, in either case,
\begin{equation}
\mathcal{R}\int B(v, u) \bar{w} dx = - \mathcal{R} \int B(v, w) \bar{u} dx 
\end{equation}
which in our notation is equivalent to \eqref{cnclg}. Properties  \eqref{libb} and \eqref{upb}  follow directly from \cite[Proposition 1]{ConstantinLevantTiti}. 
Furthermore, \eqref{zmcd} follows from the zero mean of $\phi_j$. 
Thus, all results in Sections \eqref{SectionLWP} - \eqref{ASection6InvarMeas} apply to these shell models. We remark that our framework also covers other shell models  and ODE approximations such as Lorenz 96 or 63 systems.

In addition there is another conservation law
\begin{equation}
    H(u) = \frac{1}{2}\sum_{n = 1}^\infty \left(\frac{-a}{a + b} \right)^n |u_n|^2, \qquad \textrm{if }
    u = \sum_{j = 1}^\infty u_j \phi_j \,,
\end{equation}
and it physical range of parameters is $|a/(a+b)| < 1$. Then, the assumptions of Section \ref{sec:lar-data} are satisfied and the corresponding results  on large data follow (see Section \ref{sec:lar-data} for details). In general, the shell models do not have infinitely many invariants as the active scalar equations, and therefore statements in Section \ref{sec:inf-dim} are not valid.  
\end{Ex}

Next, let us formulate our main results in a separate sub-section.



\subsection{Statement of the results.}\label{Section_res}
\begin{thm}\label{main:GWP}
Let $s\geq 4$ and $\xi:\R_+\to\R_+$ be one-to-one concave function. There is a measure $\mu =\mu_{s,\xi}$ concentrated on $H^s$ and a subset $\Sigma=\Sigma_{s,\xi}\subset H^s$ such that
\begin{enumerate}
\item $\mu(\Sigma)=1$
\item The equations \eqref{Euler1} are globally well-posed on $\Sigma$. We denote the constructed flow by $\phi_t$
\item The measure $\mu$ is invariant under $\phi_t$
\item $\phi_t\Sigma=\Sigma$ for all $t\in\R$
\item For all $u_0\in \Sigma$, we have
\begin{align}
\|\phi^t u_0\|_{s}\leq C(\|u_0\|_{s})\xi(\ln(1+|t|))\quad t\in\R.
\end{align}
\item For all $M>0$, $\mu(\{u_0\in\ \Sigma\ | \ \|u_0\|_{H^s}>M\})>0$. This states a large data property for the $\Sigma$.
\end{enumerate}
\end{thm}
\begin{thm}\label{main:AbsCont}
   When the equation \eqref{Euler1} possesses two coercive conservation laws, then there is a functional $F$ such that
\begin{equation}
\P(\{F(u)\in \Gamma\})\leq f(\lambda(\Gamma)),\label{Bound:ABC}
\end{equation}
for any Borel set $\Gamma\subset\R$, where $\lambda$ stands for the Lebesgue measure on $\R$, and $f$ is a non-negative continuous function that vanishes at $0.$ 
\end{thm}
\begin{rmq}[Remark on the statistical ensemble $\Sigma^{Euler}$ associated to the $3D$ Euler system] 
Let's start by describing the statistical ensemble $\Sigma$. It is roughly the union of sets $(\Sigma_N)_{N=1}^{\infty}$, where $\Sigma_N$ is constructed on the basis of $\mu_N$ (which is an invariant measure for the $N$-Galerkin projection of the equation), $\Sigma_\infty$ is associated to $\mu\ ("=\lim_N\mu_N\,")$.
Due to lack of second coercive conservation law for $3D$ Euler equation, the infinite-dimensional limiting measure $\mu$ is not well-understood (so is $\Sigma_\infty$). Namely, we couldn't rule out singularity of its support. However, the remaining parts of the statistical ensemble, namely the sub-ensembles $\Sigma_N$, possesses several properties. Global regularity for $3D$ Euler equation is proven in particular for all data living on these finite-dimensional sub-ensembles, for which we prove the corresponding results of Theorem \ref{main:GWP}. Since data of finite number of Fourier modes are $C^\infty$, we prove that associated the solutions remain $C^\infty$ for all times (see Subsection \ref{Subs.Euler}). Theorem \ref{main:AbsCont} also holds for $3D$ Euler equation in the following sense:\\
The quantity $\|u_N\|_{L^2}^2$, where $u_N$ is the random variable associated to $\Sigma_N$, satisfies \eqref{Bound:ABC}. Similar consideration gives rise to the large data argument for the $3D$ Euler equation. \\
The Euler statistical ensemble $\Sigma^{Euler}$ is indeed invariant under the constructed flow. Which means that the evolution of any datum $u_0\in \Sigma^{Euler}$ will give, at any time, an element of $\Sigma^{Euler}$. We do not expect that this evolution will always stay finite-dimensional. So the contribution of the sub-ensemble $\Sigma_\infty$ should not be trivial so that it can provides functions supported on infinite number of Fourier modes. This point remains a conjecture.
\end{rmq}
Let us state a result about an infinite-dimensionality property:
\begin{thm}
  When the equation \eqref{Euler1} possesses an infinite number of 'good'  conservation laws, then $\Sigma$ is infinite-dimensional in the sense that: for any compact $K$ of finite Hausdorff dimension, $\mu (K)=0$. 
\end{thm}



\subsection{Notations}
Let $W$ be a Polish space.\\
\begin{itemize}
 \item $C_tW=C([0,T],W)$
 \item $\mathfrak{p}(W)$ is the set of probability measure on $W$
\item $L^p_tW=L^p_t([0,T],W)$
\item $W^{m,p}(\R^n,\R)=\{u\in L^p(\R^n,\R) \ |\ D^{\alpha}u \in L^p \ \text{for multi-index $|\alpha|\leq m$}\}$. We denote by $H^m$ the space $W^{m,2}$. $W^{m,p}(\R^n,\R^k)$ is defined by considering components of $u$.
\item $X^m_T=C_tH^m_x$
\end{itemize}
For any (complex valued) functions $f, g \in L^2$ define 
\begin{equation}\label{eq:scalar}
\langle f, g\rangle = \frac{1}{2} \int_{\T^d} f\overline{g} + \overline{f}g \, dx  = \mathcal{R} \int_{\T^d} f\overline{g} \,dx \,,
\end{equation}
where $\mathcal{R} z$ denotes the real part of $z$. Of course, $\langle f, g\rangle$ coincide with usual inner product for real valued $f$ and $g$.\\
For two function $f, g$ we denote $f \lleq g$ to represent the inequality $f \leq C g$ where some 'independence' consideration are made on $C$. For instance, if $C$ depends on a parameter
$q$ we may write $f \lqq{q} g$ and analogously, if $C$ depends on more than one parameter. 
 
The paper is organized as follows. In Section \ref{SectionLWP} we recall the local theory for the deterministic equation \eqref{Euler1} and its Fourier projection with the goal to obtain lower bounds on the existence time. We carefully track the dependencies of constants on the parameters. In Section \ref{sec:fdfd}, the fluctuation dissipation method for finite-dimensional projections is analyzed, in particular we establish the almost sure global well-posedness of solutions and the existence of the invariant measure together with the moment bounds. Section \ref{secl:invlim} concentrates on the inviscid limit, where the fluctuation-dissipation parameter is passed to zero. In particular, the existence and moment bounds of  invariant measures for finite dimensional projections are established. In Section \ref{SE_globalflow} we study invariant measures for the full infinite-dimensional problem \ref{Euler1}. In particular, we show that there is an invariant measure which is the limit of finite-dimensional projections. We also construct statistical ensembles global well-posedness on these ensembles with a controlled growth of the solutions. We also summarize the obtained results for the 3D Euler equation. The Section \ref{ASection6InvarMeas} contains the proofs of the invariance of the measure constructed in the previous section. In Sections \ref{3D Euler} and \ref{sec:lar-data} we show that under additional assumptions on the conserved quantities, there are arbitrarily large solutions in the support of the invariant measure, for 3D Euler equation and in general. Finally, Section  \ref{sec:inf-dim} contains proof of infinite-dimensionality of the invariant measure. 
 
\section{Uniform local theory, and local convergence}\label{SectionLWP}

In this section we recall, the local well posedness theory in higher order Sobolev spaces for finite or infinite dimensional projections of \eqref{Euler1}. In particular, we obtain lower bounds on the 
time of existence of solution, which is independent of the projection. 
Although the results are probably known to experts,  we could not locate needed statements in the literature,  and therefore we provide overview of proofs with keeping track of relevant dependencies on parameters.

For any integer $N \geq 1$ let $\Pi_N$ be the projection of a (real or complex) space $L^2$ on the finite dimensional space $E_N$ 
spanned by the 
eigenfunctions of $(-\Delta)$ corresponding to eigenvalues $\lambda$ with $\lambda \leq N$. Note that by increasing $N$ we add to $E_N$ all eigenfuctions corresponding to new eigenvalues, in 
particular if we add eigenfunction $\cos(k\cdot x)$, then we add eigenfunction 
$\sin(k\cdot x)$ as well.  Since the eigenfuctions are smooth, 
 $E_N = \Pi_N H^s(\T^d)$ is independent of  $s \geq 0$,  we use this observation below
without further reference.  With slight abuse of notation,  we assume that $\Pi_N$ can also act on vectors component-wise,  and if $N = \infty$, we assume that $\Pi_N$ is the identity on an appropriate space. 

Assume a solution of \eqref{Euler1} is $k$-dimensional vector field, that is, $u : \T^d \to \mathbb{R}^k$ or $\mathbb{C}^k$. 
For any $m \geq \frac{d}{2}$  define $\mathcal{H}^m = \{u \in (H^m)^k : \textrm{div } u = 0 \textrm{ if } k \geq 2\}$ and denote $\mathcal{E}_{N,m} = \Pi_N \mathcal{H}^m$. 
If $k \geq 2$, the divergence is well defined point-wise since $u$ is sufficiently smooth. 
Since
the domain is a torus, the eigenfunctions of $(-\Delta)$ are trigonometric functions, and it is standard to show that $\Pi_N$ preserves divergence free condition and 
$\textrm{div} (E_N)^k \subset E_N$. 

Fix $N \in \{1, 2, \cdots \} \cup \{\infty \}$ and consider the problem
\begin{align}
\dt u_N &= -B_N(u_N, u_N)  \,, \qquad \nabla \cdot u_N = 0 \,, \label{Euler(N)} \\
u_N(t=t_0) &= u_{0, N} \in \mathcal{E}_N\,,   \label{ID_Euler(N)} 
\end{align}
where $B_N = \Pi_N \circ B$. Fix $m > \frac{d}{2}$ and recall that $B: \mathcal{H}^m \times \mathcal{H}^m \to  \mathcal{H}^{m-1}$
is a bilinear map that satisfies \eqref{cnclg}--\eqref{cncl}.
Observe that by \eqref{eq:scalar},  the assumptions \eqref{cnclg} or \eqref{cncl} impose only restriction on the real part of the inner product if we work with complex valued solutions. 
Since the projection $\Pi_N$ is not increasing the norm, \eqref{upb} and \eqref{libb} hold with $B$ replaced by $B_N$
and with constants independent of $N$. Also, for any $w \in \mathcal{E}_{N, m}$ one has 
\begin{equation}
 \langle B(v, u), w\rangle = \langle B_N(v, u), w\rangle 
\end{equation} 
and we use this observation below without further reference. 

Next,  let us establish a local theory for the equation \eqref{Euler(N)} and provide a  lower bound on the blow-up time. 



\begin{prop}\label{UniformLWP}
Fix and integer  $m\geq 4$, a real number $R > 0$, $N \in \{1, 2, \cdots \} \cup \{\infty \}$, and $u_0 \in B_R(\mathcal{E}_N) := \{u \in \mathcal{E}_{N, m} : \|u\|_{H^m} < R \}$. 
Then, there exists $T$ independent of $N$ with
\begin{align}
T\ggeq R^{-1},
\end{align}
such that the problem \eqref{Euler(N)}, \eqref{ID_Euler(N)} is locally wellposed on $\mathcal{E}_{N,m}$ equipped with $H^m(\T^d)$ norm, that is, 
there exists a unique solution of \eqref{Euler(N)}, \eqref{ID_Euler(N)} in $C([0, T), \mathcal{E}_{N,m})$ 
denoted $u_N(t, u_{0, N})$ 
 that depends continuously on $u_{0, N}$. 
More precisely, if 
\begin{equation}
\lim_{n \to \infty}\|u_{0, N}^n - u_{0, N}\|_{H^m} = 0 \,,
\end{equation}
then for each $t \in [-T, T]$
\begin{equation}
\lim_{n \to \infty}\|u_N^{n}(t) - u_N(t)\|_{H^m} = 0 \,,
\end{equation}
where $u_N(t) = u_N(t, u_{0, N})$ and $u_N^{n}(t) = u_N(t, u^n_{0, N})$.
In addition, 
 if we define $\phi_t^N : \mathcal{E}_{N, m} \to \mathcal{E}_{N,m}$ as $\phi_t^N(u_{0, N}) = u_N(t, u_{0, N})$, 
 then $\phi^N$ is the (local) flow
of \eqref{Euler(N)}, \eqref{ID_Euler(N)} and 
\begin{align}
\| t\mapsto \phi_{t}^Nu_{0, N}\|_{X^m_T} \leq  2\|u_{0, N}\|_{H^m} \,, \label{Local_Growth_bound}
\end{align}
where $X^m_T$ was defined in Introduction (see Notations).
\end{prop}

Since  by \eqref{cncl}, \eqref{Euler(N)}  preserves the $L^2$ norm and all norms in finite dimensional spaces are equivalent,  standard arguments yield  the following result.

\begin{cor}\label{Cor-GWP-EulerN}
Under the assumptions of Proposition \ref{UniformLWP}, if $N$ is finite, then the equation \eqref{Euler(N)} is globally well-posed on $\mathcal{E}_{N,m}$.
\end{cor}

\begin{proof}[Sketch of proof of Proposition \ref{UniformLWP}]
Let us recall the local theory of \eqref{Euler(N)} in high regularity, as a limit of a regularized equation. For a more complete exposition in the context close to ours, see for example 
\cite{tao,kato,majdabertozzi}. For the rest of the proof we fix the dimension of the Galerkin projection to be $N \in \{1, 2, \cdots\} \cup \{\infty\}$ and it is essential to 
observe that all constants in the proof are independent of $N$. To keep the notation simple, we write  $u$   instead of  $u_N$, but we keep 
track of the dependencies on $N$. 


\vspace{.2cm}
\noindent
\textbf{Local well-posedness for smooth data.}
For fixed $\nu > 0$, and integers $m \geq 4$,  $N \in \{1, 2, \cdots\} \cup \{\infty\}$, and $u_0 \in \mathcal{E}_{N,m}$,  consider the initial value problem 
\begin{align}\label{NSequ}
\dt u + B_N(u, u)  &=\nu\Delta u  \,,  \\
u_{|_{t=0}} &=u_0\,. 
\end{align}
Since $ \mathcal{E}_{N,m}$ is invariant under the Laplacian,  all terms in \eqref{NSequ} belong to $ \mathcal{E}_{N,m}$. 
Fix $T \in (0, 1]$ and define 
the space $Y_T^m=X_t^m\cap L^2_tH_x^{m+1}$ endowed with the norm
\begin{align}
\|u\|_{Y_T^m}^2:=\sup_{t\in [0,T]}e^{\frac{-t}{T}}\left(\|u(t)\|_{H^m}^2+\nu\int_0^t\|u(\tau)\|_{H^{m+1}}^2d\tau\right) \,,
\end{align}
and we use the same norm on subspaces $\mathcal{E}_{N,m}$. 
Notice that the norm $\|\cdot\|_{Y^m_T}$ is equivalent to the norm $\|\cdot \|_{X^m_T}+\sqrt{\nu}\|~\cdot~\|_{L^2_tH^{m+1}}$
with the equivalence constants depending on $T$ and $\nu$, but independent of $N$. 
For already fixed $u_0$ and $t \in [0, T]$, define the map $\Phi_t : Y^m_T\to Y^m_T$
\begin{align}
\Phi_t(u)=S_\nu(t)u_0-\int_0^tS_\nu(t-s) B_N(u, u) d\tau \,, 
\end{align}
where $S_\nu(t)=e^{-\nu t\Delta}$ is the heat semigroup. 
%

Note that the heat semigroup $S_\nu$ preserves eigenspaces of $(-\Delta)$ and since $B_N(u, u)$  and $u_0$ are divergence free (if vector valued), we have that $\nabla\cdot\Phi_t(u)=0$,  that is, 
$\Phi_t(u)$ is divergence free.  Next, we prove  bounds on $\Phi_t$  in appropriate norms. 


For smooth $u_0 \in \mathcal{E}_{N,m}$ and smooth $u \in C^\infty([0, T), \mathcal{E}_{N,m})$, the Duhamel's principle implies 
\begin{align}\label{tes}
\dt\Phi_t(u) =- B_N(u, u)  +\nu\Delta \Phi_t(u), \qquad \Phi_0(u) = u_0.
\end{align}
By the parabolic regularity theory, (in the case $N = \infty$,  otherwise the assertion is trivial) $\Phi_t(u)$ is smooth, however 
we are interested in dependencies on parameters.  Evaluating $\dt  \|\nabla^m \Phi_t(u) \|^2$ and using \eqref{tes},  one obtains
\begin{equation}
\dt  \|\nabla^m \Phi_t(u) \|^2 + 2\nu\|\Phi_t(u)\|_{H^{m+1}}^2 = -2 \langle \nabla^m \Phi_t (u),  \nabla^m B(u, u) \rangle \,.
\end{equation}
Also, 
 \eqref{upb} yields
\begin{align}
|\langle \nabla^m \Phi_t (u),  \nabla^m B(u, u) \rangle|  \leq \|\Phi_t(u)\|_{H^{m+1}}\|B(u,  u)\|_{H^{m-1}}\leq
\nu\|\Phi_t(u)\|_{H^{m+1}}^2+\frac{C_m}{4\nu}\|u\|_{H^{m-1}}^2\|u\|_{H^m}^2 \,,
\end{align}
and therefore
\begin{align}
\dt \|\Phi_t(u)\|_{H^m}^2 + \nu\|\Phi_t(u)\|_{H^{m+1}}^2 &\lqq{\nu, m}\|u\|_{Y^m_T}^4\lqq{\nu, m} e^{\frac{t}{T}}\|u\|_{Y^m_T}^4 .
\end{align}
Integrating and then multiplying by $e^{\frac{-t}{T}}$, we arrive at
\begin{align}\label{mit}
\|\Phi_t(u)\|_{Y_T^m}\leq C_{m,\nu}T^\frac{1}{2}\|u\|_{Y_T^m} + \|u_0\|_{H^m} \,.
\end{align}
By a standard approximation argument, \eqref{mit} holds for any $u_0 \in \mathcal{E}_{N,m}$  and $u \in Y_T^m$. 
Similarly, if for fixed $u, v \in Y^m_T$ we denote $\Psi_t = \Phi_t(u) - \Phi_t(v)$ and $w = u - v$, then
\begin{align}
\dt\Psi_t =- B(w, u) - B(v, w)    + \nu\Delta \Psi_t(u), \qquad \Psi_0 = 0 \,.
\end{align}
 We have 
\begin{equation}\label{esp}
\begin{aligned}
\dt \|\Psi_t\|_{H^m}^2 + 2\nu\|\Psi_t\|_{H^{m+1}}^2 &\leq 2 |\langle \nabla^m \Psi_t, \nabla^m(B(w, u) + B(v, w)) \rangle| \\
&\leq \nu \|\Psi_t\|_{H^{m+1}}^2 + C_{\nu} \|B(w, u) + B(v, w)\|^2_{H^{m - 1}} \\
&\leq \nu \|\Psi_t\|_{H^{m+1}}^2 + C_{\nu} e^{\frac{t}{T}}(\|w\|_{H^{m- 1}}^2 \|u\|_{H^m}^2 + \|w\|_{H^{m}}^2 \|v\|_{H^{m-1}}^2) \,.
\end{aligned}
\end{equation}
Integrating in $t$, multiplying by $e^{-\frac{t}{T}}$, and using $\Psi(0) = 0$ we have 
\begin{equation}\label{cib}
\|\Phi_t(u)-\Phi_t(v)\|_{Y^m_T}\leq T^\frac{1}{2}C_{\nu,m}\|u - v\|_{Y^m_T} (\|u\|_{Y^m_T}+ \|v\|_{Y^m_T}).
\end{equation}
Finally, choose $T$ depending on  $m$ and $\nu$ such that $C_{m,\nu}T^\frac{1}{2} < \frac{1}{2}$. If $u_0 \in B_R(H^m)$,  by \eqref{mit}  for any $u \in B_{2R}(Y^m_T)$ one has 
\begin{equation}
\|\Phi(u)\|_{Y_T^m}\leq C_{m,\nu}T^\frac{1}{2} 2R + R < 2R \,,
\end{equation}
and therefore $\Phi(u)$ maps $B_{2R}(Y^m_T)$ to itself. Making $T$ smaller if necessary, we can assume $T^\frac{1}{2}C_{\nu,m} 4R  \leq \frac{1}{2}$, 
and consequently from \eqref{cib} follows 
\begin{equation}
\|\Phi(u)-\Phi(v)\|_{Y^m_T}\leq T^\frac{1}{2}C_{\nu,m}\|u - v\|_{Y^m_T} (\|u\|_{Y^m_T}+ \|v\|_{Y^m_T}) \leq \frac{1}{2} \|u - v\|_{Y^m_T}.
\end{equation}
Hence, $\Phi$ is a contraction on $B_{2R}(Y^m_T)$ with $T$ depending on $m, \nu$, and $R$. Then, the existence of a fixed point of $\Phi$, and therefore a solution to \eqref{NSequ}
follows from the Banach fixed point theorem. Uniqueness follows from \eqref{esp} with $m = 1$, $\Psi = w$, and Gronwall inequality. 

\vspace{.2cm}
\noindent
\textbf{Uniform estimates.}
Since the existence time $T$ depends on $\nu$, 
before passing $\nu \to 0$, we need a priori bounds 
 that are independent of $\nu$ and $N$.  Define $Q_m(t)=\frac{1}{2}\int_{\T^3}|\nabla^mu(t,x)|^2dx$ and note that if $u$ solves \eqref{Euler(N)},  then
\begin{align}
\dt Q_m(t)=\langle \nabla^m u,\nabla^m[-B(u, u)+\nu\Delta u]\rangle = I_1 + I_2 \,.
\end{align}
Integration by parts implies $I_2= -\nu\|u\|_{H^{m+1}}^2\leq 0$.  
Since the method for estimating $I_1$ is used again below, we provide a slightly more general statement for $B(v, u)$. 
First, since $B$ is bilinear, for any sufficiently smooth $u, v$
\begin{equation}\label{prr}
\partial_{x_i} B(v, u) = B(\partial_{x_i}v, u) + B(v, \partial_{x_i} u)\,.
\end{equation}
Let $\alpha = (\alpha_1, \cdots, \alpha_d)$ and $\beta = (\beta_1, \cdots, \beta_d)$ be multli-indicies, and as usual $|\alpha| = |\alpha_1| + \cdots + |\alpha_d|$, 
$\alpha + \beta = (\alpha_1 + \beta_1, \cdots, \alpha_d + \beta_d)$, and $D^\alpha u = \partial^{\alpha_1}_{x_1}\cdots  \partial^{\alpha_d}_{x_d}u$. 
Then, by \eqref{prr} and the Leibniz formula
\begin{equation}
I_1 = \sum_{|\alpha| = m} \langle D^{\alpha} u, D^{\alpha} B(v, u)\rangle = \sum_{|\alpha| = m} \sum_{\beta + \gamma = \alpha } c_\beta \langle D^{\alpha} u,  B(D^{\beta} v, D^{\gamma}u)\rangle\,.
\end{equation}
If $\beta = 0$ in the previous sum, then \eqref{cncl} implies 
\begin{equation}
\langle D^{\alpha} u,  B(D^{\beta} v, D^{\gamma}u)\rangle = \langle D^{\alpha} u,  B(v, D^{\alpha} u)\rangle = 0\,.
\end{equation}
From \eqref{libb} with $\ell = 2$ and Poincar\' e inequality  follow for $|\beta| = k \in \{1, \cdots, m - 2\}$
 \begin{equation}\label{wle}
 |\langle D^{\alpha} u,  B(D^{\beta} v, D^{\gamma}u)\rangle| \leq
 \|u\|_{H^m} \|B(D^{\beta} v, D^{\gamma}u)\| \lqq{}  \|u\|_{H^m} \|v\|_{H^{k+2}}  \|u\|_{H^{m - k + 1}}
  \lqq{m} \|u\|_{H^m}^2 \|v\|_{H^m} \,.
 \end{equation}
On the other hand, since $m \geq 4$, for $k \in \{m - 1, m\}$ one has $k \geq 3$, and consequently by \eqref{libb} with $\ell = 0$
\begin{equation}\label{wlf}
 |\langle D^{\alpha} u,  B(D^{\beta} v, D^{\gamma}u)\rangle| \leq
 \|u\|_{H^m} \|B(D^{\beta} v, D^{\gamma}u)\|  \lqq{}  \|u\|_{H^m} \|v\|_{H^{k}}  \|u\|_{H^{m - k + 3}}
  \lqq{m} \|u\|_{H^m}^2 \|v\|_{H^m} \,.
\end{equation}
Overall, for any $m \geq 4$
\begin{equation}\label{sre}
|\langle \nabla^mu,\nabla^m B(v, u) \rangle| \lqq{m} \|u\|_{H^m}^2 \|v\|_{H^m} \,.
\end{equation}
Setting $v = u$, one has 
\begin{equation}
|I_1| \lqq{m} \|u\|_{H^m}^3 \,,
\end{equation}
and consequently 
\begin{align}
\dt Q_m(t)\lqq{m} Q_m(t)^{\frac{3}{2}}.\label{Ineq_3/2}
\end{align}
Finally, for any $\delta >0$,
\begin{align}
\dt(Q_m+\delta)\lqq{m} (Q_m+\delta)^\frac{3}{2},
\end{align}
and in particular
\begin{equation}
-\dt (Q_m+\delta)^\frac{-1}{2}\lqq{m} 1.
\end{equation}
Then,
\begin{align}
(Q_m(0)+\delta)^\frac{-1}{2}\leq C_mt+(Q_m(t)+\delta)^\frac{-1}{2}.
\end{align}
Passage $\delta\to 0$ and simple algebraic manipulations imply
\begin{align}
Q_m(t)^\frac{1}{2}\leq \frac{1}{Q_m(0)^\frac{-1}{2}-C_mt} \,,
\end{align}
or equivalently
\begin{align}
\|u(t)\|_{H^m}\leq \frac{1}{\|u_0\|_{H^m}^{-1}-C_mt}.\label{Time_unif}
\end{align}
Then, for fixed $R > 0$, there is $T = \frac{1}{2C_m R}$ such that for any 
$u_0 \in B_R(H^m)$,  we have that $\|u(t)\|_m$ is a priori bounded on the interval $[0, T)$
\begin{align}
\|u\|_{X^m_t}\leq 2 \|u_0\|_{H^m}.\label{Est_unif_nu}
\end{align}
By standard arguments, the solution of 
\eqref{NSequ} exists on $[0, T]$. 

\vspace{.2cm}
\noindent
\textbf{Passage to the inviscid limit.}
We focus only on the existence of solutions for $t \geq 0$. The results for negative times follow after the change of variables $t \to -t$, $u \to -u$, which preserves the equation 
\eqref{Euler(N)}.
Since we already established a lower bound on the  existence time and an upper bounds on the solutions that are independent of $\nu$,  we are ready to use compactness theorems to pass to the 
inviscid limit. 
Let  $(\nu_n)_n$ be a sequence of viscosity parameters  converging to $0$ and  let $u^n$ be the solution of \eqref{NSequ} with viscosity $\nu$ replaced by $\nu_n$
and $u_0 \in B_R(H^m) \cap \mathcal{E}_{N, m}$. 
Then, by \eqref{Est_unif_nu},   $u^n$ exists on the interval $(0, \tau)$ with  $\tau \geq \frac{C_{m}}{R}$ and by \eqref{Est_unif_nu}, for any $t \in (0, \tau)$
\begin{align}
\|u^n\|_{X^m_t}\leq 2 \|u_0\|_{H^m}.\label{Est_onSeq}
\end{align}
By \eqref{NSequ}, $m - 2 \geq 2$, $\nu_n \in (0, 1]$,  \eqref{upb}, and \eqref{Est_unif_nu}
\begin{equation}\label{rmb}
\begin{aligned}
\|\dt u^n \|_{H^{m - 2}} &\leq \nu_n \|\Delta u^n\|_{H^{m - 2}} + \|B_N(u^n, u^n)\|_{H^{m-2}}   \\
&\lqq{m} \|u^n\|_{H^{m}} + \|u^n\|_{H^{m - 1}}^2 \lqq{m} (1 + \|u^n_0\|_{H^{m}}^2) \,,
\end{aligned}
\end{equation}
and in particular $u^n \in W^{1, \infty}_t H^{m - 2}_x \hookrightarrow H^1_t H^{m - 2}_x$ with the norm bounded independently of $n$ and $N$. 


The estimates \eqref{Est_onSeq}, \eqref{rmb}, and Banach-Alaoglu  theorem ensure the existence of a subsequence $(u^{n_k})_{k \geq 1}$ that converges weak* in 
$L_t^\infty H_x^m$ and weakly in $H^1H^{m - 2}$ to some $u$.  Also, by using boundedness in $L_t^\infty H_x^m \cap H^1_t H^{m - 2}_x$ and compactness of embeddings
$L_t^\infty H_x^m \cap H^1_t H^{m - 2}_x \hookrightarrow L^2H^{m - 1}$, we can assume that $(u^{n_k})_{k \geq 1}$
converges strongly in $L_t^2H_x^{m-1}$ to $u$. Then, using a standard passage to the limit, we obtain that $u$ satisfies \eqref{Euler(N)}, \eqref{ID_Euler(N)}
 in the sense of distributions. 
Finally, using that norms are weak* lower semi-continuous, one obtains from  \eqref{Est_onSeq} that
\begin{align}
\|u\|_{X^m_\tau}\leq 2\|u_0\|_{H^m} \,, \label{Esti_Invisc}
\end{align}
and therefore $u$ is a strong solution of  \eqref{Euler(N)}, \eqref{ID_Euler(N)}. Also, by \eqref{Esti_Invisc}, the uniqueness is a standard consequence of energy estimates and 
Gronwall inequality.

\vspace{.2cm}
\noindent
\textbf{Continuity with respect to initial data.}
The quasilinearity of \eqref{Euler(N)} prevents us to use the usual continuity methods and we apply instead so-called Bona-Smith argument \cite{bonasmith}. 
Let $T = \frac{C_m}{R}$ be a lower bound on the existence time proved above and again we focus on positive times only. 

Fix $R > 0$ and $\tau \in [0, T]$. 
We claim that for any $u_0 \in  \mathcal{E}_{N, m}$ and $v_0 \in  \mathcal{E}_{N, m + 1}$ with $u_0, v_0 \in B_R(H^m)$ there holds
\begin{align}
\|u-v\|_{X^m_\tau}\lqq{m, R, \tau} \|u_0-v_0\|_{H^m} + \|u_0-v_0\|_{H^{m-1}}\|v_0\|_{H^{m+1}} \,,\label{Est_in_BSarg}
\end{align}
where $u$ and $v$ are solutions of \eqref{Euler(N)} with respective initial conditions $u_0$ and $v_0$. 
Indeed, for $w=u-v$, we have 
\begin{align}
\dt w + B_N(w,  v) + B_N(u, w)= 0.
\end{align}
As above, the functional $Q_{m-1}(t)=\frac{1}{2}\int_{\T^3}|\nabla^{m-1}w(x, t)|^2dx$ satisfies
\begin{align}
\dt Q_{m-1}(t) &=-\langle \nabla^{m-1}w,\nabla^{m-1} B(w,  v)\rangle - \langle\nabla^{m-1}w,\nabla^{m-1}B(u,  w)\rangle 
 = I_1 + I_2.
\end{align} 
First,  by \eqref{upb}
 \begin{align}
 |I_1| \lqq{m} \|w\|_{H^{m-1}} \|B(w, v)\|_{H^{m-1}} \lqq{m} \|w\|_{H^{m-1}}^2\|v\|_{H^m} = Q_{m-1}(w)\|v\|_{H^m}.
 \end{align}
 To estimate $I_2$ we argue as in \eqref{sre} with $m$ replaced by $m - 1$. Indeed,  if $k := |\beta| = m - 1$,  then $k \geq 3$ and  we proceed  as in \eqref{wlf}.  If $k \leq m - 3$,  we proceed
 as in \eqref{wle}.  Finally,  if $k = m - 2$,  then $k \geq 2$ and we use \eqref{libb} with $\ell = 1$ to obtain (recall $|\beta| = m - 2$ and $|\gamma| = 1$)
\begin{equation}
 |\langle D^{\alpha} w,  B(D^{\beta} u, D^{\gamma}w)\rangle| \leq
 \|w\|_{H^{m-1}} \|B(D^{\beta} u, D^{\gamma}w)\|  \lqq{}  \|w\|_{H^{m-1}} \|u\|_{H^{m-1}}  \|w\|_{H^{3}}
  \lqq{m} \|u\|_{H^m-1} \|w\|_{H^{m-1}}^2 \,.
\end{equation}
Thus, we obtain the bound
\begin{equation}
|I_2| \lqq{m} \|w\|^2_{H^{m - 1}} \|u\|_{H^{m}}  \lqq{m}  Q_{m-1}(w)\|u\|_{H^m}.
\end{equation} 
 Overall, 
 \begin{align}
 \dt Q_{m-1}(t)\lqq{m}\ Q_{m-1}(t)(\|u\|_{H^{m}}+\|v\|_{H^m})
 \end{align}
 and  Gronwall inequality with \eqref{Esti_Invisc} (for $u_0,\ v_0\in B_R(H^m)$) imply 
\begin{align}
\|w(t)\|_{X^{m-1}_\tau }\lqq{m, R} \|w(0)\|_{H^{m-1}}.\label{Est_m-1}
\end{align} 
Next, let us estimate $Q_m(t)=\frac{1}{2}\int_{\T^3}|\nabla^m w(x)|^2dx$. As above,
\begin{align}
\dt Q_m(t) &=-\langle\nabla^{m}w,\nabla^{m}B(w, v)\rangle - \langle \nabla^{m}w,\nabla^{m}B(u,  w)\rangle  
= J_1 + J_2
\end{align}
and from \eqref{sre} follows
\begin{equation}
|J_2| \lqq{m} \|w\|_{H^m}^2\|u\|_{H^m} = Q_m(w) \|u\|_{H^m}.
\end{equation}
Furthermore, since $m \geq 4$, then $H^2 \hookrightarrow H^{m - 1}$ and  \eqref{wle}, \eqref{wlf}, and \eqref{libb} yield
\begin{align}
|J_1| &\leq  \sum_{|\alpha| = m} |\langle D^\alpha w, B(w, D^\alpha v)\rangle| +
  \sum_{|\alpha| = m} \sum_{\substack{\beta + \gamma = \alpha\\ \beta \neq 0} } c_\beta \langle D^{\alpha} w,  B(D^{\beta} w, D^{\gamma}v)\rangle \\
  &\lqq{m} \|w\|_{H^m} \|B(w, D^\alpha v)\|  +  \|w\|_{H^m}^2 \|v\|_{H^m} \\
  &\lqq{m} \|w\|_{H^m} \|w\|_{H^{m-1}} \|v\|_{H^{m+1}}  +  \|w\|_{H^m}^2 \|v\|_{H^m}\\
  &\lqq{m} \|w\|_{H^m} \|w\|_{H^{m - 1}} \|v\|_{H^{m + 1}} + Q_m(w) \|v\|_{H^m} \,.
\end{align}
%
  Overall,
 \begin{align}
 \dt Q_m(t) \lqq{m} \ \sqrt{Q_m(t)}\|w\|_{H^{m-1}}\|v\|_{H^{m+1}} + Q_m(t)(\|u\|_{H^m} + \|v\|_{H^m}) \,.
 \end{align}
 Using \eqref{Esti_Invisc} and \eqref{Est_m-1}, we have for $u_0\in H^m$ and $v_0\in H^{m+1}$ and any $\delta > 0$ that 
 \begin{align}
 \dt\sqrt{Q_{m}(t)+\delta}\lqq{m} \|w_0\|_{H^{m-1}}\|v\|_{H^{m+1}}+\sqrt{Q_m(t)+\delta}(\|u_0\|_{H^m}+\|v_0\|_{H^m}).
 \end{align}
Since $u_0, v_0 \in B_R(H^m)$, from  the 
Gronwall inequality follows
\begin{align}
\frac{1}{\sqrt{2}} \|w(t)\|_{H^m} \leq \sqrt{Q_{m}(t) + \delta}  \lqq{m, t, R} \sqrt{\|w_0\|_{H^m}^2 + \delta}+ \|w_0\|_{H^{m-1}} \int_0^t \|v\|_{H^{m+1}} ds \,.
\end{align} 
Then, a passage $\delta\downarrow 0$ with \eqref{Esti_Invisc} yield
\begin{align}
 \|w(t)\|_{H^m} \lqq{m, t, R} \|w_0\|_{H^m}+ \|w_0\|_{H^{m-1}}\|v_0\|_{H^{m+1}},
\end{align} 
and \eqref{Est_in_BSarg} follows.

Following the arguments of Bona-Smith \cite{bonasmith}, let $(u_{0,n})  \subset \mathcal{E}_{N, m}$ be a sequence converging to $u_0 \in B_R(\mathcal{H}^m) \cap \mathcal{E}_{N, m}$ 
in the topology of $\mathcal{H}^m$, and 
$u_{0,n}^\epsilon \in \mathcal{H}^{m + 1} \cap \mathcal{E}_{N, m}$ be a regularisation of $u_{0, n}$ satisfying for each integer $k \geq 1$
\begin{equation}\label{rpu}
\begin{aligned}
\|u_{0,n}^\epsilon\|_{H^m} &\lleq\ \|u_{0,n}\|_{H^m},\\
\|u_{0,n}^\epsilon\|_{H^{m+k}} &\lleq\ \epsilon^{-k},\\
\|u_{0,n}^\epsilon-u_{0,n}\|_{H^m} &=o_\epsilon(1),\\
\|u_{0,n}^\epsilon-u_{0,n}\| &=o_\epsilon(\epsilon^m) \,,
\end{aligned}
\end{equation}
where $f_n \in o_\epsilon(\epsilon^k)$ means
\begin{equation}
\lim_{\epsilon \to 0} \sup_{n \geq 1} \left|\frac{f_n(\epsilon)}{\epsilon^k}\right| = 0 \,.
\end{equation}
Let $u_n$ and $u_n^\epsilon$ be solutions to \eqref{Euler(N)} with $u(0) = u_{0,n}$ and $u(0) = u_{0,n}^\epsilon$ respectively. By  \eqref{Est_in_BSarg} applied to 
$u = u_n$ and $v = u_n^\epsilon$, interpolation inequalities, and \eqref{rpu} one has
\begin{align}
\|u_n-u_n^\epsilon\|_{H^m} &\lqq{m, R, \tau} \ \|u_{0,n}-u_{0,n}^\epsilon\|_{H^m} + \|u_{0,n}-u_{0,n}^\epsilon\|_{H^{m-1}}\|u_{0,n}^\epsilon\|_{H^{m+1}}\\
&\lqq{m, R, \tau} \ \|u_{0,n}-u_{0,n}^\epsilon\|_{H^m} + \|u_{0,n}-u_{0,n}^\epsilon\|^{\frac{1}{m}}\|u_{0,n}-u_{0,n}^\epsilon\|_{H^m}^{1-\frac{1}{m}}\|u_{0,n}^\epsilon\|_{H^{m+1}}
\\
&\lqq{m, R, \tau}  o_\epsilon (1) + \epsilon(o_\epsilon (1))^{\frac{1}{m}} (o_\epsilon (1))^{1-\frac{1}{m}} \epsilon^{-1}\\
&\lqq{m, R, \tau}  o_\epsilon (1) \,.
\end{align}
Repeating the same argument with $u_{n, 0}$ and $u_{n, 0}^\epsilon$ replaced by $u_{ 0}$ and $u_{0}^\epsilon$ we have
\begin{align}
\|u-u^\epsilon\|_{H^m} &\lqq{m, R, \tau}  o_\epsilon(1) \,,
\end{align}
where $u^\epsilon$ and $u$ are  solutions of \eqref{Euler(N)} with $u(0) = u_0^\epsilon$ and $u(0) = u_0$ respectively.
Finally, using \eqref{Est_in_BSarg} for $u = u^\epsilon$ and $v = u_n^\epsilon$
\begin{align}
\|u^\epsilon-u_n^\epsilon\|_{H^m} &\lqq{m, R, \tau}  \|u_0^\epsilon - u_{0,n}^\epsilon\|_{H^m} +  \|u_0^\epsilon - u_{0,n}^\epsilon\|_{H^{m-1}} \| u_{0,n}^\epsilon\|_{H^{m + 1}} \,.
\end{align}
Moreover, by \eqref{rpu}
\begin{equation}
 \|u_0^\epsilon - u_{0,n}^\epsilon\|_{H^m} \leq  \|u_0^\epsilon - u_{0}\|_{H^m} +  \|u_0 - u_{0,n}\|_{H^m} +  \|u_{0, n} - u_{0,n}^\epsilon\|_{H^m}
 \leq o_\epsilon(1) + \|u_0 - u_{0,n}\|_{H^m} 
\end{equation}
and 
\begin{equation}
 \|u_0^\epsilon - u_{0,n}^\epsilon\|_{H^{m-1}} \leq  \|u_0^\epsilon - u_{0}\|_{H^{m-1}} +  \|u_0 - u_{0,n}\|_{H^{m-1}} +  \|u_{0, n} - u_{0,n}^\epsilon\|_{H^{m-1}}
 \leq o_\epsilon(\epsilon) + \|u_0 - u_{0,n}\|_{H^{m-1}}\,. 
\end{equation}
Hence, 
\begin{align}
\|u^\epsilon-u_n^\epsilon\|_{H^m} &\lqq{m, R, \tau} o_\epsilon(1) + \|u_0 - u_{0,n}\|_{H^m} + (o_\epsilon(\epsilon) + \|u_0- u_{0,n}\|_{H^{m-1}} ) \epsilon^{-1} \\
&\lqq{m, R, \tau}
o_\epsilon(1) + \|u_0 - u_{0,n}\|_{H^m} + \|u_0- u_{0,n}\|_{H^{m-1}}  \epsilon^{-1}
\,.
\end{align}
By choosing $\epsilon =  \|u_0- u_{0,n}\|_{H^{m-1}}^{\frac{1}{2}}$ we obtain $\epsilon \to 0$ as $n \to \infty$. Furthermore, since $o_\epsilon$ is uniform in $n$, 
\begin{equation}
\|u^\epsilon-u_n^\epsilon\|_{H^m} \lqq{m, R, \tau} o(1)   \qquad \text{as $n \to \infty$}\,.
\end{equation}
Finally,  by the triangle inequality
\begin{align}
\|u-u_n\|_{H^m} \leq \|u-u^\epsilon\|_{H^m}+\|u^\epsilon-u^\epsilon_n\|_{H^m} + \|u_n-u_n^\epsilon\|_{H^m} \lqq{m, R, \tau} o(1)   \qquad \text{as $n \to \infty$} ,
\end{align}
and we are done. 
\end{proof}

Recall that $\phi_t^N$ is the flow of the problem \eqref{Euler(N)}, \eqref{ID_Euler(N)}, and if $N = \infty$ we write $\phi_t$ instead of $\phi_t^\infty$. The following
lemma  and its corollary establish the convergence of Galerkin approximations to full solutions.
Since $N$ is varying we again indicate the dependence of $B$ and the solution $u$ on $N$. 
 Note that the convergence is in a weaker norm
than the assumed bound on initial conditions, resulting from the fact that 
there is no cancellation for the term $(I - \Pi_N)B(u, u)$ in the equation.  

\begin{nem}\label{Lem_Conv_H3}
Fix $R > 0$ and  $u_0\in B_R(H^4)$. Let $(u_{0,N})_N\subset B_R(H^4)$ be a sequence such that
\begin{enumerate}
\item $u_{0,N}\in \mathcal{E}_{N,4}$,
\item $\lim_{N\to\infty}\|u_0-u_{0,N}\|_{H^3}=0$,
\end{enumerate}
 then 
\begin{equation}
\lim_{N\to\infty}\|\phi_t u_0-\phi_t^Nu_{0,N}\|_{X^{3}_T}\to 0,
\end{equation}
where $T$ is as in Proposition \ref{UniformLWP}.
Moreover if we replace (2) by $\lim_{N\to\infty} \sup_{u_0\in B_R(H^4)} \|u_0-u_{0,N}\|_{H^3}=0$, then 
\begin{equation}
\lim_{N\to\infty}\sup_{u_0\in B_R(H^4)}  \|\phi_t u_0-\phi_t^Nu_{0,N}\|_{X^{3}_T}\to 0 \,.
\end{equation}
\end{nem}


\begin{Cor}\label{Cor_Conv}
Let $s\geq 4$, $r\in (3,s)$, $u_0\in B_R(H^s)$, and let $T$ be as in Propositon \ref{UniformLWP}. Let $(u_{0,N})_N\subset B_R(H^s)$ be a sequence such that
\begin{enumerate}
\item $u_{0,N}\in \mathcal{E}_{N,s}$,
\item $\lim_{N\to\infty}\|u_0-u_{0,N}\|_{H^3}=0$.
\end{enumerate}
Then,
\begin{equation}
\lim_{N\to\infty}\|\phi_t u_0-\phi_t^Nu_{0,N}\|_{X^{r}_T}\to 0.
\end{equation}
Moreover if we replace (2) by $\lim_{N\to\infty} \sup_{u_0\in B_R(H^s)} \|u_0-u_{0,N}\|_{H^3}=0$, then 
\begin{equation}
\lim_{N\to\infty}\sup_{u_0\in B_R(H^s)}  \|\phi_t u_0-\phi_t^Nu_{0,N}\|_{X^{r}_T}\to 0 \,.
\end{equation}
\end{Cor}

\begin{proof}[Proof of Corollary \ref{Cor_Conv}]
By interpolation and \eqref{Local_Growth_bound} we have 
\begin{align}
\|\phi_t u_0-\phi_t^Nu_{0,N}\|_{X^{r}_T} &\lqq{r, s} \|\phi_t u_0-\phi_t^Nu_{0,N}\|_{X^{3}_T}^{\frac{s-r}{s-3}}   \|\phi_t u_0-\phi_t^Nu_{0,N}\|_{X^{s}_T}^{\frac{r-3}{s-3}} \\
&\lqq{r, s} \|\phi_t u_0-\phi_t^Nu_{0,N}\|_{X^{3}_T}^{\frac{s-r}{s-3}}  ( \|\phi_t u_0\|_{X^{s}_T} + \| \phi_t^Nu_{0,N}\|_{X^{s}_T})^{\frac{r-3}{s-3}} \\
&\lqq{r, s, R} \|\phi_t u_0-\phi_t^Nu_{0,N}\|_{X^{3}_T}^{\frac{s-r}{s-3}}
\end{align}
and the assertions follow from Lemma \ref{Lem_Conv_H3}. 
\end{proof}

\begin{proof}[Proof of Lemma \ref{Lem_Conv_H3}]
If we set $u=\phi_tu_0$, $u_N=\phi^N_tu_{0,N}$, then $w_N=u-u_N$ satisfies
\begin{align}
\dt w_N &+ B_N(u, w_N)+ B_N(w_N, u_N)+ (I - \Pi_N)B(u, u)= 0 \,, \\
w_N(0) &=u_0-u_{0,N} \,.
\end{align} 
The functional $Q_3(t)=\frac{1}{2}\int_{\T^3}|\nabla^3 w_N(x, t)|^2dx$ satisfies
\begin{align*}
\frac{d}{dt}Q_3(t)&=\langle\nabla^3 w_N,\nabla^3\partial_t w_N \rangle \\
&=
- \langle \nabla^3 w_N,   \nabla^3 B_N(u, w_N) \rangle - \langle \nabla^3 w_N,   \nabla^3 B_N(w_N, u_N) \rangle 
- 
\langle \nabla^3 w_N,   \nabla^3 (I - \Pi_N)B(u, u) \rangle 
\\
&= I_1 + I_2 + I_3 \,.
\end{align*}
Since the image of $\Pi_N$ is spanned by (vectors of) sines and cosines, the projection 
$\Pi_N$ and gradient $\nabla$ commute.   Then, since the image of $B_N$ is a subset of $\mathcal{E}_{N, m}$, 
\begin{align}
I_1 =- \sum_{|\alpha| = 3} \langle D^\alpha \Pi_Nw_N, B_N(u, D^\alpha w_N) \rangle -
 \sum_{|\alpha| = 3} \sum_{\substack{\beta + \gamma = \alpha \\ \beta \neq 0}}
c_\beta \langle D^\alpha \Pi_Nw_N, B_N(\nabla^{\beta} u, D^\gamma w_N)\rangle = - J_1 - J_2.
\end{align}
In addition, using the definition of $B_N = \Pi_N B$, \eqref{cncl}, and \eqref{cnclg}
\begin{align}
J_1 &= \sum_{|\alpha| = 3}    \langle D^\alpha \Pi_Nw_N, \Pi_N B(u, D^\alpha \Pi_N w_N) \rangle + 
 \langle D^\alpha \Pi_Nw_N, \Pi_N B(u, D^\alpha (I-\Pi_N) w_N) \rangle
\\ 
&= \sum_{|\alpha| = 3} \langle D^\alpha \Pi_Nw_N, B(u, D^\alpha (I - \Pi_N) w_N) \rangle = - \sum_{|\alpha| = 3} \langle D^\alpha (I-\Pi_N) w_N, B(u, D^\alpha \Pi_N w_N) \rangle\,,
%
\end{align}
and therefore \eqref{libb}, Poincar\' e inequality,  and \eqref{Local_Growth_bound} imply
\begin{align}
|J_1|&\leq  \|(I - \Pi_N) w_N\|_{H^3} \|u\|_{H^2} \|\Pi_N w_N\|_{H^4}  \lleq \lambda_N^{-\frac{1}{2}}\|u\|_{H^2}\|w_N\|_{H^4}^2 
\\
&\lleq \lambda_N^{-\frac{1}{2}} \|u\|_{H^2} (\|u_N\|_{H^4}+\|u\|_{H^4})^2 \lleq \lambda_N^{-\frac{1}{2}}R^3.
\end{align}
To estimate $J_2$, we first treat the case $|\beta| = k \in  \{1,\ 3\}$ with the help of \eqref{libb} with $\ell = 2$ and $\ell = 0$ respectively 
%
\begin{equation}
|\langle D^\alpha \Pi_Nw_N, B_N(\nabla^{\beta} u, D^\gamma w_N)\rangle|  \leq \|w_N\|_{H^3} \|B(\nabla^{\beta} u, D^\gamma w_N)\| \lqq{}
\|w_N\|_{H^3}^2 \|u\|_{H^3} \,.
\end{equation}
If  $k=2$, then \eqref{libb} with $\ell = 1$ yields
\begin{equation}
|\langle D^\alpha \Pi_Nw_N, B_N(\nabla^{\beta} u, D^\gamma w_N)\rangle| \leq \|w_N\|_{H^3} \|B(\nabla^{\beta} u, D^\gamma w_N)\| \lqq{}
\|w_N\|_{H^3}^2 \|u\|_{H^3} \,,
\end{equation}
%
%
and overall 
\begin{align}
|I_1| \lleq \|w_N\|_{H^3}^2 \|u\|_{H^3} + \lambda_N^{-\frac{1}{2}}R^3.
\end{align}
Next, using \eqref{upb}
\begin{align}
|I_2| \leq \|w_N\|_{H^3 } \|B(w_N, u_N)\|_{H^3} \lleq \|w_N\|_{H^3} \|w_N\|_{H^3} \|u_N\|_{H^4} = \|w_N\|_{H^3}^2\|u_N\|_{H^4}.
\end{align}
Finally, from Poincar\' e inequality, \eqref{upb}, and \eqref{Local_Growth_bound} follows 
\begin{align}
|I_3| &\leq \|w_N\|_{H^4} \|(I - \Pi_N)B(u, u)\|_{H^2} \lleq  \lambda_N^{-\frac{1}{2}} (\|u_N\|_{H^4} + \|u\|_{H^4}) \|B(u, u)\|_{H^3} \\
&\lleq \lambda_N^{-\frac{1}{2}} R \|u\|_{H^3} \|u\|_{H^4} \lleq \lambda_N^{-\frac{1}{2}} R^3 \,.
\end{align}
Hence, we proved that 
\begin{equation}
\dt \|w_N(t)\|_{H^3}^2 \lleq \|w_N\|_{H^3}^2 (\|u\|_{H^3} + \|u_N\|_{H^4}) +  \lambda_N^{-\frac{1}{2}} R^3 \lleq R \|w_N\|_{H^3}^2 +  \lambda_N^{-\frac{1}{2}} R^3 \,,
\end{equation}
and consequently for any $t \in [-T, T]$
\begin{equation}
\|w_N(t)\|_{H^3}^2 \lqq{R, T} \|w_N(0)\|_{H^3}^2 + \lambda_N^{-\frac{1}{2}} \,.
\end{equation}
The first assertion of the lemma follows after passing $N \to \infty$, the second one follows after taking supremum with respect to $u_0$ and then passing $N \to \infty$. 
%
\end{proof}

\section{Fluctuation-Dissipation and Galerkin approximation}\label{sec:fdfd}

In this section, we introduce the fluctuation-dissipation method for \eqref{NSequ}. Specifically, we add an dissipative term with a 
small coefficient $\alpha$ which regularizes the dynamics and allows us to prove the global well posedness of solutions. In addition, to balance the energy, we introduce
a stochastic forcing multiplied by a small parameter $\sqrt{\alpha}$. This generic forcing presumably stirs the dynamics such that only observable solutions persist. 


Once the well-posedness of the stochastic flow is established, we prove the existence of invariant measures and obtain the moment bounds. The dependence of measures on 
$\alpha$ is crucial, since in the sections below we pass $\alpha \to 0$. 

Before we proceed,  let us recall our stochastic framework. 
Let $(\beta_m(t))$ be a sequence of independent real standard Brownian motions on a complete probability space $(\Omega,\mathcal{F},\P)$ and for an integer $N \geq 1$ define 
\begin{align}
\zeta(t,x) &=\sum_{m\in \Z^d}a_m e_m(x)\beta_m(t)\,,  \label{defnoise}\\
\zeta_N &= \Pi_N\zeta.
\end{align}
Recall that the functions $e_m = e^{\i mx}$ are eigenfuctions of Laplacian on $\T^d$, and therefore form an orthogonal basis of $L^2$. 
 To preserve zero mean condition, we always assume $a_0 = 0$. 
As above, 
we denote $\Pi_N$ the projection of $L^2$ to the finite dimensional space $E_N$ 
spanned by the 
eigenfunctions of $(-\Delta)$ corresponding to eigenvalues $\lambda$ with $\lambda \leq N$.
The sequence of vectors $a_m=(a_m^1, \cdots, a_m^d)$ is such that $|a_m|^2=|a_m^1|^2+ \cdots +|a_m^3|^2$ converges sufficiently fast to $0$ as $|m|\to \infty$ (specified below). We assume that $a_m^i\neq 0$ for all $m$ and $i$.
If the solution is vector valued, we in addition assume that $\sum_{\ell=1}^da_m^\ell m_\ell =0$ so that $\nabla \cdot \zeta = \nabla\cdot\zeta_N =0$. Note that this is assumption poses no restriction on the noise, since 
the gradient part of the noise
 can be absorbed into the pressure, which compensates diverge free condition. 
Let us denote by $(\mathcal{F}_t)_{t\geq 0}$ the natural filtration of the Brownian motion.

Define the numbers
\begin{gather}
A_{0,N} =\sum_{|m|\leq N}|a_m|^2, \qquad 
A_0 =\sum_{m }|a_m|^2 \,. \label{defsize}
\end{gather}
Throughout the rest of the paper, we assume that $A_{0, N} \leq A_0 < \infty$.

\subsection{Stochastic Global well posedness}


For any finite $N = \{1, 2, \cdots \}$ consider the stochastic 
problem 
\begin{align}
 du_N &=\left(-B_N[(u_N, u_N] - \alpha \mathcal{A} (u_N)
 \right)dt + \sqrt{\alpha}d\zeta_N\,, \label{Euler(N,alpha)}\\
u_N|_{t=0} &=u_0\in \mathcal{E}_N \,, \label{Euler(N,alpha,data)}
\end{align}
where for any smooth $u$,  the operator $\mathcal{A}$ satisfies for some $\sst \geq 4$ and each $u \in \mathcal{E}_N$
\begin{align}\label{coer}
\langle \mathcal{A}(u), u \rangle := \mathcal{G}(u) \geq \kappa_N \|u\|_{H^\sst}^4 \,.
\end{align}
Moreover, $\mathcal{G}$ satisfies for every $N \geq 1$
\begin{equation}\label{cntc}
\mathcal{G}(u) \geq \mathcal{G}(\Pi_N u)\qquad \textrm{and} \quad \mathcal{G}(\Pi_N u) \nearrow \mathcal{G}(u) \quad \textrm{as} \quad N \to \infty \,.
\end{equation}
We also require $\mathcal{A}$ to satisfy an upper bound, which although not too restrictive, is technical to formulate in full generality. For that purpose, we 
decided to proceed with a concrete operator $\mathcal{A}$.  Fix any $\sst \geq 4$, $\sstt > \sst$,  and $q \geq 2$ and define $\mathcal{A}$ as
\begin{equation}\label{aded}
\mathcal{A}(u)= e^{\rho(\|u\|_{H^\sst})}(a_1(-\Delta)^\sstt u+a_2\Delta(|\Delta u|^{q-2}\Delta u)-a_3\nabla\cdot(|\nabla u|^{2q-2}\nabla u)) \,,
\end{equation}
where $\rho : [0, \infty) \to [0, \infty)$ is a strictly increasing, convex function,  and $a_i$ are numbers in $\{0,1\}$,  not all equal to zero.  Unless otherwise specified, $a_i=1$. 
Let us remark that $\rho(x) \geq \delta x$ for some $\delta = \rho'(0) > 0$. 

To easier locate the arguments, where we used specific 
structure of $\mathcal{A}$, we use the phrase: `due to structural properties of $\mathcal{A}$'.
 Notice that for \eqref{aded} one has 
\begin{equation}\label{dfmcg}
\mathcal{G}(u) =  e^{\rho(\|u\|_{H^\sst})}\left( a_1\|u\|_{H^\sstt}^2+a_2\|u\|_{W^{2,q}}^q+a_3\|u\|_{W^{1,2q}}^{2q}\right) 
\end{equation}
and \eqref{coer} holds due to the equivalence of norms on finite dimensional space $\mathcal{E}_N$ and $\rho(x) \geq \delta x$. In addition, \eqref{cntc} is satisfied since $\rho$ is increasing and 
we work on the finite dimensional spaces spanned by eigenfunctions of the Laplacian. We remark that although the argument of $\rho$ contains $H^{\sst}$ norm, it does not play important role in 
this section, as we work in finite dimensional spaces. We keep this notation for easier comparison below. 
%

First, we establish the global well-posedness. 


\begin{prop}\label{Proposition_FDEuler}
Let $N<\infty$. The problem \eqref{Euler(N,alpha)}, \eqref{Euler(N,alpha,data)} is stochastically globally well-posed on $\mathcal{E}_{N, \sst}$. More precisely, 
for any $\mathcal{F}_{t > 0}$-independent random variable $u_0 \in \mathcal{E}_{N, \sst}$, there is  $\mathcal{F}_t$ adapted process $u_{\alpha, N}$ in $C_t\mathcal{E}_{N, \sst}$, that is almost surely unique and 
depends continuously on initial condition $u_0 \in \mathcal{E}_{N, \sst}$.   
\end{prop}

\begin{proof}[Sketch of  the proof] 
Fix $T > 0$. Since we are working in finite dimensional spaces, it suffices to consider only the $L^2$ norm of functions (all norms are equivalent and constants depend on $N$). 
We keep $H^\sst$ norm only for 
easier comparison below. First, note that if $(z_\alpha,v_\alpha)$ solves
\begin{align}
d z_\alpha &=\sqrt{\alpha} d\zeta_N,\quad z_\alpha(0)=0 \,, \label{LinEq}\\
\begin{split} \label{NonLinEq}
\dt v_\alpha + B_N(v_\alpha+z_\alpha, v_\alpha+z_\alpha) &= - \alpha \mathcal{A} (v_\alpha+z_\al), \\
  v_\alpha(0)&=u_0 \,, 
  \end{split}
\end{align}
then $u_{\alpha,N}=z_\al+v_\al$ solves \eqref{Euler(N,alpha)},  by linearity.
In addition, if $u_{\alpha, N}$ is vector valued, then since $\zeta_N$ is divergence free so is $z_\alpha$, and then $u_{\alpha,N}$ is divergence free as well.  
The solution of \eqref{LinEq} is given by 
\begin{align}
z_\al(t)=\sqrt{\alpha}\zeta_N(t) 
\end{align}
and by Burkholder-Davis-Gundy inequality and equivalence of norms, for any $k \geq 1, \ l\geq 1, s\geq 0$ there holds
\begin{align}
\E\sup_{t\in [0,T]}\|z_\al(t)\|_{W^{s,k}}^{2l}\lqq{T,k,A_0,l,N,s}\ \alpha^{p} .\label{Doob}
\end{align}

Denote 
$$
\mathcal{M}(u) :=  \|u\|_{H^\sstt}^2+ \|u\|_{W^{2,q}}^q+  \|u\|_{W^{1,2q}}^{2q} \,.
$$

To prove well-posedness of \eqref{LinEq} --  \eqref{NonLinEq}, we fix a realization of the process $z_\alpha$ and consider the corresponding deterministic equation \eqref{NonLinEq}. 
Since the local well posedness of \eqref{NonLinEq} is established by standard arguments, see for example techniques in Section \ref{SectionLWP} and 
\cite{Oksendal2000, syNLS7}, we only derive a priori $L^2$ bounds (which are equivalent to $H^\sst$ bounds in the finite dimensional space $\mathcal{E}_{N,\sst}$), 
that imply global well-posedness. 
Evaluating  $\dt\|v_\alpha\|^2$ 
and using  \eqref{NonLinEq} with \eqref{cncl}, we obtain
\begin{align}
\frac{1}{2}\dt\|v_\alpha\|^2 &=-\langle v_\al, B(v_\al+z_\al, v_\al+z_\al) \rangle   
 - \al \langle \mathcal{A} (v_\alpha+z_\al), v_\al\rangle
     \\
&=  \langle z_\al,  B(v_\al+z_\al, v_\al+z_\al)\rangle  
- \al  \mathcal{G} (v_\alpha+z_\al) -  \al\langle \mathcal{A} (v_\alpha+z_\al), z_\al\rangle
\,.
\end{align}
Since all norms are equivalent in finite dimensions, by \eqref{upb} and \eqref{coer}
\begin{align}
|\langle z_\al, B(v_\al+z_\al, v_\al+z_\al)\rangle| &\lqq{N} \|z_\alpha\| \|z_\alpha + v_\al\|^2 \leq C_{N, \kappa_N, \sst}  \frac{\|z_\alpha\|^2}{\alpha} 
+ \frac{\alpha \kappa_N}{2} \|z_\alpha + v_\al\|^4_{H^\sst} \\
&\leq C_{N, \rho, \sst} \frac{\|z_\alpha\|^2}{\alpha} + \frac{\alpha}{4}  \mathcal{G} (v_\alpha+z_\al) \,.
\end{align}
In addition, due to structural properties of $\mathcal{A}$,  integration by parts,  and H\" older inequality 
\begin{align}
|\langle \mathcal{A} (v_\alpha+z_\al),  z_\al \rangle| 
&\lesssim  
 e^{\rho(\|v_\al+z_\al\|_{H^\sst})} (\|v_\alpha+z_\al\|_{H^\sstt} \|z_\al\|_{H^\sstt}+ 
  \|v_\alpha+z_\alpha\|_{W^{2,q}}^{q-1} \|z_\alpha\|_{W^{2,q}} +
 \|v_\alpha+z_\alpha\|_{W^{1,2q}}^{2q-1}\|z_\alpha\|_{W^{1,2q}}) 
%
\\
&\leq  \frac{1}{4}  \mathcal{G} (v_\alpha+z_\al) +  C e^{\rho(\|v_\al+z_\al\|_{H^\sst})} \mathcal{M}(z_\al)
%
 \,.
\end{align}
Hence, using equivalence of norms in finite dimensions 
\begin{align}
\dt\|v_\alpha\|^2 &\leq \frac{C_{N, \rho, \sst}}{\al} \|z_\alpha\|^2 + \alpha
\left( C_{N, \sst, \sstt}^* \mathcal{M}(z_\al) -  \mathcal{M}(v_\al + z_\al) \right)
=: J (t).
\end{align}
From \eqref{Doob} follows that for $\P-$almost all $\omega\in \Omega$ there is $C(\omega,T)>0$ such that
\begin{align}
\sup_{t\in [0,T]}  \mathcal{M}(z_\al(t)) \leq C(\omega,T,\sstt,q).
\end{align}
Fix any such $\omega$. Then, 
for any fixed $t\in [0,T]$ we have either $\mathcal{M}(z_\al(t) + v_\al(t))  \leq C_{N,\sst, \sstt}^* C(\omega, T,\sst,q)$, and therefore 
\begin{align}
J(t) \leq 
 \frac{C_{N, \rho, \sst}}{\al} \|z_\alpha (t)\|^2 + 
C_{\rho, \sst, \omega, T} \leq C_{N, \rho, \sst, \sstt, \omega, T, \alpha}
\end{align}
or $\mathcal{M}(z_\al(t) + v_\al(t))\geq C_{N, \sst}^* C(\omega, T,\sst,q)$, which implies 
\begin{align}\label{sefv}
J(t) \leq  \frac{C_{N, \rho, \sst}}{\al} \|z_\alpha(t)\|^2 \leq C_{N, \rho, \sst, \sstt, \omega, T, \alpha}\,.
\end{align}
Thus, 
\begin{align}\label{csefv}
\dt\|v_\al\|^2\leq C_{N, \rho, \sst, \sstt, \omega, T, \alpha}
\end{align}
and in particular
\begin{align}
\sup_{t\in[0,T]}\|v_\al\|^2\leq \|u_0\|^2+C_{N, \rho, \sst, \sstt, \omega, T, \alpha}  \label{Control_v_alpha}
\end{align}
and the global existence of $v_\al$ and $u_{\alpha, N}$ follows from usual iteration.

Uniqueness and the continuity with respect to data follows standard arguments (see for example \cite{syNLS7}).
\end{proof}

\subsection{Stationary measures and moment bounds}
Denote $\text{Bor}(X)$ and $\mathfrak{p}(X)$ respectively the collection of Borel sets and  Borel probability measures on a Banach space $X$. 
For fixed $N \geq 1$ and $\sst \geq 0$ and $w \in \mathcal{E}_{N, \sst}$, let  $u_{\al,N}(t, w)$ be the solution of \eqref{Euler(N,alpha)} with $u_{\al,N}(0, w) = w$ 
(see Proposition \ref{Proposition_FDEuler} for well-posedness). 
We define the transition probability
\begin{align*}
T_{\al, t}^N(w,\Gamma)=\P(u_{\al,N}(t, w)\in\Gamma)\qquad \Gamma\in \text{Bor}(\mathcal{E}_{N, \sst}),\ \ t\geq 0,
\end{align*}
and  define the Markov semi-group $\mathfrak{P}_{\al,t}^{N} : L^\infty(\mathcal{E}_{N, \sst}, \R)\to L^\infty(\mathcal{E}_{N, \sst}, \R)$ 
and its dual $\mathfrak{P}_{\al, t}^{N*} : \mathfrak{p}(\mathcal{E}_{N, \sst})\to \mathfrak{p}(\mathcal{E}_{N, \sst})$ as
\begin{align*}
\mathfrak{P}_{\al, t}^{N}f(w)&= \E f(u_{\al,N}(t, w)) =  \int_{\mathcal{E}_{N, \sst}}f(w)T_{\al, t}^N (w,dv)\,,\\
\mathfrak{P}_{\al,t}^{N*}\lambda(\Gamma)&=\int_{\mathcal{E}_{N, \sst}}\lambda(dw)T_{\al,t}^N(w,\Gamma) \,.
\end{align*}
Since by Proposition \ref{Proposition_FDEuler} the solution $u_{\alpha,N}(t,u_0)$ is continuous in $u_0$, the Markov semi-group $\mathfrak{P}_{\al,t}^{N}$ is Feller, that is,  
for any $t\geq 0$, one has $ \mathfrak{P}_{\al,t}^{N}C_b(\mathcal{E}_{N, \sst})\subset C_b(\mathcal{E}_{N, \sst})$, where $C_b(X)$ denotes the space of bounded continuous functions on $X$. 

The following two propositions establish the existence of stationary measures together with their moment bounds. Since $\alpha$ is fixed, we omit the 
explicit dependence of $u$ on $\alpha$. 

\begin{prop}\label{pro:anee}
Fix an integer $N \geq 1$, $\alpha > 0$, and a smooth function $F : [0, \infty) \to \R$. If 
$u_0$ is a random variable on the space  $\mathcal{E}_{N, \sst}$ with $\E F(\|u_0\|^2) < \infty$ and $u_N$ is the solution 
of \eqref{Euler(N,alpha)} with $u_N(0) = u_0$, then 
\begin{multline}\label{gefs}
\E F(\|u_N(t)\|^2) + 2\alpha \E\int_0^t F'(\|u_N(\tau)\|^2) \mathcal{G}(u_N(\tau)) d\tau \\
= \E F(\| u_0\|^2) + 
\alpha A_{0,N}\E \int_0^t F'(\|u_N(\tau)\|^2) d\tau + 2\E \int_0^t F''(\|u_N(\tau)\|^2) \sum_m a_m^2 \langle u_N(\tau), e_m\rangle^2 d\tau   
\end{multline}
and in particular if $F(x) = x$, then 
\begin{equation}\label{sbfun}
\E \|u_N(t)\|^2 + 2\alpha \E\int_0^t \mathcal{G}(u_N(\tau)) d\tau = \E \| u_0\|^2 + 
\alpha A_{0,N}t  \,. 
\end{equation}
\end{prop}

\begin{proof}
Since by \eqref{cncl}
\begin{equation}\label{deve}
F'(\|u_N(t)\|^2)  \langle B(u_N, u_N), u_N\rangle = 0 \,,
\end{equation}
then, It\^o formula applied  to $F(\|u_N\|^2)$ imply (after omission of the argument $\|u_N\|^2$ of $F, F'$, and $F''$) 
\begin{equation}\label{wift}
dF + 2\al F'  \mathcal{G}(u_N) dt = \al \left(F' A_{0, N} + 2F'' \sum_m a_m^2 \langle u_N, e_m\rangle^2\right)dt + \sqrt{\al} d\mathcal{M} 
\end{equation}
where $\mathcal{M}$ is a martingale. An integration and taking the expectation yields \eqref{gefs} and in particular  \eqref{sbfun}.
\end{proof}

\begin{prop}\label{pro:eimn}
For any integer $N \geq 1$ and any $\alpha\in (0,1)$, the equation \eqref{Euler(N,alpha)} admits a stationary measure\, $\mu_{\alpha,N}$ satisfying the identity
\begin{align}
\int_{\mathcal{E}_{N, \sst}}\mathcal{G}(v) \mu_{\alpha,N}(dv) &=\frac{A_{0,N}}{2}\,. \label{IdentityH4_muNalpha}
\end{align}
Also, for any $m \geq 0$ one has 
\begin{equation}\label{mbfu}
\int_{\mathcal{E}_{N, \sst}} \|v\|^{2m}\mathcal{G}(v) \mu_{\alpha,N}(dv) \leq C\,,
\end{equation}
where $C$ depends only on $m, \sst, \rho$. 
\end{prop}

\begin{proof}
Since the proof follows standard lines, we only highlight crucial steps. Set the initial condition $u_0 \equiv 0$  and define measures
\begin{equation}
\mu_{\alpha, N}^t (\Gamma) = \frac{1}{t} \int_0^t T_{\alpha, \tau}^N(u_0, \Gamma) d\tau \,, \qquad \Gamma \in \textrm{Bor}(\mathcal{E}_{N, \sst}) \,.
\end{equation}
Then, for any $R > 0$,  interpolation, Chebyshev inequality, \eqref{coer}, and Proposition \ref{pro:anee} yield
\begin{align}
\mu_{\alpha, N}^t (\mathcal{E}_{N, \sst} \setminus B_R(H^\sst)) &= \frac{1}{t} \int_0^t \Prb(\|u_{\alpha, N}\|_{H^\sst} > R) d\tau  
\leq \frac{1}{t} \int_0^t \frac{\E \|u_{\alpha, N}\|^4_{H^\sst}}{R^4} d\tau \\
&\lqq{\rho}  \frac{1}{t R^4} \int_0^t \E \mathcal{G}(u_{\alpha, N}(\tau)) d\tau 
\lqq{\rho} \frac{A_{0, N}}{2R^4}  \,.
\end{align}
Since $B_R(\mathcal{H}^\sst) \cap \mathcal{E}_{N, \sst} $ is compact, we obtain that for fixed $\alpha > 0$ 
the collection of measures $(\mu_{\alpha, N}^t)_{t > 0}$ is tight in $\mathcal{E}_{N, \sst} $, and therefore by Prokhorov theorem it is compact. 
Let $\mu_{\alpha, N}$ be the limiting point as $t_n \to \infty$ and by  Krylov-Bogoliubov argument we obtain that $\mu_{\alpha, N}$ is an invariant measure of \eqref{Euler(N,alpha)}. 

Next, we derive  \eqref{IdentityH4_muNalpha} and \eqref{mbfu}.  For any $R > 0$, Proposition \ref{pro:anee} implies
\begin{align}
\int_{\mathcal{E}_{N, \sst} }\min\left\{ \mathcal{G}(v), R \right\}  \mu_{\alpha,N}(dv)  
&=
\lim_{n \to \infty} \int_{\mathcal{E}_{N, \sst}}\min\left\{ \mathcal{G}(v), R \right\} \mu_{\alpha,N}^{t_n}(dv) \leq 
\lim_{n \to \infty} \int_{\mathcal{E}_{N, \sst}} \mathcal{G}(v)  \mu_{\alpha,N}^{t_n}(dv) \leq 
  \frac{A_{0, N}}{2} 
\end{align}
and by the monotone convergence theorem for $R \to \infty$, one has
\begin{equation}\label{bbs}
\int_{\mathcal{E}_{N, \sst}} \mathcal{G}(v) \mu_{\alpha,N} (dv) \leq  \frac{A_{0, N}}{2}  \,.
\end{equation}
On the other hand, if the initial condition $u_0$ is distributed as stationary measure $\mu_{\alpha,N}$, then $\E \|u_0\|^2 < \infty$ by \eqref{bbs}. The 
corresponding solution $u_{\alpha, N}$ is stationary and by Proposition \ref{pro:anee} 
\begin{align}
\int_{\mathcal{E}_{N, \sst}} \mathcal{G}(v)  \mu_{\alpha,N} (dv) =
\E \mathcal{G}(u_{\alpha, N})
= 
\frac{1}{t} \E \int_0^t \|u_{\alpha, N}\|_{H^s}^2  \mathcal{G}(u_{\alpha, N})   d\tau = \frac{A_{0, N}}{2}
\end{align}
as desired. 

We show \eqref{mbfu} only for any $m \geq 2$, for $m \in (0, 2)$,  \eqref{mbfu}  follows from interpolation with $m = 0$. 
%
Similarly as above, set $u_0 \equiv 0$ and for any $R > 0$, Proposition \ref{pro:anee} with $F(x) = \frac{1}{m+1}x^{m+1}$ implies
\begin{multline}
\int_{\mathcal{E}_{N, \sst}}\min\left\{ \|v\|^{2m} \mathcal{G}(v), R \right\}  \mu_{\alpha,N}(dv)  
=
\lim_{n \to \infty} \int_{\mathcal{E}_{N, \sst}}\min\left\{ \|v\|^{2m} \mathcal{G}(v), R \right\} \mu_{\alpha,N}^{t_n}(dv) \\
\begin{aligned}
&\lqq{}
\lim_{n \to \infty} \int_{\mathcal{E}_N} \|v\|^{2m} \mathcal{G}(v)  \mu_{\alpha,N}^{t_n}(dv) \\
 &\lqq{m}  \lim_{n \to \infty} \frac{A_{0,N}}{2t_n} \E \int_0^{t_n} \|u_N\|^{2m} d\tau + \frac{2}{t_n}\E \int_0^{t_n} \|u_N\|^{2(m-1)} \sum_m a_m^2 \langle u_N, e_m\rangle^2 d\tau \,. 
 \end{aligned}
\end{multline}
However, 
\begin{equation}
 \sum_m a_m^2 \langle u_N, e_m\rangle^2 \leq A_{0, N} \|u_N\|^2 
\end{equation}
and by the definition of $\mathcal{G}$
\begin{equation}
\|u_N\|^{2m} \lqq{\sst}\|u_N\|^{2(m - 2)}_{H^\sst} \|u_N \|^2 \lqq{\sst, \rho}
 \mathcal{G}(v)  \,,
\end{equation}
and therefore by \eqref{IdentityH4_muNalpha} we have 
\begin{align}
\int_{\mathcal{E}_{N, \sst} }\min\left\{ \|v\|^{2m} \mathcal{G}(v), R \right\}  \mu_{\alpha,N}(dv) &\lqq{m, \sst, \rho} 
\int_{\mathcal{E}_{N, \sst} }\min\left\{  \mathcal{G}(v) , R \right\}  \mu_{\alpha,N}(dv) \leq C_{m, \sst, \rho}
\end{align}
and \eqref{mbfu} follows from the monotone convergence theorem after passing $R \to \infty$. 
\end{proof}

Let $\chi : [0, \infty) \to [0, \infty)$ be a $C^\infty$ function supported on $[0, 2]$ and $\chi(x) = 1$ for $x \in [0, 1]$. Let $C_{\chi}$ be a bound on the first two derivatives 
of $\chi$. Define $\chi_R(x) = \chi\left( \frac{x}{R}\right)$. In the next lemma we estimate the tail distribution of $\mu_{\alpha, N}$, for a proof close to our setting see 
\cite[Proposition 3.8]{syNLS7}. 

\begin{nem}\label{lem:tbf}
For any $R > 1$ one has 
\begin{equation}
\int_{\mathcal{E}_{N, \sst}} \mathcal{G}(v) (1 - \chi_R(\|v\|^2)) \mu_{\alpha,N}(dv) \lqq{m, \sst, \rho, \chi} \frac{1}{R} \,.
\end{equation}
\end{nem}

\begin{proof}
Since the proof is rather standard, let us only highlight  key steps. 
If $u_{\alpha, N}$ is a solution of \eqref{Euler(N,alpha)} with initial condition distributed as $\mu_{\alpha, N}$, then \eqref{cncl},  \eqref{deve},  and 
It\^ o formula applied to $F_R(u_{\alpha, N}) = \|u_{\alpha, N}\|^2 (1 - \chi_R(\|u_{\alpha, N}\|^2))$ yield (with omission of the argument $\|u_{\alpha, N}\|^2$ of $\chi_R$)
\begin{multline}
\frac{1}{2} d F_R(u_{\alpha, N}) + \alpha \mathcal{G}(u_{\alpha, N}) (1 - \chi_R ) dt =
  \alpha\chi_R' \|u_{\alpha, N}\|^2\mathcal{G}(u_{\alpha, N})  dt \\
  + \frac{\alpha}{2} \left( (1 - \chi_R) A_{0, N} - 2 \chi_R' \sum_{m} a_m^2 \langle u_{\alpha, N}, e_m\rangle^2 - 
  \|u_{\alpha, N}\|^2 \left[ A_{0, N} \chi'_R + \chi''_R \sum_{m} a_m^2 \langle u_{\alpha, N}, e_m\rangle^2	\right]  \right) dt + d\mathcal{M} \,,
\end{multline}
where  $\mathcal{M}$ is a martingale. Taking an expectation and using the stationarity 
of $u_{\alpha, N}$ we have 
\begin{multline}
\E \mathcal{G}(u_{\alpha, N})  (1 - \chi_R ) =
 \E \chi_R' \|u_{\alpha, N}\|^2  \mathcal{G}(u_{\alpha, N}) 
\\
+ \frac{1}{2} \E \left( (1 - \chi_R) A_{0, N} - 2 \chi_R' \sum_{m} a_m^2 \langle u_{\alpha, N}, e_m\rangle^2 
- 
  \|u_{\alpha, N}\|^2 \left[ A_{0, N} \chi'_R + \chi''_R \sum_{m} a_m^2 \langle u_{\alpha, N}, e_m\rangle^2	\right]  \right) \,.
\end{multline}
Then, from \eqref{mbfu} with $m = 1$, $\chi_R' \lqq{\chi} \frac{1}{R}$, and $\chi_R'' \lqq{\chi} \frac{1}{R^2}$ follows
\begin{equation}
\E \mathcal{G}(u_{\alpha, N})   (1 - \chi_R(\|u_{\alpha, N}\|^2) )  \leq \frac{A_{0, N}}{2}\E (1 - \chi_R(\|u_{\alpha, N}\|^2)) + \frac{C}{R} \,. 
\end{equation}
Finally, Markov inequality and \eqref{IdentityH4_muNalpha} yield 
\begin{equation}
\E (1 - \chi_R(\|u_{\alpha, N}\|^2)) \leq  \Prb(\|u_{\alpha, N}\|^2 \geq R) \leq \frac{\E \|u_{\alpha, N}\|^2}{R} \leq  \frac{C}{R}
\end{equation}
and we showed the required assertion.
\end{proof}

\section{Inviscid limit}\label{secl:invlim}
In the section we use the invariant measures $\mu_{\alpha, N}$ constructed in Section \ref{sec:fdfd} and we pass $\alpha \to 0$ to obtain an
invariant measure for finite dimensional projections  of \eqref{Euler1}.  It is crucial to keep track of dependencies of constants on $N$, since in the sections below we pass 
$N \to \infty$.

Recall that by Proposition \ref{UniformLWP} and Corollary \ref{Cor-GWP-EulerN}, the Galerkin approximation of the Euler equation \eqref{Euler(N)} (with $N<\infty$) 
admits a global flow which we denote $\phi^N_t$.  For each integer $N \geq 1$ and $t \in \R$,  let $\Phi^N_t :  L^2 \to L^2$ be the associated Markov semi-group
and $\Phi^{N*}_t : \mathfrak{p}(L^2)\to \mathfrak{p}(L^2)$  be its dual, 
defined respectively by
\begin{align}
\Phi^N_t f(v) = f(\phi^N_t\Pi_Nv), \qquad 
\Phi^{N*}_t\nu(\Gamma) =\nu(\phi^N_{-t}\Pi_N\Gamma) \,.
\end{align}
Since by Proposition \ref{UniformLWP}, $u_0 \mapsto \phi^N_t(u_0)$ is continuous for any small $t$,  then $\Phi^N_t$ is Feller. 
The main result of this section is as follows.

\begin{prop}\label{pro:cimn}
For any integer $N \geq 1$, there is an invariant measure $\mu_N$ for the flow $\phi^N_t$ such that 
\begin{align}
\int_{\mathcal{E}_{N, \sst}} \mathcal{G}(v) \mu_N(dv) &= \frac{A_{0,N}}{2} \label{EstmuN}
\end{align}
and for any $m \geq 0$
\begin{align}\label{hosog}
\int_{\mathcal{E}_{N, \sst}} \|v\|^{2m} \mathcal{G}(v) \mu_N(dv) \leq C_m \,.
\end{align}

\end{prop}

\begin{proof}
Let $\mu_{\alpha, N}$ be as in Proposition \ref{pro:eimn}. 
First, we claim that the set $(\mu_{\alpha,N})_{\alpha \in (0, 1)} $ is tight for any fixed $N$. Indeed, by Chebyshev inequality,  \eqref{coer}, and
\eqref{IdentityH4_muNalpha}
\begin{align}\label{toman}
\mu_{\alpha, N} (\mathcal{E}_{N, \sst} \setminus B_R(H^{\sst})) &\leq  \frac{1}{R^4} \int_{\mathcal{E}_{N, \sst}} \|v\|_{H^{\sst}}^4  \mu_{\alpha, N}(dv)  
 \lqq{\sst, \rho} \frac{1}{R^4} \int_{\mathcal{E}_{N, \sst}} \mathcal{G}(v) \mu_{\alpha, N}(dv) \lqq{\sst, \rho} \frac{A_{0, N}}{R^4}
\end{align}
and a tightness follows. Therefore, by Prokhorov theorem there is a sequence $\alpha_k \to 0$ as $k \to \infty$ such that 
$\mu_{k, N} := \mu_{\alpha_k, N}$ weakly converges to a measure $\mu_N$. 

Next,  Lemma \ref{lem:tbf} and Proposition \ref{pro:eimn} give us
\begin{align}
 \frac{A_{0,N}}{2} - \frac{C}{R} &\leq  
 \frac{A_{0,N}}{2} - \int_{\mathcal{E}_{N, \sst}}\mathcal{G}(v) (1 - \chi_R(\|v\|^2))\mu_{k, N}(dv) = \int_{\mathcal{E}_{N, \sst}}\mathcal{G}(v)  \chi_R(\|v\|^2)\mu_{k, N}(dv) \leq \frac{A_{0,N}}{2} \,.
\end{align}
Passing $k \to \infty$, using that all norms on finite dimensional space are equivalent, and the weak convergence $\mu_{k, N} \to \mu_N$, we obtain 
\begin{align}
\int_{\mathcal{E}_{N, \sst}}\mathcal{G}(v)  \chi_R(\|v\|^2)\mu_{N}(dv) 
= 
\lim_{k\to \infty} \int_{\mathcal{E}_{N, \sst}}\mathcal{G}(v)\chi_R(\|v\|^2)\mu_{k, N}(dv)
 \leq 
  \frac{A_{0,N}}{2} \,,
\end{align}
and by the monotone convergence theorem for $R \to \infty$, one has
\begin{equation}
\int_{\mathcal{E}_{N, \sst}} \mathcal{G}(v) \mu_{N}(dv)  \leq \frac{A_{0,N}}{2} \,.
\end{equation}
On the other hand, 
\begin{align}
\int_{\mathcal{E}_{N, \sst}} \mathcal{G}(v)  \mu_{N}(dv) &\geq 
\int_{\mathcal{E}_{N, \sst}} \mathcal{G}(v)  \chi_R(\|v\|^2)\mu_{N}(dv) = 
\liminf_{k \to \infty}
\int_{\mathcal{E}_{N, \sst}}\mathcal{G}(v)  \chi_R(\|v\|^2)\mu_{k, N}(dv)\\
&\geq \frac{A_0}{2} -\frac{C}{R} \,,
\end{align}
and since $R$ was arbitrary,  \eqref{EstmuN} follows. 

Also,  for any $R > 0$ we have by \eqref{mbfu}
\begin{align*}
\int_{\mathcal{E}_{N, \sst}} \chi_R(\|v\|^2)  \|v\|^{2m} \mathcal{G}(v) \mu_{k,N}(dv) \leq C_m
\end{align*}
and \eqref{hosog} follows if we first pass $k \to \infty$ and then use the monotone convergence theorem while passing $R \to \infty$.

 To show that $\mu_N$ is invariant under the flow $(\phi^N_t)_{t \in \R}$, it suffices to show the invariance for positive times. Indeed,
for $t<0$ the invariance for positive times implies 
\begin{align*}
\mu_N(\Gamma)=\mu_N(\phi^N_{t}\Gamma)=\mu_N(\phi^N_{2t}\phi^N_{-t}\Gamma)=\mu_N(\phi^N_{-t}\Gamma),
\end{align*}
as desired. The strategy of the proof of the invariance for positive times is summarized in the following diagram
$$
\hspace{10mm}
\xymatrix{
  \mathfrak{P}_{k, t}^{N*}\mu_{k,N} \ar@{=}[r]^{(I)} \ar[d]^{(III)} & \mu_{k,N} \ar[d]^{(II)} \\
    \Phi_t^{N*}\mu_N \ar@{=}[r]^{(IV)} & \mu_N
  }
$$
where, to avoid multiple indices, we denoted $ \mathfrak{P}_{k, t}^{N*} =  \mathfrak{P}_{\alpha_k,t}^{N*}$.
The equality $(I)$ follows from the invariance  of $\mu_{k,N}$ under $\mathfrak{P}_{k,t}^{N*}$, $(II)$ is the weak convergence of $\mu_{k,N} \to \mu_N$ as $k \to \infty$. 
Thus, we prove $(III)$ to establish  $(IV)$, the claimed invariance of $\mu_N$ under $(\phi_t^N)_{t \geq 0}$.  Note that is suffices to prove 
(IV), and therefore (III) for any $t \in [0, 1]$,  larger times follow by iteration. 

Fix any $t \in [0, 1]$. By Portmanteau theorem it suffices to prove 
\begin{equation}
\lim_{k \to \infty } (\mathfrak{P}_{k,t}^{N*}\mu_{k,N},f) - (\Phi_t^{N*}\mu_N, f)  = 0
\end{equation}
for any bounded Lipschitz function $f: \mathcal{E}_{N, \sst} \to\R$, where $(\cdot, \cdot)$ is a dual pairing of a measure and a continuous function. 
By linearity, we can without loss of generality assume that $f$ is bounded by 1. Then, 
\begin{align*}
 (\mathfrak{P}_{k,t}^{N*}\mu_{k,N},f) - ( \Phi_t^{N*} \mu_N,f) 
&=  ( \mu_{k,N},\mathfrak{P}_{k,t}^N f) - ( \mu_N,\Phi_t^N f) \\
&= ( \mu_{k,N},\mathfrak{P}_{k,t}^N f-\Phi_t^N f) - ( \mu_N-\mu_{k,N},\Phi_t^N f) \\
&=  A - B.
\end{align*}
Since $\Phi_t^N$ is Feller,  $\Phi_t^N f$ is bounded and continuous and by the weak convergence $\mu_{k, N} \to \mu_N$  we have $B\to 0$ as $k\to\infty$. 
Since $|f| \leq 1$, we trivially have $|\Phi_t^N f| \leq 1$ and $|\mathfrak{P}_{k,t}^N f| \leq 1$, and consequently 
\begin{align*}
|A|\leq \int_{B_R(\mathcal{E}_{N, \sst})} |\Phi^N_t f(u)-\mathfrak{P}_{k,t}^N f(u)|\mu_{k,N}(du)+2\mu_{k,N}(\mathcal{E}_N\backslash B_R(\mathcal{E}_N))=:A_1+A_2.
\end{align*}
From \eqref{toman} follows
\begin{align*}
A_2\lqq{\sst, \rho} \frac{1}{R^4}.
\end{align*}
To treat $A_1$, let $u_{\alpha, N} (t, u_0)$ be the solution to \eqref{Euler(N,alpha)} at time $t$ with $u_{k, N}(0) = u_0 \in B_R(\mathcal{E}_{N, \sst})$.  
For each $r > 0$ and $\alpha > 0$ define
the set (recall $z_\alpha$ was defined by \eqref{LinEq})
\begin{equation}
S_r^\alpha = \left\{\omega\in\Omega : \ \sup_{t \in [0, 1]} \left|\sqrt{\al}\sum_{|m|\leq N} a_m\int_0^t \langle u_{\alpha, N}, e_m\rangle d\beta_m\right|\vee \|z_\alpha(t)\|\leq r\sqrt{\al}  \right\} 
\end{equation}
and $S_r^k = S_r^{\alpha_k}$. 

The proof of the following convergence result is postponed till the end of the proof of our proposition. 

\begin{nem}\label{LemRemoveViscosity}
For any $R>0,$ $r> 0$, and any $t \in [0, 1]$
\begin{align}
\lim_{k \to \infty} \sup_{u_0 \in B_R(\mathcal{E}_{N, \sst})} \E(\|\phi^N_t(u_0)- u_{k, N}(t, u_0)\|\mathbb{1}_{S_r^k}) = 0 \,,
\end{align}
where as above $\phi^N_t$ is the flow of \eqref{Euler(N)}.
\end{nem}

Let us continue with the proof of the proposition. 
Denote  $u_{k, N} = u_{\alpha_k,N}$ and let $C_f$ be the Lipschitz constant of $f$. Using again $|f| \leq 1$, 
we obtain
\begin{align*}
A_1 &= \int_{B_R} | f(\phi^N_t(u)) - \E f(u_{k, N}(t, u))| \mu_{k,N}(du) 
\leq 
 \int_{B_R} \E  \min\{2,  C_f\|\phi^N_t(u) -  u_{k, N}(t, u)\| \} \mu_{k,N}(du)
\\
&\leq \int_{B_R}  \E\left[ \min\{2, C_f \|\phi^N_t(u) -  u_{k, N}(t, u)\|\}  \mathbb{1}_{S_r^k}\right] \mu_{k,N}(du) + 2 \int_{B_R}\E(\mathbb{1}_{(S_r^k)^c})\mu_{k,N}(du) =:A_{1,1}+A_{1,2}.
\end{align*}
Then, Lemma \ref{LemRemoveViscosity} and the Lebesgue dominated converge theorem imply  $A_{1, 1} = o_{r, R}(1)$ as $k \to \infty$. 

Furthermore, it follows from the It\^o isometry, $t \in [0, 1]$, and \eqref{sbfun} that
\begin{align*}
\E\left|\sqrt{\al_k}\sum_{|m|\leq N} a_m\int_0^t\langle u_{k, N}, e_m \rangle d\beta_m\right|^2&=\al_k\E \sum_{|m|\leq N} a^2_m\int_0^t\langle u_{k, N},e_m\rangle ^2dt\leq \al_k A_0\sup_{t \in [0, 1]}\E\|u_{k, N}(t)\|^2
\\
&\leq \al_k A_0\sup_{t \in [0, 1]}\E\|u_0\|^2 + \al_k^2 A_0
\leq C_R \al_k,
\end{align*}
where $C$ does not depend on $k.$ Also, from \eqref{Doob},
\begin{align*}
\E\sup_{t\in [0, T]}\|z_k (t)\|^2\leq C\al_k,
\end{align*}
where $C$ is independent of $k.$ Therefore,  the Chebyshev inequality implies
\begin{align*}
\E(\mathbb{1}_{(S^k_r)^c})=\P\left\{\omega\ : \  \max\left(\left|\sqrt{\al_k}\sum_{|m|\leq N} a_m\int_0^t \langle u_{k, N}, e_m\rangle d\beta_m\right|, \|z_k\|\right)\geq r\sqrt{\al_k} \right\}\leq 
\frac{C_R\al_k}{r^2\al_k}=\frac{C_R}{r^2}.
\end{align*}
Overall, 
\begin{equation}
|A| \leq |A_{1, 1}| + |A_{1, 2}| + |A_2| \leq o_{R,r}(1) + \frac{C_R}{r^2} + \frac{C}{R^2} \,,
\end{equation}
and therefore by passing first $k\to\infty$,  then $r\to\infty$, and finally $R\to\infty$ we obtain $(III)$, and hence $(IV)$.
\end{proof}

\begin{proof}[Proof of Lemma \ref{LemRemoveViscosity}]
Fix a sequence $(\alpha_k) \subset (0, 1]$ such that $\alpha_k \to 0$ as $k \to \infty$. 
For fixed $u_0 \in \mathcal{E}_{N, \sst}$ with $\|u_0\| \leq R$, 
set $w_k=u-v_k:=\phi_N^t  u_0-v_k(t,  u_0)$, 
where $v_k(t, u_0)$ is the solution of \eqref{NonLinEq} with $\alpha =\al_k$ 
and $v_k(0, u_0) = u _0$. Recall that $u_{k, N}= v_k + z_k$ solves \eqref{Euler(N,alpha)},  where $z_k$ solves the problem \eqref{LinEq} with $\al=\al_k.$ Due to \eqref{Doob}, we have 
$\E\sup_{t \in [0, 1]} \|z_k\|\to 0$ as $k\to\infty$,
and therefore to finish the proof, it suffices to show that
\begin{align}\label{ldct}
\E (\|w_k\|\mathbb{1}_{S_r^k})\to 0,\quad as\ k\to\infty.
\end{align}
First, by the It\^o formula for $\|u_{k, N}\|^2$ (see \eqref{wift}), we have
\begin{equation}
\|u_{k, N}(t)\|^2+2\al_k\int_0^t\mathcal{G}(u_{k, N}) d\tau 
=\|u_0\|^2 + \al_k \frac{A_{0,N}}{2}t+2\sqrt{\al_k}\sum_{|m|\leq N} a_m\int_0^t\langle u_{k, N},e_m\rangle d\beta_m \,,
\end{equation}
and on  $S_r^k$ for any $t \in [0, 1]$ we obtain
\begin{align}\label{bonn}
\|u_{k, N}(t)\|^2 &\leq \|u_0\|^2 + C(r,N) \leq C(r, N, R)\,, \\ \label{bovk}
\|v_k(t)\| &\leq \|u_{k, N}(t)\| +  \|z_k\| \leq C(r, N, R) \,,
\end{align}
where $C(r,N, R)$ does not depend on $k$. Hence, on $S_r^k$
\begin{align}
\|w_k\|\leq \|v_k\|+\|z_k\|\leq \|u_{k, N} \| + 2\|z_k\| \leq  C(r,N, R) \,,
\end{align}
and in particular $\|w_k\|$ is bounded, uniformly in $k$ and $u_0 \in B_R$. 
Thus, by the Lebesgue dominated convergence theorem it suffices to prove  $w_k \to 0$ as $k \to \infty$ almost surely, to establish 
\eqref{ldct}. 
Subtract the equations \eqref{Euler(N)} and $\eqref{NonLinEq}$ and obtain $w_k(0) = 0$ and
\begin{align*}
\dt w_k&= -B_N(w_k, u)- B_N(v_k, w_k) + B_N(v_k, z_k) + B_N(z_k, v_k) + B_N(z_k, z_k)  -\alpha_k \mathcal{A} (v_k+z_k). 
\end{align*}
Testing with $w_k$ and using \eqref{cncl}, \eqref{upb},  Young inequality, equivalence of norms in finite dimensions, \eqref{bonn}, \eqref{bovk}, and 
structural properties of $\mathcal{A}$, we obtain (recalling that $q>2$)
\begin{align*}
\dt\|w_k\|^2 &\lqq{N, \sst} \, \|w_k\|^2 (\|u\| +  \|v_k\|) +\|z_k\| \|w_k\| (\|v_k\| + \|z_k\|)+\al_k e^{\rho(\|u_{k, N}\|_{H^\sst})} \|w_k\| (1+\|v_k\|^{2q - 1}+\|z_k\|^{2q - 1})\\
&\lqq{N, \sst}  \|w_k\|^2(1+\|u\|_{L^\infty_tL^2_x} +\|v_k\|_{L^\infty_tL^2_x})+ \|z_k\|_{L^\infty_tL^2_x}^2(\|z_k\|_{L^\infty_tL^2_x}^2 + \|v_k\|_{L^\infty_tL^2_x}^2)
 \\
 &\qquad + \al_k^2 e^{2\sup_{t}\rho(\|u_{k, N}\|_{H^\sst})} (1+\|v_k\|_{L^\infty_tL^2_x}^{4q}+\|z_k\|_{L^\infty_tL^2_x}^{4q})\\
 &\lqq{r, N, R,\sst}  \|w_k\|^2 + \|z_k\|_{L^\infty_tL^2_x}^2   + \al_k^2.
\end{align*} 
Using the Gronwall lemma, the fact that $w_k(0)=0$, we have for any $t \in [0, 1]$ 
\begin{align}
\|w_k(t)\|^2 &\lqq{r, N, R, \sst} e^{t(1+\|u\|_{L^\infty_tL^2_x})}(\alpha_k^2 + \|z_k\|_{L^\infty_tL^2_x}^2) 
\label{ConvGENERALEst} \,.
\end{align}
Notice that $u$ is a solution of deterministic equation in finite dimensions, and therefore from Corollary \ref{Cor-GWP-EulerN} follows 
$\|u\|_{L^\infty_tL^2_x} \leq C_N$. In addition on $S_r^k$ one has $\|z_k\|_{L^\infty_tL^2_x}^2 \leq r \alpha_k$, and therefore for any $t \in [0, 1]$ 
\begin{equation}
\lim_{k \to \infty} \sup_{u_0 \in B_R} \|w_k(t)\|^2 \mathbb{1}_{S_r^k} =  0 \qquad \textrm{a. s.}
\end{equation}
as desired. 
\end{proof}

\section{Statistical ensemble and global flow}\label{SE_globalflow}
\subsection{General context}
 Recall that $\sst \geq 4$ and $\sstt \geq \sst$ were fixed in the definition of the operator $\mathcal{A}$ in Section \ref{sec:fdfd}. 
 Consider measures $\mu_N$ defined in Proposition \ref{pro:cimn} and we view them as measures on  a common space $H^{\sst}$ with Borel 
$\sigma$-algebra.
In this section we establish the existence of limiting measure for the full  equation \eqref{Euler1} by passing the dimension of the Galerkin approximation to infinity.

\begin{prop}\label{pro:domm}
For each integer $N \geq 1$ let $\mu_N$  be
 considered as a measure on $\mathcal{H}^\sstt$. Then, 
there is a subsequence of $(\mu_N)_{N \geq 1}$ converging weakly (in $L^2$ topology) to a measure $\mu$ supported on $\mathcal{H}^\sstt$ 
such that
\begin{align}
\int_{L^2} \mathcal{G}(v)\mu(dv)&\leq \frac{A_0}{2}.\label{EstimateH4_mu}
\end{align}
\end{prop}


\begin{proof}
Since $\mu_N$ satisfies \eqref{EstmuN},  then Chebyshev inequality and \eqref{coer} imply (cf. proof 
of Proposition \ref{pro:cimn})
\begin{align}
\mu_N(L^2\setminus B_R(\mathcal{H}^\sstt)) &\leq \frac{1}{R^4} \int_{L^2} \|v\|_{H^{\sstt}}^4 \mu_N (dv) \lleq_{\sstt, \rho} 
\frac{1}{R^4}\int_{L^2} \mathcal{G}(v) \mu_N (dv) \lleq_{\sst, \rho} \frac{A_0}{R^4}, 
\end{align}

and the 
the tightness follows since $B_R(\mathcal{H}^\sstt)$ is compact in $L^2$.  Then, the
existence of weakly convergent subsequence and the limiting measure $\mu$ follow from the Prokhorov theorem. 

To obtain
 \eqref{EstimateH4_mu}, fix any integer $M \geq 1$ and $R > 0$. Let $\chi_R(x) := \chi(x/R)$ for each $x \geq 0$, where  $\chi: [0, \infty) \to [0, 1]$ is a non-increasing smooth cut-off function supported on $[0, 2]$ with 
 $\chi(z) = 1$
 for $z \in [0, 1]$. Then, by \eqref{cntc} and Proposition \ref{pro:cimn} one has for any $M, N \geq 1$ that
 \begin{equation}
  \int_{L^2} \mathcal{G} (\Pi_M v) \chi_R(\|v\|) \mu_N (dv) \leq
 \int_{L^2} \mathcal{G} (\Pi_M v) \chi_R(\|v\|) \mu_N (dv)\leq  \frac{A_0}{2} \,,
 \end{equation}
 where we recall that $\Pi_M$ is the projection on space spanned by the first $M$ eigenfunctions of the Laplacian. 
 Since the norms $v\mapsto \|\Pi_M v\|_{H^\sst}$ and $v\mapsto \|\Pi_M v\|$ are equivalent on $\mathcal{E}_{N, \sst}$, 
 the integrand is a bounded  and continuous on $L^2$ by the structural properties of $\mathcal{A}$, and therefore we can pass $N \to \infty$ to obtain
 \begin{equation}
  \int_{L^2} \mathcal{G} (\Pi_M v)  \chi_R(\|v\|) \mu (dv)\leq \frac{A_0}{2} \,.
 \end{equation}
Finally, passing $R, M \to \infty$ and using the monotone convergence theorem and \eqref{cntc}, we obtain \eqref{EstimateH4_mu}, and in particular $\mu$ is supported on $\mathcal{H}^\sstt$.
\end{proof}

Next, we use structural properties of $\mathcal{A}$ to derive a crucial property of $\mathcal{G}$ used below.  Specifically,  using the fact that $\sstt\geq \sst$, we have
\begin{equation}
 e^{\rho(\|u\|_{H^{\sst}})}1_{\{\|u\|_{H^{\sstt}}\leq 1\}}+ e^{\rho(\|u\|_{H^{\sst}})}1_{\{\|u\|_{H^{\sstt}}\geq 1\}}\leq e^{\rho(1)}+e^{\rho(\|u\|_{H^{\sst}})}\|u\|_{H^\sstt}^2 \leq e^{\rho(1)}+ \mathcal{G}(u),
\end{equation}
and by \eqref{EstmuN}
\begin{align}
\int_{L^2}e^{\rho(\|u\|_{H^{\sst}})}\mu_N (du)\leq C=C(\rho, A_0).\label{MainEstaux}
\end{align}
To simplify notation below, define $\xi : [0, \infty) \to [0, \infty)$ as $\xi(x) = \rho^{-1}(3x)$, or equivalently $\rho = 3\xi^{-1}$. Since $\rho$ is increasing and convex, $\xi$ is an increasing concave function. Then, 
\eqref{MainEstaux} becomes
\begin{align}
\int_{L^2}e^{3\xi^{-1}(\|u\|_{H^{\sst}})}\mu_N (du)\leq C=C(\xi,A_0).\label{MainEst}
\end{align}



\begin{prop}\label{pro:utss}
Fix $N\geq 1$,   $r\leq \sst$,  and recall that  $(\phi^N_t)_{t \in \R}$ is the flow of \eqref{Euler(N)}.  
Then,  there is $C = C(\xi) > 0$, such that for any integer $i \geq 1$, there is a set $\Sigma^i_{N,r} \subset \mathcal{E}_N$ with
\begin{align}
\mu_N(\mathcal{E}_{N, r} \setminus \Sigma_{N, r}^i )\leq Ce^{-2i},\label{ConstrSigma}
\end{align}
such that for all $u_0\in\Sigma^i_{N,r}$, we have
\begin{align}
\|\phi^N_t u_0\|_{H^r} \leq 2\xi(1+i+\ln(1+|t|)), \quad \textrm{for all } t\in\R\,.\label{estimeePolynomT} 
\end{align}

\end{prop}

\begin{proof}
As above, it suffices to consider $t>0$. Fix $i \geq 1$ and for each $j\geq 1$ define  the set  
\begin{align} \label{dbij}
B^{i,j}_{N,r}=\left\{u\in \mathcal{E}_{N, r} : \ \ \|u\|_{H^r} \leq \xi(i+j)\right\}.
\end{align}
By  \eqref{Est_onSeq},  one has for all $t \in [0, T]$
\begin{align}
\phi^N_t B^{i,j}_{N,r}\subset \{u\in \mathcal{E}_{N, r} : \ \ \|u\|_{H^r}\leq 2\xi(i+j)\} \,, \label{Injection_Bij}
\end{align}
where $T \geq C (\xi(i+j))^{-1}$ and $C$ is independent of $N$.
Define the set
\begin{align}\label{dsij}
\Sigma^{i,j}_{N,r}=\bigcap_{k=0}^{\left[\frac{e^j}{T}\right]}\phi^N_{-kT}(B^{i,j}_{N,r}) \,,
\end{align}
where $[p]$ denotes the largest integer smaller than $p$. 
Using the invariance of $\mu_N$ and $\mathcal{E}_{N, r}$ under $\phi^N_t,$ we have
\begin{align*}
\mu_N (\mathcal{E}_{N, r} \backslash \Sigma^{i,j}_{N,r}) &=
\mu_N  \left(\bigcup_{k=0}^{\left[\frac{e^j}{T}\right]} \mathcal{E}_{N, r} \setminus \phi^N_{-kT}\left(B^{i,j}_{N,r}\right)\right)
\leq \sum_{k = 0}^{\left[\frac{e^j}{T}\right]} \mu_N \left( \mathcal{E}_{N, r} \setminus \phi^N_{-kT}\left(B^{i,j}_{N,r}\right) \right)
\\
&\leq \left(\left[\frac{e^j}{T}\right]+1\right)\mu_N (\mathcal{E}_{N, r} \backslash B^{i,j}_{N,r}).
\end{align*}
Since $r\leq \sst$,  the estimate \eqref{MainEst} and the Chebyshev inequality imply
\begin{equation}
\mu_N(\mathcal{E}_{N, r} \backslash B^{i,j}_{N,r}) \leq \frac{1}{e^{3\xi^{-1}(\xi (i +j))}} \int_{\mathcal{E}_{N, r} }e^{3\xi^{-1}(\|v\|_{H^\sst})}\mu_N(dv) \lqq{\xi}  e^{-3(i +j)} \,.
\end{equation}
Consequently, since $\xi$ is concave one has $\xi(x) \lqq{\xi} 1 + x$. Then, by using $T \ggeq (\xi(i+j))^{-1}$ and $e^x \geq 1 + x$ we obtain
\begin{align}
\mu_N(\mathcal{E}_{N, r} \backslash \Sigma^{i,j}_{N,r}) \lqq{\xi} e^j \xi(i+j) e^{-3(i +j)} \lqq{\xi} e^j (1 + (i+j)) e^{-3(i+j)} \lqq{\xi} e^{-2i}e^{-j}.
\end{align}
Hence,  \eqref{ConstrSigma} follows for 
\begin{align}\label{dsi}
\Sigma^i_{N,r}=\bigcap_{j\geq 1}\Sigma^{i,j}_{N,r} \,.
\end{align}
Next, we claim that for $u_0\in\Sigma^{i,j}_{N,r},$ we have
\begin{align}
\|\phi^{N}_t u_0\|_{H^r}\leq 2\xi(i+j)\quad \textrm{for all } \ t\leq e^j.\label{IneqflowN2j}
\end{align}
Indeed, for $t\leq e^j$, we can write $t=kT+\tau$, where $k$ is an integer in $\left[0,\frac{e^j}{T}\right]$ and $\tau \in [0,T]$. Also, by the definition of $\Sigma_{N,r}^{i,j}$, 
for any fixed integer $k \in \left[0,\frac{e^j}{T}\right]$, 
any $u_0\in\Sigma^{i,j}_{N,r}$ can be written as $\phi^N_{-kT}w$ for some $w \in B_{N,r}^{i,j}$.  Hence, since
\begin{align*}
\phi^N_t u_0=\phi^N_\tau \phi^N_{kT}u_0=\phi^N_\tau w \,,
\end{align*}
using \eqref{Injection_Bij}, we obtain \eqref{IneqflowN2j}.

For any fixed $t\geq 0,$ there is $j\geq 1$ such that $e^{j-1} - 1\leq t \leq e^{j}$, and in particular
\begin{align}
j\leq 1+\ln (1+t).
\end{align}
Then, for any $u_0 \in \Sigma^i_{N,r} \subset \Sigma^{i, j}_{N,r}$, the monotonicity of $\xi$ and \eqref{IneqflowN2j} yield
\begin{align}
\|\phi_t^N u_0\|_r \leq 2\xi(1+i+\ln(1+t)) 
\end{align}
and \eqref{estimeePolynomT} follows.
\end{proof}

\begin{Rmk}\label{rmk:mon}
Notice that  for any $r \geq 0$, $\mathcal{E}_{N, r} \subset \mathcal{E}_{M, r}$ if $N \leq M$. 
If $B^{i, j}_{N, r}$ is defined by \eqref{dbij}, then for any $j$, any $N \leq M$, $r \geq p$, and $i \leq k$ one has 
\begin{equation}
B^{i, j}_{N, r} \subset B^{k, j}_{M, p} \,,
\end{equation}
and consequently 
\begin{equation}
\Sigma^i_{N, r} \subset \Sigma^k_{M, p} \,.
\end{equation}
\end{Rmk}

\begin{prop}\label{Propexists1}
Fix $r \in (3, \sst)$, and any integer $i \geq 1$.
For any  $t \in \R$, there is an integer $i_1 = i_1(i, t, r, \sst) \geq 1$ such that for any $u_0\in \Sigma_{N,\sst}^i$, one has
$\phi^N_t(u_0) \in \Sigma_{N, r}^{i+i_1}$. Recall that $\Sigma_{N,\sst}^i$ was fixed in Proposition \ref{pro:utss}.
\end{prop}

\begin{proof}
Without loss of generality fix $t>0$, where negative times are treated similarly. 
 Let $i_0 := i_0(t)$ be such that for every $j\geq 1$  one has 
$e^j+t\leq e^{j+i_0}$, and notice that this inequality remains true if we increase $i_0$. 

For any  fixed integer $i \geq 1$ and  $u_0\in \Sigma_{N,\sst}^i$,  one has by\eqref{dsi} that  $u_0\in \Sigma_{N,\sst}^{i, j}$ for each $j \geq 1$. 
 If $t_1 \leq e^j$, then $t_1 + t \leq e^j + t \leq e^{j + i_0}$ and \eqref{IneqflowN2j} implies
\begin{align*}
\|\phi^N_{t_1+t}u_0\|_{H^\sst}\leq  2\xi(j + i + i_0) \,. 
\end{align*}
Due to \eqref{estimeePolynomT} with $t = 0$ and the monotonicity of $\xi$, for every $u_0\in\Sigma_{N,\sst}^i$ there holds
\begin{align*}
\|u_0\| \lqq{\sst} \| u_0\|_{H^\sst} \lqq{\sst} \xi(1+i) \,.
\end{align*}
Since the $L^2$ norm is preserved by the flow of \eqref{Euler(N)},  for every $u_0\in \Sigma_{N,s}^i$ we have
\begin{align*}
\|\phi^N_{t_1+t} u_0\| \lqq{\sst} \xi(1 + i).
\end{align*}
Then, an interpolation implies that for 
every $r \in (2, \sst)$ there is $\theta\in (0,1)$, depending on $r$ such that for any $t_1 \leq e^j$
\begin{align*}
\|\phi^N_{t_1 + t} u_0\|_{H^{r}} \leq \|\phi^N_{t_1 + t} u_0\|^{\theta} \|\phi^N_{t_1 + t} u_0\|_{H^{\sst}}^{1 - \theta}  \lqq{\sst, r} 
[\xi(1 + i)]^{\theta }  [\xi(j + i + i_0)]^{1-\theta}\leq   \xi(j+ i +i_1) \,,
\end{align*}
where the last inequality (with the constant equal one) follows from the monotonicity of $\xi$ for sufficiently large $i_1 \geq i_0$ that satisfies $\xi(j+ i +i_1) \geq C_{\sst, r} (\xi(1 + i))$. Note that 
$i_1$ depends on $i, t, r, \sst$, but it is independent of $j \geq 1$. 
%
Thus, for any $t_1 \leq e^j$
\begin{align*}
\|\phi^N_{t_1}(\phi^N_tu_0)\|_{H^r}\leq \xi(j + i + i_1) \,,
\end{align*}
and in particular $\phi_{t}^N(u_0) \in \Sigma^{i + i_1,j}_{N, r}$. Since $j \geq 1$ is arbitrary,  the assertion follows from the definition 
of $\Sigma^{i + i_1}_{N, r_1}$ in \eqref{dsi}.
\end{proof}

For any $r \leq \sst$,  let us introduce the restriction of (probability) measures $\mu_N$ to the set $\Sigma^{i}_{N,r} \subset \mathcal{E}_{N, r}$:
\begin{align*}
\mu_{N,i,r}(\Gamma)=
\frac{\mu_{N}(\Gamma\cap\Sigma_{N,r}^i)}{\mu_N(\Sigma_{N,r}^i)},\quad \Gamma\in  \text{Bor}(L^2)
\end{align*}
and observe that by Proposition \ref{pro:utss}, the denominator is bigger than $\frac{1}{2}$ for any sufficiently large $i$ independent of $N$.

\begin{prop}\label{pro:conm}
For any sufficiently large $i\in \N$ any $r<\sst$, the sequence $(\mu_{N,i,r})_{N\geq 1}$ is tight in $\mathcal{H}^r$. In particular, there is a subsequence of $(\mu_{N,i,r})_{N\geq 1}$
 that converges weakly to a measure $\mu_{i,r}$ supported in $\mathcal{H}^r$.
\end{prop}

\begin{proof}
By \eqref{MainEst}, for any sufficiently large $i$
\begin{equation}
\int_{\mathcal{H}^\sst} e^{3\xi^{-1}(\|v\|_{H^\sst})}\mu_{N, i, r} (dv) \leq 2 \int_{\Sigma_{N,r}^i} e^{3\xi^{-1}(\|v\|_{H^\sst})}\mu_{N} (dv) \leq  2 \int_{\mathcal{E}_{N,r}} e^{3\xi^{-1}(\|v\|_{H^\sst})}\mu_{N} (dv)
\lqq{\xi}  C 
\end{equation}
and the tightness of the set $(\mu_{N,i,r} )_{N\geq 1}$ follows from Chebyshev inequality
\begin{equation}
\mu_{N, i, r} (H^r \setminus B_R(H^\sst)) \leq \frac{1}{e^{3\xi^{-1}(R)}} \int_{H^r} e^{3\xi^{-1}(\|v\|_{H^\sst})}\mu_{N, i, s} (dv) \lqq{\xi} \frac{1}{e^{3\xi^{-1}(R)}} 
\end{equation}
and $\xi^{-1}(R) \to \infty$ as $R \to \infty$. The convergence is then a consequence of the 
Prokhorov theorem. 
\end{proof}

Recall $\sst \geq 4$ and for any $r \leq \sst$ define sets
\begin{align}\label{defsig}
\Sigma_r^i=  \overline{\bigcup_{N \geq 1} \Sigma_{N, r}^i} , \qquad  \Sigma_r = \bigcup_{i\geq 1}\Sigma_{r}^i \,,
\end{align}
where the closure is in the topology of $H^r$. 


\begin{prop}\label{PropComparaisonMeasures}
For any  $r \leq \sst$, 
let $\mu_{i, r}$ and $\mu$ be respectively measures as in Proposition \ref{pro:conm} and Proposition \ref{pro:domm}.
Then, the set $\Sigma_r^i$  is of full $\mu_{i,r}$-measure. Also,  
\begin{align}
\mu(\Sigma_r)=1, 
\end{align}
and in particular $\mu(\overline{\Sigma}_r \setminus \Sigma_r) = 0$.
\end{prop}

\begin{proof}
Fix any $x_0 \in \mathcal{H}^r \setminus \Sigma_r^i$. Since $\Sigma_r^i$ is closed, there is $\delta > 0$ such that $B_{\delta} (x_0) \cap \Sigma_r^i = \emptyset$,
where $B_{\delta} (x_0)$  is a ball of radius $\delta$ in $\mathcal{H}^r$ centered at $x_0$. 
Consequently,  for each $N \geq 1$,  $B_{\delta} (x_0) \cap \Sigma_{N, r}^i = \emptyset$.
Since $\mu_{N, i, r}$
is supported on $\Sigma_{N,r}^i$, and $B_{\delta} (x_0)$ is open,  we obtain by Portmanteau theorem 
\begin{equation}
0 = \liminf_{N \to \infty}  \mu_{N, i, r} (B_{\delta} (x_0)) \geq  \mu_{i, r} (B_{\delta} (x_0)) \,,
\end{equation}
and therefore $x_0$ is not in the support of $\mu_{i, r}$. Since $x_0 \in \mathcal{H}^r \setminus \Sigma_r^i$ was arbitrary, the first assertion follows. 

%

To prove the second statement we use 
the Portmanteau theorem, $\Sigma^i_{N,r}\subset\Sigma^i_r$, and \eqref{ConstrSigma} to obtain
\begin{align*}
\mu (\Sigma_{r}^i)\geq\limsup_{N\to\infty} \mu_{N} (\Sigma_r^i)
\geq\limsup_{N\to\infty} \mu_{N}(\Sigma_{N,r}^i)\geq 1-Ce^{-2i}.
\end{align*}
Since $(\Sigma^i_r)_{i\geq 1}$ is non-decreasing (see Remark \ref{rmk:mon}), 
\begin{align*}
\mu(\Sigma_r)=\lim_{i\to\infty}\mu(\Sigma^i_r)\geq 1 \,,
\end{align*}
and since $\mu$ is a probability measure,
\begin{align*}
\mu(\Sigma_r)=1\,,
\end{align*}
as desired. 
%
%
\end{proof}

\begin{prop}\label{PropGWP}
For any $u_0\in \Sigma_\sst$, there is a unique global in time solution $u$ of \eqref{Euler(N)} with $N=\infty$ and $u(0) = u_0$. 
Therefore, we obtain a global flow $\phi_t$ for \eqref{Euler1} defined on $\Sigma_\sst$. We have, in fact,
\begin{enumerate}
    \item For all $t\in\R$\begin{align}
        \|u(t)\|_{H^{\sst}}\leq C(\|u_0\|_{H^\sst})\xi(1+\ln(1+|t|));\label{Bound_on_solutions}
    \end{align}
    \item For all $u_0,\ v_0\in \Sigma_\sst$ and all $T>0$, if $u$, $v$ are solutions associated respectively to $u_0$, $v_0$, then
    \begin{align}
        \sup_{t\in [-T,T]}\|u(t)-v(t)\|_{H^\sst}\leq C(T,\|u_0\|_{H^\sst},\|v_0\|_{H^\sst})\|u_0-v_0\|_{H^\sst}.
\label{Continuity_solutions}    \end{align}
\end{enumerate}
\end{prop}

\begin{proof}
By \eqref{defsig}, for fixed $u_0 \in \Sigma_\sst \subset \mathcal{H}^{\sst}$,  there is an integer $i \geq 1$ such that $u_0 \in \Sigma^i_\sst$.
Since the sequence of sets $(\Sigma^i_{N, \sst})_{N \geq 1}$ is nested,
  there is a sequence $(u_{0,N})_{N \geq 1}$ such that $u_{0,N}\in \Sigma^{i}_{N,\sst}$ and 
\begin{align}
\lim_{N\to\infty}\|u_{0,N}-u_0\|_{H^\sst}=0.\label{Conv_H3}
\end{align}

Fix arbitrary time $T_0 \geq T$, where $T$ is as in Proposition \ref{UniformLWP}. For any fixed $N \geq 1$, let $u_N$ be the solution of \eqref{Euler(N)} 
with $u_N(0) = u_{N, 0}$. 
By \eqref{estimeePolynomT}, 
 \begin{align}
 \|u_N(t)\|_{H^\sst}\leq 2\xi(1+i+\ln(1+T_0)) =: \Lambda, \quad \textrm{ for all } t\in [-T_0,T_0] \,, \label{UsePolynT}
 \end{align}
 where $\Lambda$ is crucially independent of $N$. Setting $t = 0$, we obtain
 \begin{align}
 \|u_{0,N}\|_{H^\sst} \leq \Lambda\,. \label{Est_sleq4}
 \end{align}
and after passing $N \to \infty$, we arrive at 
 \begin{align}
 \|u_{0}\|_{H^\sst} \leq \Lambda.
 \end{align}
Thus, if $R = \Lambda+1$, then 
 $u_0, u_{0,N} \in B_R(\mathcal{H}^\sst).$ Denote $u$ the solution of \eqref{Euler(N)} with $N = \infty$ and $u(0) = u_0$, which by Proposition \ref{UniformLWP} exists on 
 time interval $[-T, T]$ with $T$ depending on $R$. Also, let $u_N$ be the solution of \eqref{Euler(N)} with $u(0) = u_{0, N}$, which exists globally by Corollary \ref{Cor-GWP-EulerN}. 
Then \eqref{Conv_H3} combined with Corollary \ref{Cor_Conv} yields for any $r \in (3, \sst)$
 \begin{align}
 \lim_{N\to\infty}\|u-u_N\|_{X^{r}_T}=0 \label{Conv-H3-solution} \,.
 \end{align}
 By  \eqref{UsePolynT} for any $r \in (3, \sst)$
\begin{align}
\|u \|_{X^{r}_T} \leq \|u-u_{N}\|_{X^{r}_T}+\|u_{N}\|_{X^{r}_T} \leq \|u-u_{N}\|_{X^{r}_T} + \Lambda \,.
\end{align}
Then, after passing $N \to \infty$, if  follows from \eqref{Conv-H3-solution} for any $r \in (3, \sst)$ 
\begin{align}
\|u(T)\|_{H^r} \leq \Lambda \label{Est_sles4u(t)} \,.
\end{align}
Since $\Lambda$ does not depend on $r$, the continuity of Sobolev norms imply
\begin{align}
\|u(T)\|_{H^{\sst}} \leq \Lambda \,.
\end{align}
Hence,  since the existence time is lower bounded by the norm of the initial condition  (bounded by $\Lambda$) we can iterate the above procedure and obtain 
that $u$ exists on $[-T_0,T_0]$ with the estimate
\begin{align}
\|u\|_{X^{\sst}_{T_0}}\leq \Lambda.
\end{align}
Since $T_0$ is arbitrary, we obtain the global existence on $\Sigma_s$. 
The growth bound \eqref{Bound_on_solutions} follows from \eqref{estimeePolynomT}. The uniqueness and continuous dependence  on initial conditions \eqref{Continuity_solutions} follow from Proposition \ref{UniformLWP} and \eqref{Bound_on_solutions}.
\end{proof}

\begin{rmq}
Fix $r \in (3, \sst)$. 
If $(\phi_t)_{t \in \R}$ is the global flow on $\Sigma_\sst$ constructed in Proposition \ref{PropGWP}, then iteration and 
Corollary \ref{Cor_Conv} gives us for each $t \in \R$ and $u_0 \in \Sigma_s$ that
\begin{equation}\label{Convergence_phit-phitN}
\lim_{N\to\infty}\|\phi_t u_0-\phi^N_{t}u_{0,N}\|_{H^r} = 0,
\end{equation}
provided $u_{0,N} \in \mathcal{E}_N$ and
\begin{equation}
\lim_{N\to\infty} \|u_0 - u_{0,N}\|_{H^r} = 0\,.
\end{equation}
\end{rmq}

The next result identifies an invariant set of the  flow $\phi_t$.

\begin{prop}
If we define 
\begin{equation}
\Sigma_\sst^* = \bigcap_{r < \sst} \Sigma_{r} \,,
\end{equation}
then  the flow $(\phi_t)_{t \in \R}$ constructed in Proposition \ref{PropGWP} satisfies $\phi_t\Sigma_\sst^* = \Sigma_\sst^*$.
\end{prop}

\begin{proof}
Fix any $t \in \R$ and any $u_0 \in \Sigma_\sst^*$. Then,  $u_0 \in \Sigma_r$ for any $r < \sst$. 
By the definition of $\Sigma_r$ (see \ref{defsig}),  there is an integer $i$ such that $u_0 \in \Sigma_r^i$. Consequently, $u_0$ is  the  limit of a sequence 
$(u_{0,N})_{N \geq 1}$ in $\mathcal{H}^r$ such that $u_{0,N}\in \Sigma_{N,r}^i$ for every $N \geq 1$. 
By Proposition \ref{Propexists1}, for each $r_1 < r$ there is $i_1:=i_1(i, t, r_1, r)$ such that $\phi^N_t(u_{0,N})\in \Sigma_{N, r_1}^{i+i_1}$.
From \eqref{Convergence_phit-phitN} follows $\phi_t(u_0)\in \Sigma_{r_1}^{i+i_1}\subset \Sigma_{r_1}$.
Since $r_1 < r$ is arbitrary, we have 
\begin{equation}
\phi_t\Sigma_\sst^* \subset \bigcap_{r_1 < r} \Sigma_{r_1}  \,.
\end{equation}
But also $r$ was arbitrary, and we obtain 
\begin{equation}
\phi_t\Sigma_\sst^* \subset  \bigcap_{r < \sst} \bigcap_{r_1 < r} \Sigma_{r_1}   = \bigcap_{r_1 < \sst}  \Sigma_{r_1} = \Sigma_\sst^*\,.
\end{equation}
Now, let $u$ be in $\Sigma_\sst$, since $\phi_t$ is well-defined on $\Sigma_\sst$ we can set  $u_0=\phi_{-t}u$, we then have $u=\phi_tu_0$, 
and hence $\Sigma_\sst^* \subset \phi^{t}\Sigma_\sst^*$, the assertion follows. 
\end{proof}
\subsection{Global well-posedness for 3D Euler}\label{Subs.Euler}
Here we summarize the GWP result we obtained for the 3D Euler equations.\\
\begin{rmq}[GWP for the 3D Euler system]\label{Remark_3D_Euler}
Due to lack of second conservation, we cannot rule out the scenario that the infinite-dimensional limiting measure of 3D Euler be concentrated on $0$.
    Now, since $\Sigma_{N,r}^i\subset \Sigma_r^i$ and using the fact that $\Sigma_{N_1,r}^i\subset\Sigma_{N_2,r}^i$ if $N_1\leq N_2$ (see Remark \ref{rmk:mon}), the GWP that has been proved in Proposition \ref{PropGWP} holds also for data $u_0\in\Sigma_{N,r}^i$. Therefore, this observation shows that the finite-dimensional data living on the statistical ensemble for 3D Euler, that we denote by $\Sigma_{s^*}^{\text{Euler}}$, lead to globally well-posed solutions that enjoy the slow growth bound \eqref{Bound_on_solutions}. Other properties of this set of data will be proved in Section \ref{3D Euler}.
\end{rmq}
\begin{rmq}[Persistence of regularity]
Since any finite-dimensional datum in $\Sigma_{s^*}^{\text{Euler}}$ is of $C^\infty$-regularity, a crucial question is about the preservation by the flow of such regularity. Initially, the estimate established in \eqref{Bound_on_solutions} provides only a bound on the $H^s$-regularity, where $s$ is the regularity order of the dissipation operator used in the construction of underlying measures. We rely here on an argument of Beale-Kato-Majda type to ensure persistence of regularity.\\ 
    To proceed to it, we remark that a similar energy estimation procedure as in the proof of Proposition \ref{UniformLWP} shows the following control (see also \cite{BealeKatoMajda1984}):
    \begin{align}
        \|u(t)\|_{H^m}\leq\|u_0\|_{H^m}e^{C\int_0^t\|\nabla u\|_{L^\infty}ds},
    \end{align}
    for all $m\in \N$.\\
    Using the fact that we consider here $s>\frac{5}{2}$, and thanks to estimate \eqref{Bound_on_solutions}, we have that 
    \begin{align}
        \|\nabla u\|_{L^\infty}\leq C \|u\|_{H^s}\leq \Tilde{C}\xi(1+|t|).\label{Est_Grad}
    \end{align}
    We then obtain for any $u_0\in C^\infty(\T^3)$, for all integer $m$, that
    \begin{align}
        \|u(t)\|_{H^m}\leq \|u_0\|_{H^m}e^{\tilde{C}\int_0^t\xi(1+|s|)ds},\label{Est_reg_Euler}
    \end{align}
    which implies regularity for all times due to control on $\xi$.\\
    We could improve the estimate  \eqref{Est_reg_Euler} above by using, in \eqref{Est_Grad}, an interpolation between $L^2$ and $H^s$ with a use of $L^2$-conservation, resulting in a power gain in the integrand.
\end{rmq}
\section{Invariance of the measure}
\label{ASection6InvarMeas}

 Recall that $\phi_t$ is the flow of \eqref{Euler1} 
on the $\Sigma_\sst$ as shown in Proposition \ref{PropGWP}. 

\begin{thm}
The measure $\mu$ is invariant under $\phi_t$, that is, for each $\Gamma \in \textrm{Bor}(\mathcal{H}^\sst)$ and each $t \in \R$ one has $\mu(\Gamma) = \mu(\phi_{-t} \Gamma)$. 
\end{thm}

\begin{proof}
Since $\mu$ is a Borel probability defined on a Polish space,  the Ulam's theorem (see Theorem $7.1.4$ in \cite{dudley}) states that $\mu$ is regular: 
for any $S\in \text{Bor}(\mathcal{H}^\sst)$
\begin{align*}
\mu(S)=\sup\{\mu(K),\ K\subset S\ compact\}.
\end{align*}
Therefore, it suffices to prove invariance for compact sets. Indeed for any $t$, using $\mu(\phi_t K) = \mu(K)$ we obtain
\begin{align}
\mu(\phi_{-t}S)&=\sup\{\mu(K),\ K\subset \phi_{-t}S\ compact\}=\sup\{\mu(\phi_{t}K),\ K\subset \phi_{-t}S\ compact\}\\
&=\sup\{\mu(\phi_{t}K),\ \phi_{t}K\subset S,\ K\  compact\}\leq \sup\{\mu(C),\ C\subset S\ compact\} =\mu(S),
\end{align}
where we used that $\phi_t$ is continuous, and therefore $\phi_t$ maps compact sets into compact sets.
In addition, for any $t$ we have
\begin{align*}
\mu(S)=\mu(\phi_{-t}\phi_t S)\leq \mu(\phi_tS).
\end{align*}
Since $t$ is arbitrary, we then obtain $\mu(S) = \mu (\phi_t S)$.

Next, we claim that it suffices to show the invariance only on a fixed interval $[-\tau,\tau],$ where $\tau>0$ can be as small as we want. 
Indeed, if  $\mu(K) = \mu(\phi_t K)$ for any $t \in [-\tau, \tau]$ with $\tau > 0$, then for $t \in [\tau, 2\tau]$, one has $\mu(\phi^{-t}K)=\mu(\phi^{-\tau}\phi^{t-\tau}K)=\mu(\phi^{t-\tau}K)=\mu(K)$, 
where use used the invariance for  $t-\tau \in [0, \tau]$. For larger $t$ we can iterate and the same argument applies for $t< 0$.

Our proof is then reduced to showing invariance for compact sets on a small time interval. 
Recall that in Section \ref{secl:invlim} we defined the Markov group $\Phi^N_t$ and its dual $\Phi^{N*}_t$ corresponding to the flow $\phi^N_t$. 
Analogously, for the flow $(\phi_t)_{t \in \R}$ 
let us define Markov group $\Phi_t :  L^\infty(L^2;\R)\to L^\infty(L^2;\R)$ and its dual
$\Phi_{t}^* : \mathfrak{p}(L^2)\to \mathfrak{p}(L^2)$ as
\begin{align}
\Phi_t f(v) = f(\phi_t v), \qquad 
\Phi_{t}^* \nu(\Gamma) =\nu(\phi_{-t} \Gamma) \qquad \textrm{ for } t\in \R \,.
\end{align}

The idea of the proof is indicated in the following commutative diagram
$$
\hspace{10mm}
\xymatrix{
  \Phi^{N*}_{t}\mu_{N} \ar@{=}[r]^{(I)} \ar[d]^{(III)} & \mu_{N} \ar[d]^{(II)} \\
    \Phi_{t}^* \mu \ar@{=}[r]^{(IV)} & \mu
  }
$$
The equality $(I)$ is the invariance of $\mu_N$ under $\Phi_{N}^t$, and $(II)$ is the weak convergence $\mu_N\to\mu^a$. Then $(IV)$ is proved once $(III)$ is verified.

As discussed above, it suffices to show (III) on compact sets and any small $t$. Fix $r \in (3, \sst)$,  a compact set $K$ and a large $R$ such that 
$K \subset B_R(\mathcal{H}^\sst) \subset B_{cR} (\mathcal{H}^r)$, where $c = c(\sst, r)$ is a constant from Poincar\' e inequality. 
Also, let $f : \mathcal{H}^r \to \R$ be a bounded, Lipschitz  function  supported on $B_{cR}(\mathcal{H}^\sst)$. Note that for any $u, v \in B_R(\mathcal{H}^\sst)$
\begin{equation}
|f(u) - f(v)| \leq C_f \|u - v\|_{H^r} \lqq{r, \sst} \|u - v\|_{H^\sst} \,,
\end{equation}
and therefore
the restriction of $f$ to $H^\sst$ is also Lipschitz, and in particular continuous. 
Let $T$ be the existence time given by Proposition \ref{UniformLWP} corresponding to the fixed $R$. Then, for $|t| < T$ we have
\begin{align*}
( \Phi_{t}^{N*} \mu_N, f) - ( \Phi_t^{*}\mu, f) &= (\mu_N, \Phi_t^{N} f) - (\mu, \Phi_t f)\\
&=(\mu_N, \Phi^{N}_t f-\Phi_t f) - (\mu-\mu_N, \Phi_t f)\\
&=A-B.
\end{align*}
By the continuity  of $\phi^t$,  $\Phi_t f$ is bounded and continuous on $\mathcal{H}^\sst$ and by the weak convergence $\mu_N \to \mu$ in $\mathcal{H}^\sst$ (up to a subsequence, 
see Proposition \ref{pro:domm}), we have  $B\to 0$ as $N\to \infty$.

Finally, from the Lipschitz property of $f$ on $H^r$ and Lemma \ref{Lem_Conv_H3} follows
\begin{align*}
|A|\leq C_f\sup_{u\in B_R(\mathcal{H}^s)}\|\phi^N_t (u)-\phi^t(u)\|_{H^r} \leq C_f\sup_{u\in B_R(\mathcal{H}^s)}\|\phi_N^t(u)-\phi^t(u)\|_{H^r}\to 0,\ \ as\ N\to\infty.
\end{align*}
We obtain the claim.
\end{proof}

\section{Absolute continuity and large data properties for 3D Euler}\label{3D Euler}
In this section we restrict ourselves to the 3D Euler case, we prove that the measures $(\mu_N)_N$ associated to Euler and the generated statistical ensemble $\Sigma_{s^*}^{\text{Euler}}$ (defined in Remark \ref{Remark_3D_Euler}) have some qualitative properties such as absolute continuity and large data.
\begin{thm}
    For any $N\geq 1$, the measure $E_*\mu_N$ is absolutely continuous w.r.t to the Lebesgue measure on $\R$. Here $E$ denotes the energy of the solutions of the Euler equations. Namely,
     \begin{align}\label{AC}
        E_*\mu_N(\Gamma)=\P(\frac{1}{2}\|u\|_{L^2}^2\in \Gamma)\leq h(\ell(\Gamma)), \quad \text{for all Borel set $\Gamma\subset\R$},
    \end{align}
    where $h:\R\to\R_+$ satisfies $h(0)=0$.
    In particular, for any $N$, $\Sigma_N$ is not included in a countable union of $(2N-1)$-spheres.
\end{thm}
\begin{proof}
    We employ the local time method introduced by Shirikyan \cite{armen_nondegcgl}. Notice that the original approach used two coercive conservation laws to obtain the needed bound. Here we show that, despite the lack of a second coercive conservation laws, we can achieve the same conclusion by making use of the finite-dimensional context.
    Let us first notice, by invoking regularity of the Lebesgue measures, that is suffices to show \eqref{AC} for $\mu_{\al,N}$ with $h$ independent of $\al$. For simplicity of the notations, in this proof we use the notational abuse $u:=u_{\al,N}$, where $u_{\al,N}$ is the stationary process constructed for the equation \eqref{Euler(N,alpha)}.\\
    Let's now establish a key property of absolute continuity around $0$, namely
    \begin{align}
        \P(0\leq \|u\|_{L^2}\leq \delta)\leq C(N)\delta,
    \end{align}
    where $C$ depends only on $N$.
    
    To prove it, we proceed by the relation \eqref{AC_gen} below, which is derived from the local time properties of the process $\|u\|_{L^2}^2$. The detailed proof follows \cite{armen_nondegcgl}.

    \begin{align}\label{AC_gen}
        &\E \int_\Gamma 1_{(a,\infty)}(g(\|u\|_{L^2}^2))\left(g'(\|u\|_{L^2}^2)\left(\frac{A_{0,N}}{2}-e^{\rho(\|u\|_{H^\sst})}\|u\|_{H^\sstt}^2\right)+g''(\|u\|_{L^2}^2)\sum_{|m|\leq N}a_m^2|\langle u,e_m\rangle|^2\right)da\\
        &+\sum_{|m|\leq N}a_m^2\E\left(1_\Gamma(g(\|u\|_{L^2}^2))(g'(\|u\|_{L^2}^2))^2|\langle u,e_m\rangle|^2\right)=0.\nonumber
    \end{align}
    In the this relation $g$ is any $C^2(\R,\R)$ function and $\Gamma$ in any Borel subset of $\R$.
    Then, for any such $g$ and $\Gamma$, we have that
    \begin{align}
        \E \int_\Gamma 1_{(a,\infty)}\left(g(\|u\|_{L^2}^2)\right)\left(g'(\|u\|_{L^2}^2)\left(\frac{A_{0,N}}{2}-e^{\rho(\|u\|_{H^\sst})}\|u\|_{H^\sstt}^2\right)+g''(\|u\|_{L^2}^2)\sum_{|m|\leq N}a_m^2|\langle u,e_m\rangle|^2\right)da\leq 0.\label{AC_ineq_fund}
    \end{align}
    Consider $\Gamma =[\beta,\gamma]$ with $\beta>0$. If we take $g\in C^2(\R)$ such $g(x)=\sqrt{x}$ for $x\geq\beta$ and $g(x)=0$ for $x\leq 0$, the we obtain
    \begin{align}
        \E \int_\beta^\gamma 1_{(a,\infty)}\left(\|u\|_{L^2}\right)\left(\frac{A_{0,N}-2e^{\rho(\|u\|_{H^\sst})}\|u\|_{H^\sstt}^2}{4\|u\|_{L^2}}-\frac{1}{4\|u\|_{L^2}^3}\sum_{|m|\leq N}a_m^2|\langle u,e_m\rangle|^2\right)da\leq 0.
    \end{align}
    Then,
    \begin{align}
        \E \int_\beta^\gamma \frac{1_{(a,\infty)}(\|u\|_{L^2})}{\|u\|_{L^2}}\left(A_{0,N}\|u\|_{L^2}^2-\sum_{|m|\leq N}a_m^2|\langle u,e_m\rangle|^2\right)\,da\leq 2(\gamma-\beta)\E\frac{2e^{\rho(\|u\|_{H^\sst})}\|u\|_{H^\sstt}^2}{{\|u\|_{L^2}}}.
    \end{align}
    Now,
    \begin{align}
        \E\frac{2e^{\rho(\|u\|_{H^\sst})}\|u\|_{H^\sstt}^2}{{\|u\|_{L^2}}}\leq C_2(N)\E e^{\rho(\|u\|_{H^\sst})}(1+\|u\|_{H^\sstt}^2)\leq C_3(N).
    \end{align}
    Also, since $a_m\neq 0$ for all $m$, we the number $k_N=A_{0,N}-\max_{|m|\leq N}a_m^2$ is positive. We obtain
    \begin{align}
        A_{0,N}\|u\|_{L^2}^2-\sum_{|m|\leq N}a_m^2|\langle u,e_m\rangle|^2\geq k_N\|u\|_{L^2}^2.
    \end{align}
    We arrive at
    \begin{align}
        \E \int_\beta^\gamma 1_{(a,\delta)}(\|u\|_{L^2})\|u\|_{L^2}da\leq C_4(N)(\gamma-\beta)
    \end{align}
    where $\delta>0$. Hence
    \begin{align}
        \frac{1}{\gamma}\int_0^\gamma\P(a\leq \|u\|_{L^2}\leq \delta)\,da\leq C_4(N)\delta. 
    \end{align}
    Letting $\gamma\to 0$, we find
    \begin{align}
        \P(0<\|u\|_{L^2}\leq \delta)\leq C_4(N)\delta
    \end{align}
    Now it remains to show that $\P(u=0)=0$; but this property follows easily from the equation and the fact that $\P (\beta_m=0)=0$ where $\beta_m$ is the Brownian noise entering the definition of the noise \eqref{defnoise}.

    Now, let $\Gamma$ be a Borel set in $\R$. Let $l_N=\min_{|m|\leq N}|a_m|$, we have that $l_N>0$. Let's take $g(x)=x$ in \eqref{AC_gen}. We have
    \begin{align}
        \E\left(1_{\Gamma}(\|u\|_{L^2}^2)\sum_{|m|\leq N}a_m^2|\langle u,e_m\rangle|^2\right)\leq \int_\Gamma\E\left(1_{(a,\infty)}(\|u\|_{L^2}^2)\mathcal{G}(u)\right)da\leq C\ell(\Gamma)
    \end{align}
    Then,
    \begin{align}
        \E\left(1_{\Gamma}(\|u\|_{L^2}^2)\|u\|_{L^2}^2\right)\leq \frac{C}{l_N}\ell(\Gamma)=:C(N)\ell(\Gamma).
    \end{align}
    Now,
    \begin{align}
        \P(\|u\|_{L^2}^2\in \Gamma)&=\P(\{\|u\|_{L^2}^2\in \Gamma\}\cap\{\|u\|_{L^2}^2\leq\delta\})+\P(\{\|u\|_{L^2}^2\in \Gamma\}\cap\{\|u\|_{L^2}^2>\delta\})\\
        &\leq \P(\{\|u\|_{L^2}^2\leq\delta\})+\P(\{\|u\|_{L^2}^2\in \Gamma\}\cap\{\|u\|_{L^2}^2>\delta\})\leq C(N)\delta+\P(\{\|u\|_{L^2}^2\in \Gamma\}\cap\{\|u\|_{L^2}^2>\delta\}).\nonumber
    \end{align}
    To treat the second term, we make use of the property $\{\|u\|_{L^2}^2>\delta\}$, we have
    \begin{align}
        \E\left(1_{\Gamma}(\|u\|_{L^2}^2)\|u\|_{L^2}^2\right)\geq \delta\E(1_{\Gamma}(\|u\|_{L^2}^2)1_{(\delta,\infty)}(\|u\|_{L^2}^2))=\delta\P(\{\|u\|_{L^2}^2\in \Gamma\}\cap\{\|u\|_{L^2}^2>\delta\}).
    \end{align}
    Therefore 
    \begin{align}
        \P(\{\|u\|_{L^2}^2\in \Gamma\}\cap\{\|u\|_{L^2}^2>\delta\})\leq \delta^{-1}C(N)\ell(\Gamma).
    \end{align}
    Overall,
    \begin{align}
        \P(\|u\|_{L^2}^2\in \Gamma)\leq C(N)(\delta +\delta^{-1}\ell(\Gamma)).
    \end{align}
    This finishes finish the proof.
\end{proof}

\begin{thm}
    $\Sigma_{s^*}^{\text{Euler}}$ contains arbitrary large data.
\end{thm}
\begin{proof}
    The proof is based on the arguments of the section \ref{sec:lar-data} below. We apply it on the measures $\mu_{N}$ whose statistical (sub)ensembles $\Sigma_{N,r}^i$ are contained in $\Sigma_{s^*}^{\text{Euler}}$. Therefore to obtain the result if suffices to show that these (sub)ensembles contain large data (at least for $i\to\infty$). Due to the finite-dimensionality of $\mu_N$ we have, from a passage to the limit $\alpha\to 0$ on \eqref{IdentityH4_muNalpha},
    \begin{align}
        \E_{\mu_N}\mathcal{G}(u)=\frac{A_{0,N}}{2}.
    \end{align}
    We can now apply Corollary \ref{LargeD} below.
\end{proof}

\section{Large data}\label{sec:lar-data}


In this section, we consider problems of the form \eqref{Euler1} that satisfy an additional quadratic conservation law and we 
prove that the statistical ensemble $\Sigma_*$ constructed in \eqref{defsig} contains large data as stated in Corollary \ref{LargeD}. 
More precisely,  we assume that there is $H:H^{s_1^*} \to \R$ which satisfies for each $u \in H^{s_1^*}$ and each $m$ (recall $(e_m)$ is a basis of eigenfunctions of Laplacian)
\begin{equation}\label{qcls}
D^2H(u)[e_m,e_m] = K_m, \qquad DH(u)[B(u, u)] =0\,,
\end{equation}
where $DH(u)[v]$ denotes the derivative of $H$ at $u$ in direction $v$ (and analogously for higher order derivatives).
We also note that since $H$ is quadratic, then $(u, v) \mapsto DH(u)[v]$ is a bilinear map. 
Note that the first assumption in \eqref{qcls} follows from the fact that $H$ is quadratic,  whereas the second one is equivalent to $H$ being a conservation law of \eqref{Euler1}.  
In addition,  for $\mathcal{G}$ defined in \eqref{dfmcg} we suppose that $\mathcal{H}(u):= DH(u)[\mathcal{A}(u)]$ satisfies for each $u \in H^{s_1^*}$ that 
\begin{equation}\label{uboch}
|\mathcal{H}(u)| \lesssim \mathcal{G}(u)
\end{equation}
and for each $N \in \mathbb{N}$ and $u \in H^{s_1^*}$
\begin{equation}\label{hobch}
|DH((I - \Pi_N)u)[\mathcal{A}(u)]|
 \lesssim \kappa_N \mathcal{G}(u) \,,
\end{equation}
where $\kappa_N \to 0$ as $N \to \infty$.

\begin{Rmk}
Let us verify that the additional conservation laws for the generalized SQG equation (see Example \ref{ex:active-scalar}) and shell models (see Example \ref{ex:shell} satisfy assumptions 
\eqref{uboch} and \eqref{hobch}. 
First assume that 
\begin{equation}
    B(u, v) = K[u] \nabla v, \qquad K[u] = \nabla^\perp (-\Delta)^{-1 + \alpha} u
\end{equation}
as for generalized SQG equation in Example \ref{ex:active-scalar}. Recall that 2D Euler equation is the special case with $\alpha = 0$. Then, 
\begin{equation}
    H(u) = \frac{1}{2} \int_{\T^2} ((-\Delta)^{\frac{-1 + \alpha}{2}} u)^2 \, dx
\end{equation}
which is clearly a quadratic function. Then, 
\begin{multline}
   |\mathcal{H}(u)| = e^{\rho(\|u\|_{H^\sst})} \left|\left\langle (-\Delta)^{\frac{-1 + \alpha}{2}} u, \right.\right.
   \\ 
   \left.  \left.(-\Delta)^{\frac{-1 + \alpha}{2}}\left((a_1(-\Delta)^\sstt u+a_2\Delta(|\Delta u|^{q-2}\Delta u)-a_3\nabla\cdot(|\nabla u|^{2q-2}\nabla u))  \right) \right\rangle \right| \,.
\end{multline}
After integration by parts, H\" older inequality, $\alpha \leq 1$, and interpolation we have 
\begin{align}
   |\mathcal{H}(u)|
   &\leq e^{\rho(\|u\|^2_{H^\sst})} \left(a_1\|u\|_{H^{s_1 + \alpha - 1}} + 
   a_2 |\langle |(-\Delta)^{\alpha} u|, |\Delta u|^{q- 1} \rangle + a_3 \langle |(-\Delta)^{\alpha - 1} \nabla u| , |\nabla u|^{2q-1}\rangle
   \right) \\
   &\leq 
   e^{\rho(\|u\|^2_{H^\sst})} \left(a_1\|u\|_{H^{s_1 + \alpha - 1}} + a_2 \|u\|_{W^{2\alpha, q}}
   \|u\|_{W^{2, q}}^{q-1} + a_3 \|u\|_{W^{2\alpha - 1, 2q}} \|u\|_{W^{1,2q}}^{2q-1}
   \right) \\
   &\lesssim 
   e^{\rho(\|u\|^2_{H^\sst})} \left(a_1\|u\|_{H^{s_1}} + a_2 
   \|u\|_{W^{2, q}}^{q} + a_3 \|u\|_{W^{1,2q}}^{2q}
   \right) = \mathcal{G}(u)
\end{align}
and \eqref{uboch} follows. Analogous calculations give us for $Q_N := I - P_N$ that 
\begin{multline}
|DH(Q_N u)[\mathcal{A}(u)]| \\
\leq 
e^{\rho(\|u\|_{H^\sst})} \left(a_1 \|Q_N u\|_{H^{s_1 + \alpha - 1}}\|u\|_{H^{s_1 + 2\alpha - 1}} + a_2 \|Q_N u\|_{W^{2\alpha, q}}
   \|u\|_{W^{2, q}}^{q-1} + a_3 \|Q_N u\|_{W^{2\alpha - 1, 2q}} \|u\|_{W^{1,2q}}^{2q-1}
   \right)
\end{multline}
and if $\alpha < 1$, then the inverse Poincar\' e inequality yields
\begin{equation}
    |DH(Q_N u)[\mathcal{A}(u)]| \lesssim
\lambda_N^{\alpha - 1} e^{\rho(\|u\|_{H^\sst})} \left(a_1  \|u\|_{H^{s_1 + 2\alpha - 1}}^2 + 
a_2 \|u\|_{W^{2, q}}^{q} + a_3 \|u\|_{W^{1,2q}}^{2q}
   \right)
   \lesssim \kappa_N \mathcal{G}(u) \,,
\end{equation}
where $\lambda_N$ is the $N$th largest eigenvalue of $(-\Delta)$, and therefore
$\kappa_N := \lambda_N^{\alpha - 1} \to 0$ as $N \to \infty$, as desired. 

Calculations are similar for the shell models in Example \ref{ex:shell}, so we provide only some calculations. Recall, 
\begin{equation}
       H(u) = \frac{1}{2}\sum_{n = 1}^\infty \left(\frac{-a}{a + b} \right)^n |u_n|^2, \qquad \textrm{if }
    u = \sum_{j = 1}^\infty u_j \phi_j \,, 
\end{equation}
and therefore if we denote $c_n =  \left(\frac{-a}{a + b} \right)^n$, then similarly as above
\begin{align}
   |\mathcal{H}(u)| &\leq e^{\rho(\|u\|_{H^\sst})}  \left|\left\langle\sum_{n = 1}^\infty c_n u_n \phi_n, \left((a_1(-\Delta)^\sstt u+a_2\Delta(|\Delta u|^{q-2}\Delta u)-a_3\nabla\cdot(|\nabla u|^{2q-2}\nabla u))  \right) \right\rangle \right| \\
   &\leq  
   e^{\rho(\|u\|_{H^\sst})} \left( a_1\Big\|\sum_{n = 1}^\infty c_n u_n (-\Delta)^{\frac{\sstt}{2}} \phi_n \Big\| \|u\|_{H^\sstt} + a_2\Big\|\sum_{n = 1}^\infty c_n u_n \Delta \phi_n \Big\|_{L^{q}} \|u\|_{W^{2,q}}^{q-1} \right.\\
   &\qquad +
   \left. a_3 \Big\|\sum_{n = 1}^\infty c_n u_n \nabla \phi_n \Big\|_{L^{2q}} \|u\|_{W^{1, 2q}}^{2q-1}
   \right)
   \,.
\end{align}
Since the Fourier multilipers $c_n \to 0$ exponentially, 
by the discrete version of the  H\" ormander-Mikhlin theorem (see e.g. \cite{RuzhanskyWirth2015}) we obtain 
\begin{align}
     |\mathcal{H}(u)| &\lesssim e^{\rho(\|u\|_{H^\sst})} \left( a_1\Big\|\sum_{n = 1}^\infty u_n (-\Delta)^{\frac{\sst}{2}} \phi_n \Big\| \|u\|_{H^\sst} + a_2\Big\|\sum_{n = 1}^\infty u_n \Delta \phi_n \Big\|_{L^{q}} \|u\|_{W^{2,q}}^{q-1} \right.\\
   &\qquad +
   \left. a_3 \Big\|\sum_{n = 1}^\infty u_n \nabla \phi_n \Big\|_{L^{2q}} \|u\|_{W^{1, 2q}}^{2q-1}
   \right)\\
   &= 
   e^{\rho(\|u\|_{H^\sst})} \left( a_1 \|u\|_{H^\sst}^2 + a_2 \|u\|_{W^{2,q}}^{q} + a_3  \|u\|_{W^{1, 2q}}^{2q}
   \right) = \mathcal{G}(u) 
\end{align}
and \eqref{uboch} follows. The proof of \eqref{hobch} is analogous, where we only sum $n$ between $N$ and infinity and we factor out $c_N$ and set $\kappa_N = c_N \to 0$ as $N \to \infty$.

 \end{Rmk}

\begin{Lem}
For all $N \geq 1$ we have that
\begin{align}
\int_{L^2}\mathcal{H}(\Trs)\mu_N(d\Trs) &= \frac{A^H_{N}}{2} := \frac{1}{2} \sum_{|m| \leq N}  D^2H(u)[e_m,e_m] =   \frac{1}{2} \sum_{m \in \mathbb{Z}^d_N} K_m a_m^2 \,.
 \label{Equality_N}
\end{align}
\end{Lem}

\begin{proof}
If $u_N$ satisfies  \eqref{Euler(N,alpha)},  \eqref{Euler(N,alpha,data)},  then 
similarly as in the proof of Proposition \ref{pro:anee} we have by It\^ o lemma and \eqref{qcls} that 
\begin{equation}
dH(u_N) + \alpha DH(u_N)[\mathcal{A}(u_N)]dt = \frac{\alpha}{2} A^H_N + \sqrt{\alpha} d\mathcal{M} \,, 
\end{equation}
where $\mathcal{M}$ is a martingale.  If $u_N(0)$ is distributed as $\mu_{\alpha, N}$,  where  $\mu_{\alpha, N}$ is a stationary measure as in Proposition \ref{pro:eimn}, then 
similarly as in the proof of Proposition \ref{pro:eimn} we obtain
\begin{equation}
\int_{L^2} DH(v)[\mathcal{A}(v)] \mu_{\alpha, N}(dv) = \frac{A^H_N}{2} \,,
\end{equation}
or in other notation
\begin{equation}
\int_{L^2} \mathcal{H}(v) \mu_{\alpha, N}(dv) = \frac{A^H_N}{2} \,.
\end{equation}
By \eqref{uboch} the assertion of Lemma \ref{lem:tbf} holds true with $\mathcal{G}$ replaced by $\mathcal{H}$.  Then the assertion of the lemma follows as in the proof of Proposition 
\ref{pro:cimn} with $\mathcal{G}$ replaced by $\mathcal{H}$. 
\end{proof}

\begin{Thm}\label{Thm:Equality}
If  $s^*\geq 4$ and $\mathcal{H}$ satisfies \eqref{qcls} with $K_m \leq K$ for each $m$,  and \eqref{uboch},  \eqref{hobch}, then
\begin{align}
\int_{L^2}\mathcal{H}(v)\mu(dv) &=\frac{A^H_{\infty}}{2} =: \frac{A^H}{2} \,, \label{Equality}
\end{align}
where $\mu$ is as in Proposition \ref{pro:domm}. 
\end{Thm}

\begin{proof}
By invoking the Skorohod representation, there is a sequence of random variables $(\Trs^{N})$ such that for each $N >0$,  $\Trs^{N}$ is distributed as $\mu_{N}$ and a random variable $u$ distributed as $\mu$, such that $\Trs^{N}$ (up to a sub-sequence) converges almost surely to $\Trs$ in $H^{\sstt-}$ for some sequence $N \to \infty$.  
Let $\chi :[0, \infty) \to \R$ be a $C^\infty$ function supported on $[0, 2]$ such that $\chi = 1$ on $[0, 1]$.  Define $\chi_R(\cdot)=\chi(\frac{\cdot}{R})$ and then
\eqref{Equality_N} implies
\begin{align}
 \frac{A^H_{N}}{2}  &= \E \mathcal{H}(\Trs^N) = \E(\chi_R(\|\Trs^N\|) \mathcal{H}(\Trs^N)) + \E\big((1-\chi_R(\|\Trs^N\|)) \mathcal{H}(\Trs^N)\big) = I_1 + I_2 \,.
\end{align}
Using \eqref{uboch} and \eqref{hosog} we obtain
\begin{align}
I_2 =  \E\left(\frac{\|\Trs^N\|^2}{\|\Trs^N\|^2} (1-\chi_R(\|\Trs^N\|)) \mathcal{H}(\Trs^N)\right) \lesssim \frac{1}{R^2} 
\E\left( \|\Trs^N\|^2 |\mathcal{H}(\Trs^N)|\right) \lesssim \frac{1}{R^2}  \E\left( \|\Trs^N\|^2 \mathcal{G}(\Trs^N)\right) \lesssim  \frac{1}{R^2} \,.
\end{align}
For a fixed frequency cut-off $M>0$,  the  linearity of $u \mapsto DH(u)[v]$ implies
\begin{align}
I_1  &=   \E (\chi_R(\|\Trs^N\|)  DH(\Trs^N)[\mathcal{A}(\Trs^N)]) \\
 &= \E (\chi_R(\|\Trs^N\|) DH(\Pi_M\Trs^N)[\mathcal{A}(\Trs^N)]) + 
 \E (\chi_R(\|\Trs^N\|) DH((I-\Pi_M)\Trs^N)[\mathcal{A}(\Trs^N)]) =: J_1 + J_2 \,.
\end{align}
By \eqref{hobch},  and \eqref{EstmuN} we have 
\begin{equation}
|J_2| \lesssim  \E ( |DH((I-\Pi_M)\Trs^N)[\mathcal{A}(\Trs^N)]|) \lesssim  \kappa_M \E\mathcal{G}(\Trs^N) \lesssim \kappa_M A_0  \,.
\end{equation}
Overall,  we showed
\begin{align}
\frac{A^H_{N}}{2} - C\kappa_M - \frac{C'}{R^2} \leq J_1 \leq \frac{A^H_{N}}{2} + C\kappa_M + \frac{C'}{R^2} \,.
\end{align}
Passing to the limit  first $N\to\infty$, then $R\to\infty$, and finally $M\to\infty$ we obtain
\begin{equation}
\lim_{M\to\infty} \lim_{R\to\infty} \lim_{N\to\infty} J_1=  \frac{A^H}{2}.
\end{equation}
On the other hand,  for any $\Trs \in H^{s^*_1}$ from \eqref{uboch},  \eqref{hobch},  and boundedness of $(\kappa_M)$ follows
\begin{equation}\label{dcta}
\begin{aligned}
|\chi_R(\|\Trs\|) DH(\Pi^M\Trs)[\mathcal{A}(\Trs)]| &\leq  |DH(\Pi^M\Trs)[\mathcal{A}(\Trs)]| \\
&\leq 
 |DH(\Trs)[\mathcal{A}(\Trs)]| +  |DH((I-\Pi^M)\Trs)[\mathcal{A}(\Trs)]| \lesssim \mathcal{G}(\Trs) \,.
\end{aligned}
\end{equation}
Hence,  the weak convergence $\mu_N \to \mu$,  and then dominated convergence theorem using \eqref{dcta} and \eqref{EstimateH4_mu} yield
\begin{align}
\lim_{M\to\infty} \lim_{R\to\infty} \lim_{N\to\infty} J_1&= \lim_{M\to\infty} \lim_{R\to\infty} \lim_{N\to\infty} \int \chi_R(\|v\|) DH(\Pi^Mv)[\mathcal{A}(v)] \,  \mu_N(dv) \\
&= \lim_{M\to\infty} \lim_{R\to\infty} \int \chi_R(\|v\|) DH(\Pi^Mv)[\mathcal{A}(v)] \,  \mu(dv) = 
\int DH(v)[\mathcal{A}(v)] \,  \mu(dv)\,,
\end{align}
and the assertion follows. 
%
\end{proof}

\begin{Cor}\label{LargeD}
Assume the all conditions of Theorem \ref{Thm:Equality} and in addition assume that there is $m \in \mathbb{Z}^d$ such that $K_m \neq 0$ (see \eqref{qcls} for the definition of $K_m$).
Then there is an invariant measure $\mu$ for \eqref{Euler1} enjoying the results established in Sections \ref{SE_globalflow} and \ref{ASection6InvarMeas}  and having the following property: for all $K>0$
\begin{align}
\mu(\Trs\in L^2\ :\ \mathcal{H}(\Trs)>K)>0.
\end{align}
In particular,  for any $K>0$ the  statistical ensemble contains a set of  positive $\mu$-measure such that all point in the set  are larger than $K.$
\end{Cor}

\begin{proof}
Without loss of generality assume that $K_m > 0$ for some $m$,  otherwise change the sign of $H$.
Choose the coefficients of the noise such that $A^H  = \sum_j K_j a_j^2 >  0$,  and note that such choice is possible since $K_m > 0$ for some $m$. 
Next,  for any $\lambda > 0$ we consider a new noise with coefficients $a^\lambda_m = \lambda a_m$ for each $m$.  
Then $A_{0,N}$ and $A_0$ defined in \eqref{defsize} become respectively  $A^{\lambda}_{0,N}=\lambda^2A_{0,N}$ and $A^\lambda_0=\lambda^2A_0$. 
Of course, all results,  especially Proposition \ref{pro:domm},  established above are still valid with the new noise.  Let us denote $\mu_\lambda$ the corresponding invariant measures constructed in Proposition 
\ref{pro:domm}.  Also,  Theorem \ref{Thm:Equality} implies that 
\begin{align}
\int_{L^2}\mathcal{H}(\Trs)\mu_\lambda(d\Trs) &=\lambda^2\frac{A^H}{2}.
\end{align}
By choosing $\lambda = \sqrt{n}$,  
we obtain  a sequence of invariant measures $(\mu_{n})_{n\geq 1}$ such that  
\begin{align}
\int_{L^2}\mathcal{H}(\Trs)\mu_n(d\Trs) &= n \frac{A^H}{2},
\end{align}
and in particular there is $\alpha > 0$ independent of $N$ such that 
\begin{align}
\mu_n\{\mathcal{H}(\Trs)\geq \alpha n\}>0.
\end{align}
We construct a new measure as convex combination of $(\mu_n)$ as 
\begin{align}
\mu =\sum_{n\geq 1}\frac{\mu_n}{2^n} \,.
\end{align}
Then,  $\mu$  is  invariant under the flow of \eqref{Euler1} and satisfies the claimed large data property. 
\end{proof}

\section{Infinite dimensionality}\label{sec:inf-dim}

The main result in this section provides an estimate on the dimension of the support of the invariant measure $\mu$ constructed in Proposition \ref{pro:domm}. 
We formulate and prove the theorem in more general setting, just by assuming that $\mu$ is a weak limit of 
measures $\mu_{N, \alpha}$ satisfying the moment bound \eqref{IdentityH4_muNalpha}. As such our result holds in more general setting, for example for equations on domains. 

Our main assumption concerns  the number of conservation laws, that is,  functionals  $Q : H^\sst \to \RR$ satisfying for any sufficiently large $N \geq 1$ and any 
$u \in \mathcal{E}_{N, \sst}$
\begin{equation}\label{tw4e}
DQ(u) [B_N(u, u)] = 0 \,,
\end{equation}
where $DQ(x)[v]$ denotes the derivative of $Q$ in the direction $v$. All examples  with sufficiently many admissible invariants that we were able to locate in the literature, have 
the same structure, and therefore for concreteness, we focus on those,  however the methods allow for other applications. 
For any smooth function $f : \R \to \R$, suppose
\begin{equation}\label{cnsl}
u \mapsto  \int_{\T^d} f(u(x)) dx 
\end{equation}
is a conservation law of \eqref{Euler(N)} for any sufficiently large $N$. 


Here we consider that all $a_i$ in the definition \eqref{aded} are nonzero. 

\begin{thm}
assume that for any smooth $f$ and any sufficiently large $N$,  \eqref{cnsl} is a conservation law of \eqref{Euler(N)} (that is,  it satisfies \eqref{tw4e}). 
Let $\mu$ be any invariant measure of \eqref{Euler1} constructed in Proposition \ref{pro:domm}.  
Then, for any point $w \in \textrm{supp}(\mu) \setminus \{0\}$ (if it exists) and any open neighborhood $\mathcal{N}$ of $w$, the set $\mathcal{N} \cap  \textrm{supp}(\mu)$ is not 
a subset of a finite dimensional manifold. 
\end{thm}

\begin{proof}
Fix any integer $N \geq 1$ and let $u_N$ satisfy \eqref{Euler(N)} and for the simplicity of the notation, let us drop the subscript $N$, when not needed. 
 For any integer $p \geq 1$ fix a smooth, even function $f_p : \R \to \R$ such that $f_p'(z), f_p''(z) > 0$
on $(0, \infty)$,  $f_p (z) = z^{2p} + z^2$ for any $z \in [-1, 1]$ and 
there is $C_p > 0$ such that 
$f_p(z) = C_p z + z^{-1}$ for each $z \in [2, \infty)$.  Note that $f_p$ is smooth and (almost) linear outside of a compact set, and therefore it is easy to check that $|f_p^{(k)}| \leq C_{k, p}$ for any $k \geq 1$. 
Furthermore, since $f$ is strictly convex, there is $c_p > 0$ such that $f''_p \geq c_p$ on $[0, 2]$. Due to the explicit form of $f_p$ outside of $[0, 2]$, there is a constant $C_{k, p}$ such that  for any $k \geq 3$
and any $z$  we have 
\begin{equation}
|f_p^{(k)}(z)| \lqq{k, p} f_p''(z) \,.
\end{equation}

Denote
\begin{equation}
Q_p(\uua) = \int_{\T^d} f_p(\uua(x)) \, dx \,.
\end{equation}
Then, since $|f_p'| \leq C_p$  we have for any
$\uua_0$ with $\E \|\uua_0\|^2 < \infty$ and any solution $u$ of \eqref{Euler(N,alpha)},  that 
\begin{align}\label{srmb}
\sum_{m \in \mathbb{Z}^d_N} a_m^2 \E \int_0^t |\langle f_p'(\uua), e_m \rangle|^2 ds \lqq{p, t} C \,,
\end{align}
where $\mathbb{Z}^d_N = \{m \in \mathbb{Z}^d \setminus \{0\} : |m| \leq N\}$ and importantly $C$ is independent of $N$ and $\alpha$.
Then, It\^ o fromula \cite[Corollary A.7.6]{KS12} yields
\begin{align}\label{sfeq}
 Q_p(\uua(t)) &=  Q_p(\uua(0)) + \alpha \int_0^t \langle f_p'(\uua),- \mathcal{A}_u(\uua)\rangle + \frac{1}{2} \sum_{m \in \mathbb{Z}^d_N} a_m^2  f''_p(\uua)|\langle \uua, e_m \rangle|^2 ds  \\
&\qquad + 
\sqrt{\alpha} \sum_{m \in \mathbb{Z}^d_N} \int_0^t  a_m \langle f_p'(\uua), e_m \rangle dW_m.
\end{align}
The following framework based on Krylov estimate is standard and we only outline main ideas and differences, for more details see \cite[Theorem 5.2.14]{KS12} or \cite{foldessy}. 
Fix $n \geq 1$ and denote $\mathcal{Q} = (Q_1, \cdots, Q_n)$. Then, by \eqref{sfeq},  $\mathcal{Q}$ satisfies 
\begin{equation}
 \mathcal{Q}(t) =   \mathcal{Q}(0) + \alpha \int_0^t x_s ds + \sqrt{\alpha} \sum_{m \in \mathbb{Z}^d_N} \int_0^t y_m(s) dW_m(s) 
\end{equation}
where $x_s$ is $n$-dimensional vector with 
\begin{equation}
x_s^p = \langle f_p'(\uua),- \mathcal{A}_u(\uua)\rangle + \frac{1}{2} \sum_{m \in \mathbb{Z}^d_N} a_m^2  f''_p(\uua)|\langle \uua, e_m \rangle|^2 \,,
\qquad \textrm{where } p = 1,  \cdots, n
\end{equation}
and $y_s$ is $n$-dimensional vector with 
\begin{equation}
y_s^p =  a_m \langle f_p'(\uua), e_m \rangle \,, \qquad \textrm{where } p = 1,  \cdots, n \,.
\end{equation}
Let $M$ be the random $n \times n$ matrix with coefficients 
\begin{equation}\label{mndf}
M_{i, j} = \sum_{m \in \mathbb{Z}^d_N} y_m^i y_m^j 
\end{equation}
and note that $M$ depends on $N$ and $\alpha$ which is implicitly assumed. 
Then, from the Krylov estimate (see \cite[Theorem A.9.1]{KS12}) for the stationary distribution $\mu_{N, \alpha}$ follows that 
\begin{equation}\label{kes}
\E_{\mu_{N, \alpha}} \int_0^1 (\textrm{det}\, M)^{1/n} g(\mathcal{Q}) dt \lqq{n} \|g\|_{L^n} \E_{\mu_{N, \alpha}}  \int_0^1 |x_s| ds \,, 
\end{equation}
where $g \in L^n (\RR^n)$ is arbitrary. 
The inequality \eqref{kes} is central in the proof.  A key step is to control the right hand side of \eqref{kes},  that is,  $\E_{\mu_{N, \alpha}}  \int_0^1 |x_s| ds$. We first remark that
\begin{align}
|x| \leq \sum_{p=1}^n |\langle f_p'(\uua),- \mathcal{A}(\uua)\rangle|  + \Big| \frac{1}{2} \sum_{m \in \mathbb{Z}^d_N} a_m^2  f''_p(\uua)|\langle \uua, e_m \rangle|^2\Big|.
\end{align}
From the boundedness of $f''_p$,  and \eqref{IdentityH4_muNalpha},  we have 
\begin{align}
\E_{\mu_{N,\alpha}}\sum_{m \in \mathbb{Z}^d_N} a_m^2  f''_p(\uua)|\langle \uua, e_m \rangle|^2 \leq \sup_{m \in  \mathbb{Z}^d_N}a_m^2\E_{\mu_{N,\alpha}} \sum_{m \in \mathbb{Z}^d_N}  |\langle \uua, e_m \rangle|^2
= \sup_{m \in  \mathbb{Z}^d_N}a_m^2\E_{\mu_{N,\alpha}}  \|\uua\|^2 \leq C,
\end{align}
where $C$ is independent of $\alpha$ and $N$.
Therefore,  we are left to bound the term
\begin{align}
\E_{\mu_{N,\alpha}}\left|\langle f_p'(\uua),- \mathcal{A}_u(\uua)\rangle\right|.
\end{align}
More explicitly, it suffices then to bound the following three terms:
\begin{gather}
J_1 := \E_{\mu_{N,\alpha}} e^{\rho(\|\uua\|_{H^\sst}^2)}\left|\langle\Delta^5\uua,f_p'(\uua)\rangle\right|, \qquad 
J_2 := \E_{\mu_{N,\alpha}} e^{\rho(\|\uua\|_{H^\sst}^2)}\left|\langle\Delta(|\Delta\uua|^{q-2})\Delta\uua),f_p'(\uua)\rangle\right|,  \\
J_3 := \E_{\mu_{N,\alpha}} e^{\rho(\|\uua\|_{H^\sst}^2)}\left|\langle\nabla(|\nabla\uua|^{2q-2})\nabla\uua),f_p'(\uua)\rangle\right| \,.
\end{gather}
From the boundedness of $f''$ follows 
\begin{align}
  |\left\langle |\nabla \theta|^{2q-2} \nabla \theta, f_p'' (\theta) \nabla \uua \rangle\right| \lesssim\|\theta\|_{W^{1, 2q}}^{2q} \,,
\end{align}
and therefore 
\begin{equation}
J_3 \lesssim \E_{\mu_{N,\alpha}} e^{\rho(\|\uua\|_{H^\sst}^2)} \|\theta\|_{W^{1, 2q}}^{2q}
\end{equation}
To estimate $J_2$ we use the boundedness of $f_p''$ and $f_p'''$,  H\" older's inequality and Young's inequality to obtain 
\begin{multline}
 | \langle\Delta \left(|\Delta \uua|^{q-2} \Delta \uua\right),  f_p^{\prime}(\uua)\rangle| =
 |\langle|\Delta \uua|^{q-2} \Delta \uua, f_p''(\uua)\Delta\uua\rangle+\langle|\Delta \uua|^{q-2} \Delta \uua,f_p'''(\uua)|\nabla \uua|^2\rangle| \\
\leq \| \Delta\uua\|_{L^q}^{q} +\langle|\Delta \uua|^{q-1},   |\nabla \uua|^2\rangle \leq \| \Delta\uua\|_{L^q}^{q}  + \| \Delta\uua\|_{L^q}^{q-1}\|\nabla \uua\|_{L^{2q}}^{2}
\lesssim \|\uua\|_{W^{2, q}}^{q}+\|\uua\|_{W^{1, 2q}}^{2q} \,,
\end{multline}
and therefore 
\begin{equation}
J_2 \lesssim  \E_{\mu_{N,\alpha}} e^{\rho(\|\uua\|_{H^\sst}^2)} ( \|\uua\|_{W^{2, q}}^{q}+\|\uua\|_{W^{1, 2q}}^{2q})
\end{equation}
To estimate $J_1$ we denote $D^m \theta$ any $m$th derivative of $\theta$ with respect to $m$, possibly repeating spatial variables $(x_j)$.  Since there are only finitely many permutations 
of spatial variables, the integration by parts,  triangle inequality,  and boundedness of $f^{(k+1)}$ (the $k$th derivative) for $k \geq 1$ yield
\begin{equation}\label{joexp}
\begin{aligned}
J_1 &\lesssim \E_{\mu_{N,\alpha}} e^{\rho(\|\uua\|_{H^\sst}^2)}\sum_{k = 1}^5 \sum_{\substack{m_1 + \cdots + m_k = 5\\ m_j \geq 1 }}  \langle |f^{(k+1)}(\uua)|| D^{5}\uua|,  |D^{m_1} \uua| \cdots |D^{m_k} \uua| \rangle 
\\
&\lesssim  \E_{\mu_{N,\alpha}} e^{\rho(\|\uua\|_{H^\sst}^2)} \|\uua\|_{W^{5, 2}}  \sum_{k = 1}^5 \sum_{\substack{m_1 + \cdots + m_k = 5\\ m_j \geq 1 }} \| |D^{m_1} \uua| \cdots |D^{m_k} \uua|\|
=:  \E_{\mu_{N,\alpha}} e^{\rho(\|\uua\|_{H^\sst}^2)} \sum_{k = 1}^5 I_k \,.
\end{aligned}
\end{equation}
Let us estimate each $I_k$ separately.  
First, we have 
$$
I_1 =  \|\uua\|_{W^{5, 2}} \| D^5 \uua \| \lesssim \|\uua\|_{W^{5, 2}}^2 \,.
$$
Then,  by H\" older's,  Gagliardo-Nirenberg's,  and Young's inequalities  we obtain
\begin{align}
I_2 &\lesssim\|\uua\|_{W^{5, 2}} ( \| |D^1 \uua||D^4 \uua| \| + \| |D^2 \uua||D^3 \uua| \|)  \lesssim 
\|\uua\|_{W^{5, 2}}(\|D^1 \uua\|_{L^{2q}}  \|D^4 \uua\|_{L^{(2q)^*}} + 
 \|D^2 \uua\|_{L^{q}}  \|D^3 \uua\|_{L^{q^*}})   \\
 &\lesssim 
 \|\uua\|_{W^{5, 2}}(\|D^1 \uua\|_{L^{2q}} \|D^1 \uua\|_{L^{2q}}^{1 - \eta_1}  \|\uua\|_{W^{5, 2}}^{\eta_1} + 
 \|D^2 \uua\|_{L^{q}}  \|D^2 \uua\|_{L^{q}}^{1 - \eta_2}  \|\uua\|_{W^{5,2}}^{\eta_2}) \\
 &= 
\|\uua\|_{W^{1,2q}}^{2 - \eta_1}  \|\uua\|_{W^{5, 2}}^{\eta_1 + 1} + 
  \|D^2 \uua\|_{W^{2,q}}^{2 - \eta_2}  \|\uua\|_{W^{5,2}}^{1 + \eta_2} 
  \lesssim
 \|\uua\|_{W^{5, 2}}^2 +   \|\uua\|_{W^{1,2q}}^{\frac{2(2 - \eta_1)}{1-\eta_1}} 
 +   \|\uua\|_{W^{2,q}}^{\frac{2(2 - \eta_2)}{1-\eta_2}} \,,
\end{align}
where for any $r \in (2, \infty)$ we denote $r^*$ a number that satisfies $\frac{1}{r} + \frac{1}{r^*} = \frac{1}{2}$ and 
\begin{equation}
\eta_1 = \frac{5q + 6}{7q + 3}, \qquad \eta_2 = \frac{3q + 6}{7q + 6} \,.
\end{equation}
Then if $q \geq 6$ it follows that 
$$
\frac{2(2 - \eta_1)}{1-\eta_1} = 2q \frac{9}{2q - 6} \leq 2q, \qquad 
\frac{2(2 - \eta_2)}{1-\eta_2} = \frac{11q + 6}{2q} \leq q \,,
$$
and consequently 
\begin{equation}
I_2 \lesssim  1 +  \|\uua\|_{W^{5, 2}}^2 +   \|\uua\|_{W^{1,2q}}^{2q} +   \|\uua\|_{W^{2,q}}^{q} \,.
\end{equation}
Similarly,  by using H\" older's,  Gagliardo-Nirenberg's,  and Young's inequalities  we have for $q > 4$
\begin{align}
I_3 &\lesssim\|\uua\|_{W^{5, 2}} ( \| |D^1 \uua|^2|D^3 \uua| \| + \| |D^2 \uua|^2|D^1 \uua| \|)  \lesssim 
\|\uua\|_{W^{5, 2}}(\|D^1 \uua\|_{L^{2q}}^2  \|D^3 \uua\|_{L^{q^*}} + 
 \|D^2 \uua\|_{L^{q}}^2  \|D^1 \uua\|_{L^{(q/2)^*}})   \\
 &\lesssim 
 \|\uua\|_{W^{5, 2}}(\|D^1 \uua\|_{L^{2q}}^2 \|D^1 \uua\|_{L^{2q}}^{1 - \eta_3}  \|\uua\|_{W^{5, 2}}^{\eta_3} + 
 \|D^2 \uua\|_{L^{q}}^2  \|D^1 \uua\|_{L^{2q}}) \\
 &= 
\|\uua\|_{W^{1,2q}}^{3 - \eta_3}  \|\uua\|_{W^{5, 2}}^{\eta_3 + 1} + 
 \|\uua\|_{W^{5, 2}} \|D^2 \uua\|_{W^{2,q}}^{2}  \|\uua\|_{W^{1,2q}}
  \lesssim
 \|\uua\|_{W^{5, 2}}^2 +   \|\uua\|_{W^{1,2q}}^{\frac{2(3 - \eta_3)}{1-\eta_3}} 
 +  \|\uua\|_{W^{1,2q}}^6
 +   \|\uua\|_{W^{2,q}}^6 \,,
\end{align}
where 
\begin{equation}
\eta_3 = \frac{3q + 9}{7q + 3} \,.
\end{equation}
Then if $q \geq 6$ it follows that 
$$
\frac{2(3 - \eta_3)}{1-\eta_3} =  \frac{18q}{2q - 3} \leq 2q \,,
$$
and consequently 
\begin{equation}
I_3 \lesssim  1 + \|\uua\|_{W^{5, 2}}^2 +   \|\uua\|_{W^{1,2q}}^{2q} +   \|\uua\|_{W^{2,q}}^{q} \,.
\end{equation}
Also,  by using H\" older's,  Gagliardo-Nirenberg's,  and Young's inequalities  we have 
\begin{align}
I_4 &\lesssim\|\uua\|_{W^{5, 2}}  \| |D^1 \uua|^3|D^2 \uua| \| \lesssim 
\|\uua\|_{W^{5, 2}} \|D^1 \uua\|_{L^{2q}}^3  \|D^2 \uua\|_{L^{(2q/3)^*}}    \\
 &\lesssim 
 \|\uua\|_{W^{5, 2}}\|D^1 \uua\|_{L^{2q}}^3 \|D^1 \uua\|_{L^{2q}}^{1 - \eta_4}  \|\uua\|_{W^{5, 2}}^{\eta_4}  \\
 &= 
\|\uua\|_{W^{1,2q}}^{4 - \eta_4}  \|\uua\|_{W^{5, 2}}^{\eta_4 + 1} 
  \lesssim
 \|\uua\|_{W^{5, 2}}^2 +   \|\uua\|_{W^{1,2q}}^{\frac{2(4 - \eta_4)}{1-\eta_4}} \,,
\end{align}
where 
\begin{equation}
\eta_4 = \frac{6q}{7q + 3} \,.
\end{equation}
Then if $q \geq 6$ it follows that 
$$
\frac{2(4 - \eta_4)}{1-\eta_4} =  \frac{9q + 3}{q} \leq 2q \,,
$$
and consequently 
\begin{equation}
I_4 \lesssim 1+  \|\uua\|_{W^{5, 2}}^2 +   \|\uua\|_{W^{1,2q}}^{2q}  \,.
\end{equation}
Finally,  if $q \geq 6$
\begin{equation}
I_5 \lesssim\|\uua\|_{W^{5, 2}}  \| |D^1 \uua\|^5 \lesssim 
\|\uua\|_{W^{5, 2}}^2 +  \|D^1 \uua\|_{L^{2q}}^{10}  \lesssim 
1 + \|\uua\|_{W^{5, 2}}^2 +  \|\uua\|_{W^{1, 2q}}^{2q} 
\,.
\end{equation}
Overall,  by after substitution into \eqref{joexp} we conclude that
\begin{equation}
J_1 \lesssim  \E_{\mu_{N,\alpha}} e^{\rho(\|\uua\|_{H^\sst}^2)} (1 + \|\uua\|_{W^{5, 2}}^2 +   \|\uua\|_{W^{1,2q}}^{2q} +   \|\uua\|_{W^{2,q}}^{q}) 
 \,,
\end{equation}
and then  \eqref{MainEstaux}, and \eqref{IdentityH4_muNalpha} yield
\begin{align}
\E_{\mu_{N,\alpha}}\left|\langle f_p'(\uua),- \mathcal{A}_u(\uua)\rangle\right| &\leq J_1 + J_2 + J_2 \lesssim
 \E_{\mu_{N,\alpha}} e^{\rho(\|\uua\|_{H^\sst}^2)} (1 + \|\uua\|_{W^{5, 2}}^2 +   \|\uua\|_{W^{1,2q}}^{2q} +   \|\uua\|_{W^{2,q}}^{q}) 
 \\
&\lesssim \E_{\mu_{N,\alpha}} e^{\rho(\|\uua\|_{H^\sst}^2)} + \E_{\mu_{N,\alpha}} \mathcal{G}(\uua) \leq C \,,
\end{align}
where $C$ is independent of $N$ and $\alpha$.

Thus, \eqref{kes} with explicit dependencies on $\alpha$ and $N$ implies 
\begin{equation}
\E_{\mu_{N, \alpha}} \int_0^1 (\textrm{det}\, M^N(y_t^{\alpha, N}(u_t^{\alpha, N})))^{1/n} g(\mathcal{Q}(u_t^{\alpha, N})) dt \lqq{n, \sst, \sstt, \ell} \|g\|_{L^n} \,,
\end{equation}
where $u_0^{\alpha, N}$ is distributed as an invariant measure $\mu_{N, \alpha}$. Since $u_t^{N, \alpha}$ is stationary, 
\begin{equation}
\E_{\mu_{N, \alpha}} [\textrm{det}\, M^N(y_0^{\alpha, N}(u_0^{\alpha, N})))^{1/n} g(\mathcal{Q}(u_0^{\alpha, N})] \lqq{n, \sst, \sstt, \ell} \|g\|_{L^n} \,,
\end{equation}
and by the triangle inequality
\begin{multline}\label{deteq}
\E_{\mu_{N, \alpha}} [\textrm{det}\, M^\infty (y_0^{\alpha, N}(u_0^{\alpha, N})))^{\frac{1}{n}} g(\mathcal{Q}(u_0^{\alpha, N})]
\\ - \E_{\mu_{N, \alpha}} [\left|\textrm{det}\, M^\infty - \textrm{det}\, M^N\right| (y_0^{\alpha, N}(u_0^{\alpha, N})))^{\frac{1}{n}} |g(\mathcal{Q}(u_0^{\alpha, N})|]
 \lqq{n, \sst, \sstt, \ell} \|g\|_{L^n} \,,
\end{multline}
where $M^\infty$ is defined as $M^N$ in \eqref{mndf} with the exception that the summation is taken over $\mathbb{Z}^d$ instead of $\mathbb{Z}^d_N$.
Using that $|f'_p| \lqq{p} C$,  we obtain $|y_m^i| \leq |a_m|$,  and consequently $\|M^\infty - M^N\|_{1} \lqq{n} \sum_{j \geq N} a_j^2$, where $\| \cdot \|_{1}$ denotes the matrix norm corresponding to
$\ell^1$ norm on the vectors (all norms are equivalent in finite dimensions). 
Determinant is a
continuous function of the coefficients of the matrix, and since $\sum_{j \geq N} a_j^2 \to 0$ as $N \to \infty$,  we obtain $\left|\textrm{det}\, M^\infty - \textrm{det}\, M^N\right| \to 0$ as $N \to \infty$ uniformly in $\alpha$. Thus for any bounded $g$ 
we obtain that the second term on the left hand side of  \eqref{deteq} converges to zero as $N \to \infty$, uniformly in $\alpha > 0$. 

Recall, that $\mu_{N, \alpha} \to \mu$ where we first pass $\alpha \to 0$ and then $N \to \infty$, where $\mu$
is as in Proposition \ref{pro:domm}.  Then, 
by Skorokhod representation theorem, we can assume 
that  $u_0^{\alpha, N} \to u_0^{0, \infty}$ almost surely, where $u_0^{0, \infty}$ has distribution $\mu$. Furthermore, 
\begin{align}
M_{i, j}^\infty (y_0^{\alpha, N}(u_0^{\alpha, N})) &= \sum_{m \in \mathbb{Z}^d} a_m^2 \langle f'_i(u_0^{\alpha, N}), e_m\rangle  \langle f'_j(u_0^{\alpha, N}), e_m\rangle
\to  \sum_{m \in \mathbb{Z}^d} a_m^2 \langle f'_i(u_0^{0, \infty}), e_m\rangle  \langle f'_j(u_0^{0, \infty}), e_m\rangle \\
&= 
M_{i, j}^\infty (y_0^{\alpha, N}(u_0^{0, \infty})) \,, 
\end{align}
where in the passage to the almost sure limit $\alpha \to 0$ and  then $N \to \infty$ we used that $f'_p$ is bounded, continuous function and $ \sum_{m \in \mathbb{Z}^d} a_m^2 < \infty$. Thus, 
by the Lebesgue dominated convergence theorem (determinant is a continuous function of the matrix and $f'_p$ is bounded), we obtain for any bounded $g$ 
%
%
\begin{equation}
\E_{\mu} \int_0^1 (\textrm{det}\, M)^{1/n} g(\mathcal{Q}) dt \lqq{n, \sst, \sstt} \|g\|_{L^n} \,,
\end{equation}
where $M = M^\infty$.

For each $\epsilon > 0$ denote $I_\epsilon = \{\uua : \|\uua\| \geq \epsilon, \|\uua\|_{H^{d/2}} \leq \frac{1}{\epsilon}\}$. Following \cite[proof of Theorem 5.2.14]{KS12} or \cite[proof of Theorem 6.2]{foldessy}
we claim that on $I_\epsilon$ there holds 
\begin{equation}
\textrm{det}\, M \geq c_\epsilon > 0  \,.
\end{equation}
The matrix $M$ (depending on $\uua$) is symmetric and non-negative and 
we show that $\det M > 0$ whenever $\uua \neq 0$.  It suffices to show that $v^T M v > 0$ for any vector $v \neq 0$, which is equivalent to 
\begin{equation}
0 < v^T M v = \sum_{m \in \mathbb{Z}^d \setminus \{0\}} (v^T y_m(\uua))^2 =   \sum_{m \in \mathbb{Z}^d \setminus \{0\}} a_m^2 |\langle v^T f'(\uua), e_m  \rangle|^2 \,,
\end{equation}  
where $f' = (f_1', \cdots, f_n')$. Since $a_m \neq 0$ for each $m \in \mathbb{Z}^d \setminus \{0\}$,  the desired result follows once $ v^Tf (\uua) \neq 0$ on $\T^d$. However, 
$\uua$ has zero mean, and unless $\uua = 0$, then the range of $\uua$ contains an interval $(- \delta, \delta)$ for some $\delta > 0$. Then, 
$0 = v^Tf (\xi) = \sum_{j = 1}^n v_j (\xi +  j\xi^{2j-1})$ for any $\xi \in (-\delta, \delta)$,  and since the right hand side is a polynonial in $\xi$,  it follows that $v_j = 0$ for each $j$.  

The rest of the proof follows standard arguments that can be found in \citep{KS12} or \cite{foldessy}. 
\end{proof}

 
%
%
%
%
%
%

\bibliographystyle{alpha}
\bibliography{euler}
\end{document}